\documentclass[11pt]{article}
\usepackage{amssymb}
\usepackage{amsfonts}
\usepackage{amsmath}
\usepackage{mathrsfs}
\usepackage{graphicx}
\usepackage{amsbsy}
\usepackage{theorem}
\usepackage{color}
\usepackage{hyperref}%
\usepackage{tikz}
\usepackage{epstopdf}
\usepackage[normalem]{ulem}
\newtheorem{theorem}{Theorem}[section]
\newtheorem{lemma}[theorem]{Lemma}

\newtheorem{proposition}[theorem]{Proposition}
\newtheorem{remark}[theorem]{Remark}
\newtheorem{example}[theorem]{Example}

\def\thetheorem{\thesection.\arabic{theorem}}
\def\thesection{\arabic{section}}

\def\theequation {\thesection.\arabic{equation}}
\def\beq{\begin{equation}\displaystyle}
\def\eeq{\end{equation}}
\def\bel{\begin{equation} \displaystyle \begin{array}{l} }
\def\eel{\end{array} \end{equation} }
\def\bell{\begin{equation} \displaystyle \begin{array}{ll}  }
\def\eell{\end{array} \end{equation} }

\def\bea{\begin{eqnarray}}
\def\eea{\end{eqnarray} }
\def\bean{\begin{eqnarray*}}
\def\eean{\end{eqnarray*} }
\def\m1n{M_{1,\Delta x}}
\newenvironment{proof}{\noindent{\bf Proof.~}}
{{\mbox{}\hfill {\small \fbox{}}\\}}
\catcode`@=11
\renewcommand\appendix{\bigskip {\noindent \Large \bf Appendix}
  \setcounter{section}{0}%
  \setcounter{subsection}{0}%
\setcounter{equation}{0}%
\setcounter{theorem}{0}%
\def\thetheorem{A.\arabic{theorem}}
\def\theequation {A.\arabic{equation}}}
\catcode`@=12

\newcommand{\dv}{\mathop{\rm div}\nolimits}

\def\NN{\mathbb{N}}

\def\RR{\mathbb{R}}

\def\ZZ{\mathbb{Z}}

\def\ds{\displaystyle}
\def\bs{\bigskip}

\def\eps{\varepsilon}

\def\pa{\partial}

\def\calM{{\mathcal M}}

\def\calP{{\mathcal P}}

\def\calT{{\mathcal T}}
\def\calV{{\mathcal V}}

\def\smes{{\cal S}_{\cal M}}

\def\achapo{\widehat{a}}

\def\nabWchapo{\widehat{\nabla W}}

\begin{document}

\title{Convergence analysis of upwind type schemes for the aggregation equation with pointy potential}

\author{F. Delarue\thanks{Laboratoire J.-A. Dieudonn\'e, 
UMR CNRS 7351, 
Univ. Nice, Parc Valrose, 06108 Nice Cedex 02, France. Email: \texttt{delarue@unice.fr}}, 
F. Lagouti\`ere\thanks{Univ Lyon,  Universit\'e Claude Bernard Lyon 1,  CNRS UMR 5208,  Institut Camille Jordan,  43 blvd. du 11 novembre 1918, F-69622 Villeurbanne cedex, France, Email: \texttt{lagoutiere@math.univ-lyon1.fr}},
N. Vauchelet\thanks{Universit\'e Paris 13, Sorbonne Paris Cit\'e, CNRS UMR 7539, Laboratoire Analyse G\'eom\'etrie et Applications, 93430 Villetaneuse, France, Email: \texttt{vauchelet@math.univ-paris13.fr}}
}

\maketitle

\begin{abstract}
A numerical analysis of upwind type schemes for the nonlinear
nonlocal aggregation equation is provided.
In this approach, the aggregation equation is interpreted as a conservative transport 
equation driven by a nonlocal nonlinear velocity field with low regularity. 
In particular, we allow the interacting potential to be pointy, 
in which case the velocity field may have discontinuities.
Based on recent results of existence and uniqueness of a Filippov flow for this type of equations, we study
an upwind finite volume numerical scheme 
and we prove that it is convergent at order
$1/2$ in Wasserstein distance. The paper is illustrated by numerical simulations that indicate that this convergence order should be optimal. 
\end{abstract}

\bs



\bs

\section{Introduction}

This paper is devoted to the numerical approximation of measure valued solutions to the so-called aggregation equation in space dimension $d$. This equation reads
\beq\label{EqInter}
\pa_t\rho = \dv\big((\nabla_x W*\rho) \rho\big) , \qquad t>0,\quad x\in\RR^d,
\eeq
with the initial condition $\rho(0,\cdot)=\rho^{ini}$. Here, $W$ plays the role of an interaction potential whose gradient $\nabla_x W(x-y)$ measures the relative force exerted by a unit mass localized at a point $y$ onto a unit mass located at a point $x$.

This system appears in many applications in physics and population dynamics.
In the framework of granular media, equation \eqref{EqInter} is used to
describe the large time dynamics of inhomogeneous kinetic models, see \cite{benedetto,CCV,Toscani}.
Models of crowd motion with a nonlinear term of the form $\nabla_xW*\rho$
are also addressed in \cite{pieton,pieton2}.
In population dynamics, \eqref{EqInter} provides a biologically meaningful description of aggregative phenomena. 
For instance, the description of the collective migration of cells by swarming
leads to such a kind of PDEs with non-local interaction, see e.g. \cite{morale,okubo,topaz}.
Another example is the modelling of bacterial chemotaxis. In this framework, the quantity $S=W*\rho$ is the chemoattractant concentration, which is a substance emitted by bacteria allowing them to interact with one another. The dynamics can be macroscopically modelled by the Patlak-Keller-Segel system \cite{keller,patlack}. In the kinetic framework, 
the most frequently used model is the Othmer-Dunbar-Alt system,  the hydrodynamic limit of which leads to the aggregation equation \eqref{EqInter}, see \cite{dolschmeis,filblaurpert,NoDEA}. In many of these examples, the potential $W$ is usually mildly singular, i.e. $W$ has a weak singularity at the origin. Because of this low regularity, smooth solutions of such systems may blow-up in finite time, see e.g. \cite{Li,BV,Bertozzi2,Carrillo}. In the latter case, finite time concentration may be regarded as a very simple mathematical way to account for aggregation of individuals, as opposed to diffusion. 

Since finite time blow-up of smooth solutions may occur and since equation \eqref{EqInter} conserves mass, a natural framework to
study the existence of global in time solutions is to work in the space of probability measures. In this regard, two strategies have been proposed in the literature. In \cite{Carrillo}, the aggregation equation is seen as a gradient flow taking values in the Wasserstein space and minimizing the interaction energy. In \cite{NoDEA,GF_dual,CJLV, lava}, this system is considered as a conservative transport equation with velocity field $\nabla_x W*\rho$. Then a unique flow, say $Z=(Z(t,\cdot))_{t \geq 0}$, can be constructed, hence allowing to define the solution as a pushforward measure by the flow, namely $\rho=(\rho(t)=Z(t,\cdot)_\# \rho^{ini})_{t \ge 0}$. When the singularity of the potential is stronger than the mild form described above, 
such a construction 
has been achieved 
in the radially symmetric case  in \cite{Andrea_c_toaa}, but uniqueness is
then lacking. Actually, 
the assumptions on the potential $W$ that are needed to ensure the well-posedness of the equation in the space of measure valued 
solutions require a certain convexity property of the potential
that allows only for a mild singularity at the origin.
More precisely, we assume that the interaction potential $W\,:\,\RR^d\to\RR$ satisfies the following properties:
\begin{itemize}
\item[{\bf (A0)}] $W(x)=W(-x)$ and $W(0)=0$;
\item[{\bf (A1)}] $W$ is $\lambda$-convex for some 
{$\lambda  \in \RR$}, i.e.
  $W(x)-\frac{\lambda}{2}|x|^2$ is convex;
\item[{\bf (A2)}] $W\in C^1(\RR^d\setminus\{0\})$;
\item[{\bf (A3)}] $W$ is Lipschitz-continuous.
\end{itemize}
Such a potential will be referred to as a {\it pointy} potential. 
Typical examples are fully attractive potentials $W(x)=1-e^{-|x|}$, which is $-1$-convex, and $W(x) = |x|$, which is $0$-convex.
Notice that the Lipschitz-continuity of the potential allows to bound the velocity field:
there exists a nonnegative constant $w_\infty$ such that for all $x\neq 0$,
\begin{equation}
\label{borngradW}
|\nabla W(x)| \leq w_\infty.
\end{equation}
Observe also that {\bf (A3)} forces $\lambda$ in \textbf{\bf (A1)} to be non-positive, as otherwise $W$ would be at least of quadratic growth, whilst {\bf (A3)} forces it to be at most of linear growth. 
However, we shall sometimes discard \textbf{(A3)}, when the initial datum is compactly supported. In this case, as $W - \lambda |x|^2/2$ is convex, it is locally Lipschitz-continuous, so that $W$ is locally Lipschitz-continuous, what will be sufficient for compactly supported initial data. 
In that case it perfectly makes sense to assume $\lambda >0$ in \textbf{\bf (A1)}. For numerical analysis, we will assume in this case that the potential is radial, that is to say that $W$ is a function of the sole scalar $|x|$, $W(x) = \mathcal{W}(|x|)$. 

Although very accurate numerical schemes have been developped to study the blow-up profile for smooth solutions, see \cite{Huang1,Huang2},
very few numerical schemes have been proposed to simulate the behavior of solutions to the aggregation equation after blow-up. 
The so-called sticky particles method was shown to be convergent in \cite{Carrillo} 
and used to obtain qualitative properties of the solutions such as the time of  total collapse. However, this method is not so practical to catch the behavior of the solutions after blow-up in dimension $d$ larger than one. In dimension $d=1$, this question has been addressed in \cite{NoDEA}. In higher dimension, particle methods have been recently proposed and studied in \cite{CB,Freda}, 
but only the convergence of smooth solutions, before the blowup time, has been proved. Finite volume schemes have also been developed. In \cite{sinum}, the authors propose a finite volume scheme to approximate the behavior of the solution to the aggregation equation \eqref{EqInter} after blow-up and prove that it is convergent. A finite volume method for a large class of PDEs including in particular \eqref{EqInter} has been also proposed in \cite{CCH}, but no convergence result has been given. Finally, a finite volume scheme of Lax-Friedrichs type for general measures as initial data has been introduced and investigated in \cite{CJLV}.
Numerical simulations of solutions in dimension greater than one have been obtained, allowing to observe the behavior after blow-up. Moreover, convergence towards measure valued solutions has been proved. However, no estimate on the order of convergence has been established so far. In the current work, we provide a precise estimate of the order of convergence in Wasserstein distance for an upwind type scheme. This scheme is based on an idea introduced in \cite{NoDEA} and used later on in \cite{sinum,CJLV}. It consists in discretizing properly the macroscopic velocity so that its product with the measure solution $\rho$ is well-defined. In this paper, we introduce an upwind scheme for which this product is treated accurately, and we prove its convergence at order $1/2$ in Wasserstein distance (the definition of which is recalled below).

For a given velocity field, the study of the order of convergence for the finite volume
upwind scheme for the transport equation has received a lot of attention. This scheme is known to be first order convergent in $L^\infty$ norm for any smooth initial data in $C^2(\RR^d)$ and for well-suited meshes, provided a standard stability condition (Courant-Friedrichs-Lewy condition) holds, see \cite{bouche}. However, this order of convergence falls down to $1/2$ in $L^p$ norm when considering non-smooth initial data or more general meshes. This result has been first proved in the Cartesian framework by Kuznetsov in \cite{Kuznetsov}. In \cite{Despres}, a $1/2$ order estimate in the $L^\infty([0,T],L^2(\RR^d))$ norm for $H^2(\RR^d)$ initial data has been established.  Finally in \cite{MV,caniveau}, a $1/2$ order estimate in $L^1$ has been proved for initial data in $L^1(\RR^d)\cap BV(\RR^d)$, whilst, for Lipschitz-continuous initial data, an estimate of order $1/2-\eps$ in $L^\infty$ for any $\eps>0$ has been obtained in \cite{M,caniveau}. We emphasize that the techniques used in \cite{M,MV} and \cite{caniveau} are totally different.
In the former, the strategy of proof is based on entropy estimates, 
whereas in the latter, the proof
relies on the construction and the analysis of stochastic characteristics for the numerical scheme.
Finally, when the velocity field is only $L^\infty$ and one-sided Lipschtiz-continuous, solutions of the conservative transport equation are defined only in 
the sense of measures. In this regard, Poupaud and Rascle \cite{PoupaudRascle} have proved that
solutions of the conservative transport equation could be defined as the pushforward of the initial condition 
by a flow of characteristics. A stability estimate for such solutions has been stated later in \cite{Bianchini}.
In dimension $d=1$, these solutions, as introduced in 
\cite{PoupaudRascle},
 are equivalent to duality solutions, as defined 
in \cite{bj1}. Numerical investigations may be found in \cite{GJ}.
In such a framework with a low regularity, numerical analysis requires to 
work with a sufficiently weak topology, which is precisely what 
has been done in \cite{DLV}. Therein, the convergence at order $1/2$ of a finite volume upwind scheme 
has been shown 
in Wasserstein distance by means of a
stochastic characteristic method, as done in 
\cite{caniveau}.
Observe also that, recently, 
such an approach has been successfully used in \cite{schlichting} for 
the numerical analysis of the upwind scheme 
for the transport equation with rough coefficients.
In the current work, we adapt the strategy initiated in \cite{DLV} to prove the convergence at order $1/2$ of an upwind scheme for the aggregation equation for which the velocity field depends on the solution in a nonlinear way.
We will strongly use the fact that, as mentioned above, measure valued solutions of \eqref{EqInter} are constructed by pushing forward the initial condition by an $\RR^d$-valued flow.
Noticeably, we entirely reformulate the stochastic approach used in \cite{DLV} by means of analytical tools. In the end, our proof is completely deterministic. Although using analytical instead of probabilistic arguments do not change the final result (neither nor the general philosophy of the proof), it certainly makes the whole more accessible for the reader. As we pointed out, the 
key fact in \cite{DLV} is to represent the scheme through a Markov chain; here, 
the main idea is to use the sole transition kernel of the latter Markov chain to couple 
the measure-valued numerical solution at two consecutive times (and hence to bypass any use of the Markov chain itself). We refer to Remark \ref{comparison} below for more details.

The outline of the paper is the following. 
In the next section, we introduce the notations and recall the theory for the existence 
of a measure solution to \eqref{EqInter}. Then we present the upwind scheme and state the main result: the scheme is convergent at order $1/2$. In case when the potential $W$ is strictly convex and radially symmetric and the initial condition has a bounded support, the rate is claimed to be uniform in time. 
Section \ref{sec:num} is devoted to the properties of the scheme. 
The proof of the main result for a Cartesian grid mesh 
is presented in Section \ref{sec:ordre}.
In Section \ref{sec:unstruct}, we explain briefly how to extend our result to simplicial meshes. Finally, numerical illustrations are given in Section \ref{sec:sim}.
In particular, we show that the order of convergence is optimal and we provide several 
numerical simulations in which we recover the behavior of the solutions after blow-up time.

\section{Notations and main results}

\subsection{Notations}

Throughout the paper, we will make use of the following notations.
We denote by $C_0(\RR^d)$ the space of continuous functions from $\RR^d$ to $\RR$ that tend to $0$ at $\infty$.
We denote by $\calM_b(\RR^d)$ the space of Borel signed measures whose total variation is finite.
For $\rho\in {\cal M}_{b}(\RR^d)$, we call $|\rho|(\RR^d)$ its total variation.
The space $\calM_b(\RR^d)$ is equipped with the weak topology $\sigma({\cal M}_b(\RR^d),C_0(\RR^d))$.
For $T>0$, we let $\smes :=C([0,T];{\cal M}_b(\RR^d)-\sigma({\cal M}_b(\RR^d),C_0(\RR^d)))$.
For $\rho$ a measure in $\calM_b(\RR^d)$ and $Z$ a measurable map, 
we denote $Z_\#\rho$ the pushforward measure of $\rho$ by $Z$; 
it satisfies, for any continuous function $\phi$,
$$
\int_{\RR^d} \phi(x)\, Z_\#\rho(dx) = \int_{\RR^d} \phi(Z(x))\,\rho(dx).
$$
We call $\calP(\RR^d)$ the subset of $\calM_b(\RR^d)$ of probability measures.
We define the space of probability measures with finite
second order moment by
$$
\calP_2(\RR^d) := \left\{\mu \in \calP(\RR^d),\  \int_{\RR^d} |x|^2 \mu(dx) <\infty\right\}.
$$
Here and in the following, $|\cdot|^2$ stands for the square Euclidean norm, and $\langle\cdot,\cdot\rangle$ for the Euclidean inner product.
The space ${\mathcal P}_{2}(\RR^d)$ is equipped with the Wasserstein distance $d_W$ defined by (see e.g. \cite{Ambrosio,Villani1, Villani2, Filippo c touo})
\beq\label{defWp}
d_W(\mu,\nu) := \inf_{\gamma\in \Gamma(\mu,\nu)} \left\{\int_{\RR^d\times \RR^d} |y-x|^2\,\gamma(dx,dy)\right\}^{1/2}
\eeq
where $\Gamma(\mu,\nu)$ is the set of measures on $\RR^d\times\RR^d$ with marginals $\mu$ and $\nu$, i.e.
\begin{align*}
\Gamma(\mu,\nu) = \left\{ \gamma\in \calP_2(\RR^d\times\RR^d); \ \forall\, \xi\in C_0(\RR^d), \right. &
\int \xi(y_1)\gamma(dy_1,dy_2) = \int \xi(y_1) \mu(dy_1), \\
& \left.\int \xi(y_2)\gamma(dy_1,dy_2) = \int \xi(y_2) \nu(dy_2) \right\}.
\end{align*}
By a minimization argument, we know that 
the infimum in the definition of $d_{W}$ is actually a minimum.
A measure that realizes the minimum in the definition \eqref{defWp}
of $d_W$ is called an {\it optimal plan}, the set of which is denoted by $\Gamma_0(\mu,\nu)$.
Then, for all $\gamma_0\in \Gamma_0(\mu,\nu)$, we have
$$
d_W(\mu,\nu)^2= \int_{\RR^d\times \RR^d} |y-x|^2\,\gamma_0(dx,dy).
$$

We will make use of the following properties of the Wasserstein distance.
Given $\mu \in {\mathcal P}_{2}(\RR^d)$ and two $\mu$-square integrable Borel measurable maps $X,Y:\RR^d\to \RR^d$, we have the inequality
\begin{equation*}
d_W(X_\#\mu,Y_\#\mu) \leq \|X-Y\|_{L^2(\mu)}.
\end{equation*}
It holds because $\pi=(X,Y)_\# \mu\in \Gamma(X_\#\mu,Y_\#\mu)$ and 
$\int_{\RR^d\times \RR^d} |x-y|^2\,\pi(dx,dy)=\|X-Y\|_{L^2(\mu)}^2$.

\subsection{Existence of a unique flow}

In this section, we recall the existence and uniqueness result for the aggregation equation \eqref{EqInter} obtained in \cite{CJLV} (and extend it a bit for non-globally Lipschitz-continuous potentials).
For $\rho \in C([0,T];\calP_2(\RR^d))$, we define the velocity field $\widehat{a}_{\rho}$ by
\beq\label{achapo}
\achapo_{\rho}(t,x)
:=
 -\int_{\RR^d} \nabWchapo(x-y) \rho(t,dy)\,,
\eeq
where we have used the notation
$$
\nabWchapo(x) := \left\{
\begin{array}{ll} \nabla W(x), \qquad & \mbox{ for } x\neq 0, 
\\
0, & \mbox{ for } x=0. \end{array}
\right.
$$
Due to the $\lambda$-convexity of $W$, see {\bf (A2)}, we deduce that, for all 
$x$, $y$ in $\RR^d\setminus \{0\}$, 
\beq\label{lambdaconv}
\langle \nabla W(x)-\nabla W(y) , x-y\rangle \geq \lambda |x-y|^2.
\eeq
Moreover, since $W$ is even, $\nabla W$ is odd and by taking $y=-x$ in
\eqref{lambdaconv}, we deduce that inequality \eqref{lambdaconv} still
holds for $\nabWchapo$, even when $x$ or $y$ vanishes:
\beq
\label{lambdaconvWchapo}
\forall\, x,y\in\RR^d, \qquad
\langle\nabWchapo(x)-\nabWchapo(y),x-y\rangle \geq \lambda |x-y|^2.
\eeq
This latter inequality provides a one-sided Lipschitz-continuity (OSL) estimate 
for the velocity field $\achapo_\rho$ defined in \eqref{achapo}, i.e. we have
\begin{equation*}
\forall\, x,y\in\RR^d, \ t \geq 0, \qquad
\bigl\langle \achapo_{\rho}(t,x)-\achapo_{\rho}(t,y), x-y\bigr\rangle \leq -\lambda |x-y|^2.
\end{equation*}
We recall that, for a velocity field $b\in L^\infty([0,+\infty);L^\infty(\RR^d))^d$ satisfying an OSL estimate, i.e. 
$$
\forall\, x,y\in \RR^d,\ t \geq 0, \qquad
\langle b(t,x)-b(t,y),x-y\rangle \leq \alpha(t) |x-y|^2,
$$ 
for $\alpha\in L^1_{loc}([0,+\infty))$,
it has been established in \cite{Filippov} that a Filippov characteristic flow
could be defined. For $s\geq 0$ and $x\in \RR^d$, a Filippov characteristic
starting from $x$ at time $s$ is defined
as a continuous function $Z(\cdot;s,x)\in C([s,+\infty);\RR^d)$ 
such that $\frac{\pa}{\pa t}Z(t;s,x)$ exists 
for a.e. $t\in[s,+\infty)$ and satisfies $Z(s;s,x)=x$ together with the differential inclusion
$$
\frac{\pa}{\pa t} Z(t;s,x) \in \bigl\{ \textrm{\rm Convess}\bigl(\achapo_{\rho}\bigr)(t,\cdot)\bigr\}(Z(t;s,x)),
\qquad \textrm{\rm for a.e.} \quad t \geq s. 
$$
In this definition, $\{\textrm{\rm Convess}(\achapo_{\rho})(t,\cdot) \}(x)$ 
denotes the essential convex hull of the vector field 
$\achapo_{\rho}(t,\cdot)$ at $x$. We remind briefly the definition for the sake of completeness (see 
\cite{Filippov,AubinCellina} for more details). We denote by $\textrm{\rm Conv}(E)$ the classical convex
hull of a set $E \subset \RR^d$, i.e., the smallest closed convex set containing $E$. 
Given the vector field 
$\achapo_{\rho}(t,\cdot):\RR^d\rightarrow\RR^d$, its essential convex hull at point $x$ is defined as
$$
\bigl\{ \textrm{\rm Convess} \bigl(\achapo_{\rho} \bigr)(t,\cdot) \bigr\}(x) :=\bigcap_{r>0} \bigcap_{N\in \mathcal{N}_0} 
\textrm{\rm Conv}\bigl[ \achapo_{\rho}\bigl( t, B(x,r)\setminus N\bigr)\bigr]\,,
$$
where $\mathcal{N}_0$ is the set of zero Lebesgue measure sets.

Moreover, we have the semi-group property: for any $t,\tau,s \in [0,+\infty)$ such that $t \geq \tau \geq s$ and 
$x\in \RR^d$,
\begin{equation}
\label{eq:characteristics}
Z(t;s,x)=Z(\tau;s,x)+\int_{\tau}^t \achapo_{\rho}\bigl(\sigma,Z(\sigma;s,x)\bigr)\,d\sigma.
\end{equation}
From now on, we will make use of the notation $Z(t,x)=Z(t;0,x)$.
Using this characteristic, it has been established in \cite{PoupaudRascle} that solutions to the
conservative transport equation with a given bounded and one-sided Lipschitz-continuous velocity field could be defined as the pushforward of the initial condition by the Filippov characteristic flow.
Based on this approach, existence and uniqueness of solutions
to \eqref{EqInter}
 defined by a Filippov flow
has been established in \cite{CJLV}. More precisely the statement reads:

\begin{theorem}\label{Exist}\cite[Theorem 2.5 and 2.9]{CJLV}
(i) Let $W$ satisfy assumptions {\bf (A0)--(A3)} 
and let $\rho^{ini}$ be given in $\calP_2(\RR^d)$.
Then, there exists a unique solution $\rho \in C([0,+\infty);\calP_2(\RR^d))$
satisfying, in the sense of distributions, the aggregation equation
\begin{equation}
\label{eq:agreg:TH}
\pa_t \rho + \dv\bigl(\achapo_{\rho} \rho \bigr) = 0, \qquad \rho(0,\cdot)=\rho^{ini},
\end{equation}
where $\achapo_{\rho}$ is defined by \eqref{achapo}.

This solution may be represented as the family of pushforward measures 
$(\rho(t):=Z_{\rho}(t,\cdot){}_\# \rho^{ini})_{t \geq 0}$ 
where $(Z_{\rho}(t,\cdot))_{t \geq 0}$ is the unique Filippov characteristic flow associated to the velocity
field $\achapo_{\rho}$.

Moreover, the flow $Z_{\rho}$ is Lipschitz-continuous and we have
\[
\sup_{x,y\in \RR^d, \, x \not = y} 
\frac{\vert Z_{\rho}(t,x) - Z_{\rho}(t,y) \vert}{\vert x- y \vert}
\leq e^{{\vert \lambda \vert}t}, \quad t \geq 0.
\]

At last, if $\rho$ and  $\rho'$
are the respective solutions of 
\eqref{eq:agreg:TH}
with $\rho^{ini}$ and $\rho^{ini,\prime}$ as initial conditions in ${\mathcal P}_{2}(\RR^d)$, 
then 
\[
d_{W}(\rho(t),\rho'(t)) \leq e^{\vert \lambda \vert t}
d_{W}(\rho^{ini},\rho^{ini,\prime}), \qquad t \geq 0.  
\]
(ii) Let $W$ satisfy \textbf{\bf (A0)}--\textbf{\bf (A2)} and be radial, $\lambda$ be (strictly) positive and let $\rho^{ini}$ be given in $\calP_2(\RR^d)$ with compact support included in $B_\infty(M_1,R)$, where $M_1$ is the first moment of $\rho^{ini}$ ({\em i.e.} its center of mass) and $B_\infty(M_1,R)$ the closed ball for the infinite norm on $\RR^d$ centered at $M_1$ with radius $R$. Then, there exists a unique solution $\rho \in C([0,+\infty);\calP_2(\RR^d))$ with support included in $B_\infty(M_1,R)$ satisfying, in the sense of distributions, the aggregation equation \eqref{eq:agreg:TH} where $\achapo_{\rho}$ is defined by \eqref{achapo}.

Moreover, the flow $Z_{\rho}$ is Lipschitz-continuous and we have
\begin{equation}\label{bound1Xbis}
\sup_{x,y\in \RR^d, \, x \not = y} 
\frac{\vert Z_{\rho}(t,x) - Z_{\rho}(t,y) \vert}{\vert x- y \vert}
\leq e^{-\lambda t}, \quad t \geq 0.
\end{equation} 

At last, if $\rho^{ini}$ and $\rho^{ini,\prime}$ have a bounded support, 
then,
\[
d_{W}(\rho(t),\rho'(t)) \leq 
d_{W}(\rho^{ini},\rho^{ini,\prime}), \qquad t \geq 0.
\]
\end{theorem}

The stability estimates that are present in this result are Dobrushin type estimates in the quadratic Wasserstein distance, in the case where the kernel is not Lipschitz-continuous but only one-sided Lipschitz-continuous. See \cite{dob} and \cite{Golse}. 

We mention that the solution, 
which is here represented by the Filippov characteristic flow, 
may be also constructed as a 
gradient flow solution in the Wasserstein space ${\mathcal P}_{2}(\RR^d)$, see
\cite{Carrillo}.
Here it is also important to remark that   \eqref{bound1Xbis} is true under the sole assumptions 
\textbf{\bf (A0)}--\textbf{\bf (A2)}
whenever $\lambda >0$ (which is a mere consequence of 
\eqref{eq:aux:proof:1}
and
\eqref{eq:aux:proof:2}
below). In that case, it ensures that $B_2(M_1,R)$ (the closed Euclidean ball) is preserved by the flow {\em without the assumption that $W$ is radial.}
As a result, it may be tempting to address the analysis below without requiring the potential to be radial. 
Nevertheless, the problem is that the 
numerical scheme does not satisfy a similar property. Indeed, the Euclidean ball $B_2(M_1,R)$ is not convex from a numerical point of view, that is to say, if we regard the mesh underpinning the scheme, then the union of the square cells whose center is included in $B_2(M_1,R)$ is not convex.
Due to this drawback, the flow associated to the scheme does not preserve the ball $B_2(M_1,R)$.
This is in contrast with Lemma
\ref{lem:CFL:lambda:>0} below, which shows that, in the radial setting, the ball $B_\infty(M_1,R+\Delta x)$ is kept stable by the scheme, where 
$\Delta x$ is the step of the spatial mesh. This latter fact is the reason why we here assume that the potential is radial.
\vskip 4pt

\begin{proof}
For the first two statements of the Theorem, existence of a unique solution and Lipschitz-continuity of the flow, we refer to \cite{CJLV}. 
These statements remain true whenever the sole \textbf{\bf (A0)}--\textbf{\bf (A2)} hold true, $W$ is radial, $\lambda$ is (strictly) positive and the support of $\rho^{ini}$ is bounded, provided that the notion of solution is limited to collections $(\rho(t,\cdot))_{t \geq 0}$ that have a compact support, uniformly in $t$ in compact subsets. Indeed, if we denote by $M_1(t)$ the center of mass of the solution at time $t$, namely $M_1(t) := \int_{\RR^d} \rho(t,dx)$, then this center of mass is known to be preserved: $M_1(t) = M_1(0) =: M_1$ (see \cite{CJLV} or Lemma \ref{bounddismom} below for the discrete counterpart). 
Now, if $\lambda \geq 0$ and if $W$ is radial, $\nabla W(x - y)$ is positively proportional to $x-y$, so that $-\nabla W(x - y)$ is parallel to $x - y$ and directed from $x$ to $y$. 
Thus, if $\rho(t)$ is zero outside the ball
$B_\infty(M_1,R)$, then, for any $x \in \partial B_\infty(M_1,R)$, the velocity $\achapo_\rho(t,x)$ is directed toward the interior of $B_\infty(M_1,R)$. This shows that $B_\infty(M_1,R)$ is preserved by the flow and guarantees that $\rho(t)$ has its support included in $B_\infty(M_1,R)$ for any time $t \geq 0$, if it is the case for $t = 0$. 
Given the fact that the support of $\rho(t)$ remains bounded in $B_\infty(M_1,R)$, everything works as if $W$ was globally Lipschitz-continuous. 
Existence and uniqueness of a solution to the aggregation equation can thus be proved by a straightforward localization argument. Indeed, observe that from the very definition of the velocity $a$, the Lipschitz-continuity constant of $W$ that is involved in the existence and uniqueness theory is the local one of $W$ on the compact subset 
$B_{\infty}(M_{1},R)$, provided that the support of $\rho^{ini}$ is included in 
$B_{\infty}(M_{1},R)$. 

Now it only remains to prove the two inequalities regarding the Wasserstein distance between solutions starting from different data. Under assumptions {\bf (A0)--(A3)} on the potential, it was proven in \cite{CJLV}, but with a constant $2|\lambda|$ instead of $|\lambda|$ in the exponential (as in \cite{dob} and \cite {Golse}, where the convolution operator is however replaced with a slightly more general integral operator), thus we here provide a proof of the present better estimate. 

We consider the two Filippov flows $(Z_{\rho}(t,\cdot))_{t \geq 0}$
and $(Z_{\rho'}(t,\cdot))_{t \geq 0}$ as defined in the statement of 
Theorem \ref{Exist}. We recall that 
\begin{equation}
\label{eq:aux:proof}
Z_{\rho}(t,\cdot){}_{\#} \rho^{ini}=
\rho(t,\cdot), \qquad 
Z_{\rho'}(t,\cdot){}_{\#} \rho^{ini,\prime}=
\rho'(t,\cdot), \qquad t \geq 0.
\end{equation}
To simplify, we just write 
$Z(t,\cdot) = Z_{\rho}(t,\cdot)$
and
$Z'(t,\cdot) = Z_{\rho'}(t,\cdot)$. Then, for any $x,y \in \RR^d$ and $t \geq 0$,
\begin{equation*}
\begin{split}
&\frac{d}{dt} \vert Z(t,x) - Z'(t,y) \vert^2
\\
&\hspace{15pt}= - 2 \Bigl\langle
Z(t,x) - Z'(t,y), 
\\
&\hspace{45pt} \int_{\RR^d}
 \widehat{\nabla W}\bigl(  Z(t,x)
- 
Z(t,x') \bigr) \rho^{ini}(dx') 
-
\int_{\RR^d}
 \widehat{\nabla W}\bigl(  Z'(t,y)
- 
Z'(t,y') \bigr) \rho^{ini,\prime}(dy')
\Bigr\rangle.  
\end{split}
\end{equation*}
Call $\pi \in \Gamma_{0}(\rho^{ini},\rho^{ini,\prime})$ an optimal plan between $\rho^{ini}$ and $\rho^{ini,\prime}$. 
Then,
\begin{equation*}
\begin{split}
&\frac{d}{dt} \vert Z(t,x) - Z'(t,y) \vert^2
\\
&\hspace{5pt}= - 2 \Bigl\langle
Z(t,x) - Z'(t,y), 
 \int_{\RR^{2d}}
\bigl[
 \widehat{\nabla W}\bigl(  Z(t,x)
- 
Z(t,x') \bigr) 
-
 \widehat{\nabla W}\bigl(  Z'(t,y)
- 
Z'(t,y') \bigr) \bigr] 
\pi(dx',dy')
\Bigr\rangle.  
\end{split}
\end{equation*}
Integrating in $(x,y)$ with respect to $\pi$, we get 
\begin{equation*}
\begin{split}
&\frac{d}{dt} \int_{\RR^{2d}} \vert Z(t,x) - Z'(t,y) \vert^2
\pi(dx,dy)
\\
&\hspace{15pt}
=- 2
 \int_{\RR^{2d}}
  \int_{\RR^{2d}}
 \Bigl\langle
Z(t,x) - Z'(t,y), 
\\
&\hspace{95pt}
\bigl[
 \widehat{\nabla W}\bigl(  Z(t,x)
- 
Z(t,x') \bigr) 
-
 \widehat{\nabla W}\bigl(  Z'(t,y)
- 
Z'(t,y') \bigr) \bigr] 
\Bigr\rangle \, \pi(dx,dy) \, \pi(dx',dy'). 
\end{split}
\end{equation*}
Thanks to the fact that $\widehat{\nabla W}$ is odd, see
\textbf{(A0)}, we can write, by a symmetry argument,
\begin{equation*}
\begin{split}
&\frac{d}{dt} \int_{\RR^{2d}} \vert Z(t,x) - Z'(t,y) \vert^2
\pi(dx,dy)
\\
&\hspace{15pt}
=- 
 \int_{\RR^{2d}}
  \int_{\RR^{2d}}
 \Bigl\langle
Z(t,x) - Z'(t,y)- \bigl(
Z(t,x') - Z'(t,y')
\bigr), 
\\
&\hspace{95pt}
\bigl[
 \widehat{\nabla W}\bigl(  Z(t,x)
- 
Z(t,x') \bigr) 
-
 \widehat{\nabla W}\bigl(  Z'(t,y)
- 
Z'(t,y') \bigr) \bigr] 
\Bigr\rangle \, \pi(dx,dy) \, \pi(dx',dy'). 
\end{split}
\end{equation*}
Using 
\eqref{lambdaconvWchapo}, we obtain
\begin{equation}
\label{eq:aux:proof:1}
\begin{split}
&\frac{d}{dt} \int_{\RR^{2d}}\vert Z(t,x) - Z'(t,y) \vert^2
\pi(dx,dy)
\\
&\hspace{15pt}
\leq - \lambda
 \int_{\RR^{2d}}
  \int_{\RR^{2d}}
 \bigl\vert 
Z(t,x) - Z'(t,y)- \bigl(
Z(t,x') - Z'(t,y')
\bigr)
\bigr\vert^2 \, \pi(dx,dy) \, \pi(dx',dy'). 
\end{split}
\end{equation}
Observe that the above right-hand side is equal to
\begin{equation}
\label{eq:aux:proof:2}
\begin{split}
&\int_{\RR^{2d}}
  \int_{\RR^{2d}}
 \bigl\vert 
Z(t,x) - Z'(t,y)- \bigl(
Z(t,x') - Z'(t,y')
\bigr)
\bigr\vert^2 \, \pi(dx,dy) \, \pi(dx',dy')
\\
&\hspace{15pt}= 2 \int_{\RR^{2d}}
 \bigl\vert 
Z(t,x) - Z'(t,y)
\bigr\vert^2 \, \pi(dx,dy)
 - 2 \biggl\vert 
  \int_{\RR^{2d}}
\bigl( Z(t,x) - Z'(t,y)
\bigr)
 \, \pi(dx,dy) \biggr\vert^2.
\end{split}
\end{equation}
\vskip 4pt

\textbf{1st case.} If $\lambda \leq 0$, we deduce 
from 
\eqref{eq:aux:proof:1}
and
\eqref{eq:aux:proof:2}
that 
\begin{equation*}
\begin{split}
&\frac{d}{dt} \int_{\RR^{2d}}\vert Z(t,x) - Z'(t,y) \vert^2
\pi(dx,dy)
\leq 2 \vert \lambda \vert 
  \int_{\RR^{2d}}
 \bigl\vert 
Z(t,x) - Z'(t,y)
\bigr\vert^2 \, \pi(dx,dy), 
\end{split}
\end{equation*}
which suffices to complete the proof of the first claim by noting 
that 
\begin{equation*}
\int_{\RR^{2d}}\vert Z(0,x) - Z'(0,y) \vert^2
\pi(dx,dy)
= 
\int_{\RR^{2d}}\vert x - y \vert^2
\pi(dx,dy) = d_{W}(\rho^{ini},\rho^{ini,\prime})^2,
\end{equation*}
and
\begin{equation*}
\int_{\RR^{2d}}\vert Z(t,x) - Z'(t,y) \vert^2
\pi(dx,dy) \geq 
d_{W}(\rho(t),\rho(t)^{\prime})^2,
\end{equation*}
see \eqref{eq:aux:proof}.
\vskip 4pt

\textbf{2nd case.} If $\lambda \geq 0$, 
we just use the fact that the right-hand side in 
\eqref{eq:aux:proof:1} is non-positive. Proceeding as above, this permits to complete the proof of the second claim. 
\end{proof}

\subsection{Main result}

The aim of this paper is to prove the convergence at order $1/2$
of an upwind type scheme in distance $d_W$ for the aggregation equation.
The numerical scheme is defined as follows.
We denote by $\Delta t$ the time step and consider a Cartesian grid with
step $\Delta x_i$ in the $i$th direction, $i=1,\ldots,d$; we then let $\Delta x:=\max_i \Delta x_i$.
We also introduce the
following notations. 
For a multi-index $J=(J_1, \ldots, J_d)\in \ZZ^d$, 
we call
$C_J:=[(J_1-\frac{1}{2})\Delta x_1,(J_1+\frac{1}{2})\Delta x_1)\times \ldots [(J_d-\frac{1}{2})\Delta x_d,(J_d+\frac{1}{2})\Delta x_d)$ the corresponding elementary cell. The center of the cell is denoted by
$x_J := (J_1\Delta x_1, \ldots, J_d \Delta x_d).$
Also, we let 
$e_i := (0,\ldots,1,\ldots,0)$
be the $i$th vector of the canonical basis, for $i \in \{1,\ldots,d\}$, and 
we expand the velocity field
in the canonical basis under the form
 $a=(a_1,\ldots,a_d)$. 

For a given nonnegative measure $\rho^{ini}\in
\calP_2(\RR^d)$, we put, for any $J\in \ZZ^d$, 
\beq\label{disrho0}
\rho_{J}^0:= \int_{C_J} \rho^{ini}(dx)\geq 0.
\eeq 
Since $\rho^{ini}$ is a probability
measure, the total mass of the system is $\sum_{J\in \ZZ^d} \rho_{J}^0 = 1$. 
We then construct iteratively the collection $((\rho_J^n)_{J \in \ZZ^d})_{n \in {\mathbb N}}$, each $\rho^n_{J}$ being intended to provide an approximation of the value $\rho(t^n,x_J)$, for $J\in \ZZ^d$.
Assuming that the approximating sequence $(\rho_{J}^n)_{J\in \ZZ^d}$ is already given at time $t^n:=n \Delta t$,
we compute the approximation at time $t^{n+1}$ by:
\beq\label{dis_num}
\begin{array}{ll}
\ds \rho_{J}^{n+1} :=  \ds \rho_{J}^n - \sum_{i=1}^{d} \frac{\Delta t}{\Delta x_i}
\Big(({a_i}^n_{J})^+ \rho_{J}^n - ({a_i}^n_{J+e_i})^- \rho_{J+e_i}^n 
-({a_i}^n_{J-e_i})^+ \rho_{J-e_i}^n + ({a_i}^n_{J})^- \rho_{J}^n \Big).
\end{array}
\eeq
The notation $(a)^+ = \max\{0,a\}$ stands for the positive part of the real $a$
and respectively $(a)^- = \max\{0,-a\}$ for the negative part.
The macroscopic velocity is defined by
\begin{equation}
\label{def:aij} 
{a_i}^n_{J} :=  -\sum_{K\in \ZZ^d} \rho_{K}^n \,D_iW_J^K, 
\quad \mbox{ where } \quad
D_iW_J^K := \widehat{\pa_{x_i} W}\bigl(x_J-x_K \big).
\end{equation}
Since $W$ is even, we also have:
\begin{equation}
\label{eq:gradients:symmetry:cells}
D_iW_{J}^{K} =  -D_iW^{J}_{K}.
\end{equation}

The main result of this paper is the proof of the convergence at order $1/2$ 
of the above upwind scheme. More precisely the statement reads:

\begin{theorem}\label{TH}
(i) Assume that $W$ satisfies hypotheses {\bf (A0)--(A3)}
and that the so-called strict $\frac 12$-CFL condition holds:
\begin{equation}
\label{CFL}
w_\infty  \sum_{i=1}^d \frac{\Delta t}{\Delta x_i} < \frac 12,
\end{equation}
with $w_{\infty}$ as in 
\eqref{borngradW}.

For $\rho^{ini} \in \calP_2(\RR^d)$, let $\rho=(\rho(t))_{t \ge 0}$ be the unique measure solution to the aggregation equation with initial data $\rho^{ini}$, as
given by Theorem \ref{Exist}.
Define $((\rho_J^n)_{J\in \ZZ^d})_{n \in {\mathbb N}}$
as in 
\eqref{disrho0}--\eqref{dis_num}--\eqref{def:aij}
and let
$$
\rho_{\Delta x}^n := \sum_{J\in \ZZ^d} \rho_J^n \delta_{x_J}, 
\quad n \in {\mathbb N}. 
$$
Then, there exists a nonnegative constant $C$, 
only depending on $\lambda$, $w_{\infty}$ and $d$, such that, for all $n\in \NN^*$,
\begin{equation}
\label{eq:TH:bound:1}
d_W(\rho(t^n),\rho_{\Delta x}^n ) \leq C \, e^{\vert \lambda  \vert (1+\Delta t)t^n} \, \bigl( \sqrt{t^n \Delta x} + \Delta x \bigr).
\end{equation}

(ii) Assume that $W$ is radial and satisfies hypotheses
\textbf{\bf (A0)}--\textbf{\bf(A2)} with $\lambda$ (strictly) positive, that 
$\rho^{ini}$ is compactly supported in $B_\infty(M_1,R)$ where $M_1$ is the center of mass of $\rho^{ini}$, and that the CFL condition \eqref{CFL} holds, with $w_{\infty}$ defined as
\begin{equation}
\label{borngradW:lambda>0}
w_{\infty} = \sup_{x \in B_{\infty}(0,2R+2\Delta x) \setminus \{0\} } \vert \nabla W(x) \vert,
\end{equation}
Assume also that $\Delta t \leq 1/2$ and $2 \lambda \Delta t < 1$.
Then, there exists a nonnegative constant $C$, 
only depending on $\lambda$, $w_{\infty}$, $d$ and $R$ such that, for all $n\in \NN^*$, \eqref{eq:TH:bound:1} is valid, as well as 
\begin{equation}
\label{eq:TH:bound:2}
d_W(\rho(t^n),\rho_{\Delta x}^n) \leq C  \, 
\bigl( \sqrt{\Delta x} + \Delta x \bigr),
\end{equation}
which proves that the error can be uniformly controlled in time. 
\end{theorem}

We stress the fact that, under the setting defined in $(ii)$,
\eqref{eq:TH:bound:1}
is valid. In small time, it provides a better estimate than  
\eqref{eq:TH:bound:2}. 
As indicated in the statement, the constant $C$
in \eqref{eq:TH:bound:2} may depend
on the value of $R$ 
in the assumption 
$\textrm{\rm Supp}(\rho^{ini}) \subset B_{\infty}(M_{1},R)$.

We also point out that, although the computations below are performed for the sole upwind scheme,
the first part of the statement, which holds true under the full set of 
hypotheses {\bf (A0)--(A3)}, 
can be 
 straightforwardly adapted to other diffusive schemes, see for instance 
our previous article \cite{DLV}.
As for $(ii)$, the statement remains true provided that the supports of the approximating measures 
$(\rho^n)_{n \geq 0}$ remain bounded as $n$ grows up. It must be stressed that there are some schemes for which the latter property fails (e.g. 
Lax-Friedrichs' scheme).

Moreover, as already mentioned in Introduction, the convergence
rate is optimal; this latter fact will be illustrated by numerical examples in
Section \ref{sec:sim}.

\begin{example}
\label{ex1D}
In one dimension, the scheme \eqref{dis_num} reads
$$
\rho_{i}^{n+1} = \rho_i^n - \frac{\Delta t}{\Delta x}\Big((a_i^n)^+ \rho_i^n - (a_{i+1}^n)^- \rho_{i+1}^n 
- (a_{i-1}^n)^+ \rho_{i-1}^n + (a_i^n)^-\rho_i^n\Big),
$$
where $i$ is just taken in $\ZZ$. 
The scheme has then the following interpretation. Given $\rho^n_{\Delta x} = \sum_{j\in \ZZ} \rho_j^n \delta_{x_j}$,
we construct the approximation at time $t^{n+1}$ by implementing the following two steps:
\begin{itemize}
\item The Delta mass $\rho_i^n$ located at position $x_i$ moves with velocity $a_i^n$ to the position $x_i+a_i^n \Delta t$.
Under the CFL condition $w_\infty \Delta t \leq \Delta x$ (which is obviously weaker than what we require in \eqref{CFL}), the point $x_i+a_i^n\Delta t$
belongs to the interval $[x_i,x_{i+1}]$ if $a_i^n\geq 0$, and to the interval $[x_{i-1},x_{i}]$ if $a_i^n\leq 0$.
\item Then the mass $\rho_{i}^n$ is split into two parts; 
if $a_{i}^n \geq 0$,
a fraction $a_{i}^n \Delta t/\Delta x$ of it is transported to the cell $i+1$, while the remaining fraction is left in cell $i$; 
if $a_{i}^n \leq 0$,
the same fraction $\vert a_{i}^n \vert \Delta t/\Delta x$ of the mass is not transported to the cell $i+1$
but to the cell $i-1$. This procedure may be regarded as a linear interpolation of the mass $\rho_{i}^n$ between the points
$x_{i}$ and $x_{i+1}$ if $a_{i}^n \geq 0$ and between the points 
$x_{i}$ and $x_{i-1}$ if $a_{i}^n \leq 0$.  
\end{itemize}
This interpretation holds only in the one dimensional case. However thanks to this interpretation, we can define a forward semi-Lagrangian scheme in any dimension on (unstructured) simplicial meshes, which is then different from \eqref{dis_num}. 
Such a scheme is introduced in Section \ref{sec:unstruct}.

Finally, we emphasize that this scheme differs from the standard finite volume upwind scheme in which the velocity
is computed at the interface $a_{i+1/2}^n$. This subtlety is due to the particular structure of the equation, as the latter requires the product $\achapo_{\rho} \rho$ to be defined properly.
A convenient way to make it proper is to compute, in the discretization, the velocity and the density at the same grid points.
This fact has already been noticed
in \cite{sinum,sisc} and is also 
illustrated numerically in Section \ref{sec:sim}. \end{example}

\section{Numerical approximation}
\label{sec:num}

\subsection{Properties of the scheme}

The following lemma explains why we called CFL the condition on the ratios 
$(\Delta t/\Delta x_{i})_{i=1,\cdots,d}$ that we formulated in the statement of 
Theorem 
\ref{TH}.
\begin{lemma}
\label{lem:CFL}
Assume that $W$ satisfies hypotheses {\bf (A0)--(A3)}
and that the condition
\eqref{CFL}
is in force. For $\rho^{ini}\in \calP_2(\RR^d)$, define 
$(\rho_{J}^0)_{J \in \ZZ^d}$ by \eqref{disrho0}.
Then the sequences $(\rho_J^n)_{n \in \NN,J \in \ZZ^d}$ and $({a_i}_J^n)_{n \in \NN,J \in \ZZ^d}$, $i=1,\ldots,d$, 
given by the scheme defined in 
\eqref{dis_num}--\eqref{def:aij},
satisfy, for all $J\in \ZZ^d$ and $n\in \NN$,
$$
\rho_{J}^n \geq 0, \qquad |{a_i}_{J}^n|\leq w_\infty, \quad i=1,\ldots, d,
$$
and, for all $n \in \NN$,
\begin{equation*} 
\sum_{J \in \ZZ^d} \rho_{J}^n=1.
\end{equation*}
\end{lemma}
\begin{proof}
The total initial mass of the system is $\sum_{J} \rho_{J}^0=1$. 
By summing equation \eqref{dis_num} over $J$,
we can show that the mass is conservative, namely, for all $n\in \NN^*$, $\sum_{J} \rho_{J}^n=
\sum_{J} \rho_{J}^0=1$.

Also, we can rewrite equation \eqref{dis_num} as
\begin{equation}
\rho_{J}^{n+1} = \rho_{J}^n \left[ 1 - \sum_{i=1}^d \frac{\Delta t}{\Delta x_i} |{a_i}^n_{J}| \right]
+ \sum_{i=1}^d \rho_{J+e_i}^n \frac{\Delta t}{\Delta x_i}({a_i}^n_{J+e_i})^-
+ \sum_{i=1}^d \rho_{J-e_i}^n \frac{\Delta t}{\Delta x_i}({a_i}^n_{J-e_i})^+.
\label{schemarho}
\end{equation}

We prove by induction on $n$ that $\rho_{J}^n \geq 0$ for all $J \in \ZZ^d$ and for all $n \in {\mathbb N}$.
Indeed, if, for some $n \in {\mathbb N}$, 
it holds $\rho_{J}^n \geq 0$
for all $J \in \ZZ^d$, 
then, by definition \eqref{def:aij} and assumption \eqref{borngradW},
we clearly have
$$
|{a_i}_{J}^{n}|\leq w_\infty \sum_{K\in \ZZ^d} \rho_{K}^n = w_\infty, \qquad i=1,\ldots,d.
$$
Then, assuming that the condition \eqref{CFL} holds, we deduce that, in the relationship \eqref{schemarho},
all the coefficients in front of $\rho_{J}^n$, $\rho_{J-e_i}^n$ and $\rho_{J+e_i}^n$,
$i=1,\ldots,d$, are nonnegative.
Thus, using the induction assumption, we deduce that
$\rho_{J}^{n+1}\geq 0$ for all $J\in \ZZ^d$.
\end{proof}

In the following lemma, we collect two additional properties of the scheme: the conservation of the 
center of mass and the finiteness of the second order moment.

\begin{lemma}\label{bounddismom}
Let $W$ satisfy {\bf (A0)--(A3)}
and condition \eqref{CFL} be in force. For
$\rho^{ini}\in \calP_2(\RR^d)$,  define $(\rho_{J}^0)_{J\in \ZZ^d}$ by \eqref{disrho0}.
Then, the sequence $(\rho_{J}^n)_{J\in \ZZ^d}$ given by the numerical scheme
\eqref{dis_num}--\eqref{def:aij} satisfies:

$(i)$ Conservation of the center of mass.
For all $n\in \NN^*$,
$$
\sum_{J\in \ZZ^d} x_J \rho_{J}^n = \sum_{J\in \ZZ^d} x_J \rho_{J}^0.
$$
We will denote the right-hand side (and thus the left-hand side as well) by $\m1n{}$. 

$(ii)$ Bound on the second moment. There exists a constant $C>0$, independent of the parameters of the mesh, such that, 
for all $n\in \NN^*$, 
\begin{equation*}
M_{2,\Delta x}^n := \sum_{J\in \ZZ^d} |x_J|^2\rho_{J}^n \leq e^{C t^n} \big(M_{2,\Delta x}^0 + C\big),
\end{equation*}
where we recall that $t^n=n\Delta t$.
\end{lemma}

\begin{proof}
We recall from Lemma \ref{lem:CFL} that, for all $n\in \NN$, the sequence $(\rho_{J}^n)_{J \in \ZZ^d}$ is nonnegative and that its sum is equal to 1.

$(i)$ Using \eqref{dis_num} together with a discrete integration by parts, we have:
$$
\begin{array}{ll}
\ds \sum_{J\in \ZZ^d} x_J\rho_{J}^{n+1} = & \ds \sum_{J\in
\ZZ^d} x_J\rho_{J}^n - \sum_{i=1}^d\frac{\Delta t}{\Delta x_i} \sum_{J\in
\ZZ^d} \left(({a_i}^n_{J})^+ \, \rho_{J}^n \big(x_J-x_{J+e_i}\big)
-({a_i}^n_{J})^- \, \rho_{J}^n \big(x_{J-e_i}-x_J\big) \right).
\end{array}
$$
By definition of $x_J$, we deduce
$$
\sum_{J\in \ZZ^d} x_J\rho_{J}^{n+1} = \sum_{J\in \ZZ^d}
x_J\rho_{J}^n + \Delta t \sum_{i=1}^d \sum_{J\in \ZZ^d} {a_i}^n_{J} \, \rho_{J}^n.
$$
By 
definition of the macroscopic velocity \eqref{def:aij} and by 
\eqref{eq:gradients:symmetry:cells}, we also have
\begin{equation*}
\begin{split}
\sum_{J\in \ZZ^d} {a_i}^n_{J} \, \rho_{J}^n =
-\sum_{J\in \ZZ^d} \sum_{K\in \ZZ^d} D_iW_{J}^{K}\, \rho_{K}^n \, \rho_{J}^n
&= \sum_{J\in \ZZ^d} \sum_{K\in \ZZ^d} D_iW^{J}_{K}\, \rho_{K}^n \, \rho_{J}^n
\\
&= \sum_{J\in \ZZ^d} \sum_{K\in \ZZ^d} D_iW^{K}_{J}\, \rho_{K}^n \, \rho_{J}^n,
\end{split}
\end{equation*}
where we exchanged the role of $J$ and $K$ in the latter sum. We deduce that
it vanishes. Thus,
$$
\sum_{J\in \ZZ^d} x_J\rho_{J}^{n+1} = \sum_{J\in \ZZ^d} x_J\rho_{J}^n.
$$

$(ii)$ For the second moment, still using \eqref{dis_num} and a similar discrete integration by parts,
we get
$$
\begin{array}{ll}
\ds \sum_{J\in \ZZ^d} |x_J|^2 \rho_{J}^{n+1} &=  \ds \sum_{J\in
\ZZ^d} |x_J|^2\rho_{J}^n  
\\
[2mm]
&\hspace{5pt}- \ds \sum_{i=1}^d
\frac{\Delta t}{\Delta x_i} \sum_{J\in
\ZZ^d} \Bigl[ ({a_i}^n_{J})^+ \, \rho_{J}^n \big(|x_J|^2-|x_{J+e_i}|^2\big)
-({a_i}^n_{J})^- \, \rho_{J}^n \big(|x_{J-e_i}|^2-|x_{J}|^2\big)
\Bigr].
\end{array}
$$
By definition of $x_J$, 
$|x_J|^2-|x_{J+e_i}|^2=-2J_i\, \Delta x_i^2 - \Delta x_i^2$ and
$|x_{J-e_i}|^2-|x_J|^2=-2J_i\, \Delta x_i^2 + \Delta x_i^2$. Therefore, we get
$$
\sum_{J\in \ZZ^d} |x_J|^2 \rho_{J}^{n+1} = \sum_{J\in \ZZ^d}
|x_J|^2 \rho_{J}^{n} + 2\Delta t \sum_{i=1}^d\sum_{J\in \ZZ^d} 
J_i \Delta x_i \, {a_i}^n_{J} \, \rho_{J}^n
+ \Delta t\sum_{i=1}^d \Delta x_i \sum_{J\in \ZZ^d} \rho_{J}^n |{a_i}^n_{J}|.
$$
As a consequence of Lemma \ref{lem:CFL}, we have $|{a_i}^n_{J}|\leq w_\infty$. 
Using moreover the mass conservation, we deduce that the last term is bounded by $w_\infty \Delta t \sum_{i=1}^d \Delta x_i$.
Moreover, applying Young's inequality and using the mass conservation again, we get
$$
\Big|\sum_{J\in \ZZ^d} {a_i}^n_{J} \, \rho_{J}^n \, J_i \Delta x_{i} \Big| \leq 
\frac{1}{2} \Big( w_\infty^2 + \sum_{J\in \ZZ^d} |J_{i} \Delta x_i|^2 \, \rho_{J}^n \Big)
\leq 
\frac{1}{2} \Big( w_\infty^2 + \sum_{J\in \ZZ^d}  \rho_{J}^n 
\, \vert x^n_{J} \vert^2\Big).
$$
We deduce then that there exists a nonnegative constant $C$ 
only depending on $d$ and $w_\infty$ such that
$$
\sum_{J\in \ZZ^d} |x_J|^2 \rho_{J}^{n+1} \leq
\Big(1+C\Delta t\Big)\sum_{J\in \ZZ^d} |x_J|^2 \rho_{J}^{n} 
+ C\Delta t\left(\sum_{i=1}^d\Delta x_i+1\right).
$$
We conclude the proof using a discrete version of Gronwall's lemma.
\end{proof}

{In case when $W$ is radial and satisfies \textbf{\bf (A0)}--\textbf{(A2)}, $\lambda$ is (strictly) positive and $\rho^{ini}$ has a bounded support, Lemmas \ref{lem:CFL} and \ref{bounddismom} become: 

\begin{lemma}
\label{lem:CFL:lambda:>0}
Assume that $W$ is radial and satisfies \textbf{\bf (A0)}--\textbf{\bf (A2)}, 
$\lambda$ is (strictly positive) and $\rho^{ini}$ 
has a bounded support, then the conclusions of 
Lemmas \ref{lem:CFL} and \ref{bounddismom}
remain true provided that $w_{\infty}$ is defined  as in 
\eqref{borngradW:lambda>0}. 

Moreover, for any $R \geq 0$ such that $\textrm{\rm Supp}(\rho^{ini}) \subset B_{\infty}(M_{1},R)$, it holds, for any $n \in \NN$,
\begin{equation*}
{\rm Supp}(\rho^n_{\Delta x}) \subset 
B_{\infty}(\m1n{},R+\Delta x),
\end{equation*}
that is 
\begin{equation*}
\forall J \in \ZZ^d, \quad x_{J} \not \in B_{\infty}(\m1n{},R+\Delta x)
\Rightarrow \rho^n_{J} = 0. 
\end{equation*}
\end{lemma}

The meaning of Lemma 
\ref{lem:CFL:lambda:>0} is pretty clear. For $R$ as in the statement, the mass, as defined by the numerical scheme, cannot leave the ball $B_{\infty}(\m1n{},R+\Delta x)$. We here recover the same idea as in Theorem \ref{Exist}.
\vskip 4pt

\begin{proof}
As long as we can prove that the mass, as defined by the numerical scheme, cannot leave the ball 
$B_{\infty}(\m1n{},R+\Delta x)$, the proof is similar to that of Lemmas 
\ref{lem:CFL} and \ref{bounddismom}. So, we focus on the second part of 
the statement. 
 
We first recall that $\rho^0_{J} = \int_{C_{J}}
 \rho^{ini}(dx)$, for $J \in \ZZ^d$. Hence, if $x_{J} \not \in B_{\infty}(\m1n{},R+\Delta x)$, we have 
$x_{J} \not \in B_{\infty}(M_{1},R+\Delta x/2)$ and then
 $C_{J} \cap B_{\infty}(M_{1},R) = \emptyset$ and thus $\rho^0_{J}=0$. Below, we prove by induction that the same holds true
for any $n \in \NN$.  To do so, we assume that
 there exists an integer $n \in \NN$ such that, 
 for all $J \in \ZZ^d$, $\rho^n_{J} =0$ if 
\begin{equation}
\label{eq:lambda>0:xJ}
x_{J} \not \in B_{\infty}(\m1n{},R + \Delta x).
\end{equation}
The goal is then to prove that, for any $J$ satisfying \eqref{eq:lambda>0:xJ}, 
$\rho^{n+1}_{J}=0$. By \eqref{schemarho}, it suffices to prove that, for any coordinate $i \in \{1,\cdots,d\}$ and any $J$ as in 
 \eqref{eq:lambda>0:xJ},  
\begin{equation}
\label{eq:lambda>0:induction}
\rho^{n}_{J+e_{i}} \bigl( {a_{i}}^n_{J+e_{i}}\bigr)^- = 0, 
\quad \textrm{\rm and}
\quad
\rho^{n}_{J-e_{i}} \bigl( {a_{i}}^n_{J-e_{i}}\bigr)^+ = 0.  
\end{equation} 
Without any loss of generality, we can assume that
there exists a coordinate $i_{0} \in \{1,\cdots,d\}$ such that 
$(x_{J})_{i_{0}} > R  + \Delta x + (\m1n{})_{i_{0}}$
(otherwise $(x_{J})_{i_{0}} < -R - \Delta x+ (\m1n{})_{i_{0}}$ and the argument below is the same).
Hence, $(x_{J+e_{i_{0}}})_{i_{0}} > 
R  + \Delta x + (\m1n{})_{i_{0}}$ and, by the induction hypothesis, 
$\rho^n_{J+e_{i_{0}}}=0$, which proves the first equality in 
\eqref{eq:lambda>0:induction}
when 
$i=i_{0}$. In order to prove the second equality when $i=i_{0}$, we notice 
from 
\eqref{def:aij} 
that 
\begin{equation*}
\begin{split}
{a_{i_{0}}}^n_{J-e_{i_{0}}} =  -\sum_{K\in \ZZ^d} \rho_{K}^n \, \widehat{\pa_{x_{i_{0}}} W}\bigl(x_{J-e_{i_{0}}}-x_K \bigr)
&=
 -\sum_{K\in \ZZ^d : (x_{K})_{i_{0}} \le 
 R  + \Delta x  +  (\m1n{})_{i_{0}}} 
\rho_{K}^n \, \widehat{\pa_{x_{i_{0}}} W}\bigl(x_{J-e_{i_{0}}}-x_K \bigr)
 \\
 &=
  -\sum_{K\in \ZZ^d : (x_{K})_{i_{0}} < (x_{J})_{i_{0}}} \rho_{K}^n \, \widehat{\pa_{x_{i_{0}}} W}\bigl(x_{J-e_{i_{0}}}-x_K \bigr) 
 \\
 &=
  -\sum_{K\in \ZZ^d : (x_{K})_{i_{0}} \leq (x_{J-e_{i_{0}}})_{i_{0}}} \rho_{K}^n \, \widehat{\pa_{x_{i_{0}}} W}\bigl(x_{J-e_{i_{0}}}-x_K \bigr).
 \end{split}
\end{equation*}
As $W$ is radial and 
$\lambda >0$, $\nabla W(x - y)$ is positively proportional to $x-y$. Hence, $\widehat{\pa_{x_{i_{0}}} W}(x_{J-e_{i_{0}}}-x_K) \geq 0$
when 
$(x_{K})_{i_{0}} \leq (x_{J-e_{i_{0}}})_{i_{0}}$. 
Therefore, 
$({a_{i_{0}}}^n_{J-e_{i_{0}}})^{+}=0$, 
which proves the second equality in 
\eqref{eq:lambda>0:induction}. 
 
It remains to prove 
\eqref{eq:lambda>0:induction} for $i \not = i_{0}$. Obviously,
$(x^n_{J-e_{i}})_{i_{0}}= 
(x^n_{J+e_{i}})_{i_{0}}
=
(x^n_{J})_{i_{0}} > R  + \Delta x + 
(\m1n{})_{i_{0}}$. By the induction hypothesis, 
$\rho^n_{J-e_{i}} = 
\rho^n_{J+e_{i}} = 0$, which completes the proof.  
\end{proof}
}

{\begin{remark}
Lemma 
\ref{lem:CFL:lambda:>0} is the main rationale for requiring 
$W$ to be radial. Indeed, the counter-example below shows that 
the growth of the support of $\rho^{ini}$ can be hardly controlled 
whenever 
$\lambda >0$ and 
$W$ is just assumed to satisfy 
\textbf{\bf (A0)}--\textbf{\bf (A2)}. 
Consider for instance the following potential in dimension $d=2$:
\begin{equation*}
W(x_{1},x_{2}) = \frac12 \bigl( x_{1} - q x_{2} \bigr)^2 +  \frac{q^2}2
x_{2}^2, \quad (x_{1},x_{2}) \in \RR^2,
\end{equation*}
where $q$ is a free integer whose value will be fixed later on. 
It is well checked that
\begin{equation*}
\partial_{x_{1}} W(x_{1},x_{2}) = x_{1}- q x_{2},
\quad \partial_{x_{2}} W(x_{1},x_{2}) = q ( q x_{2} - x_{1}) + q^2 x_{2}.
\end{equation*}
Standard computations show that the smallest eigenvalue of the Hessian matrix (which is independent of $(x_{1},x_{2})$) is 
\begin{equation*}
\begin{split}
&\frac{(1+2q^2) - 2q^2 \sqrt{1+1/(4q^4)}}{2} \sim_{q \rightarrow \infty} \frac12,
\end{split}
\end{equation*}
so that $W$ is $\lambda$-convex with $\lambda$ converging to $1/2$ as $q$ tends to $\infty$.

Take now a centered probability measure $\rho$ and compute the first coordinate of the velocity field
$\achapo_{\rho}$. By centering,
\begin{equation*}
\bigl(\achapo_{\rho}\bigr)_{1}(x_{1},x_{2})  = qx_{2}- x_{1}. 
\end{equation*}
In particular, if $x_{2}=1$, then 
$(\achapo_{\rho})_{1}(x_{1},1)  = q- x_{1}$, which is non-negative as long as $x_{1}< q$. Therefore, if the numerical scheme is initialized with some centered $\rho^0_{\Delta x}$ supported by the unit square $[-1,1]^2$, 
it holds 
\begin{equation*}
(\achapo_{\rho^0_{\Delta x}})_{1}(1,1) >0,
\end{equation*}
if $q>1$. 
Hence, provided that condition \eqref{CFL} holds true, $\rho^{1}_{\Delta x}$ charges the point $(1+\Delta x,1)$. Since the numerical scheme preserves the centering, 
we also have 
\begin{equation*}
(\achapo_{\rho^1_{\Delta x}})_{1}(1+\Delta x,1) >0, 
\end{equation*}
if $q>1+\Delta x$, 
and then $\rho^2_{\Delta x}$ also charges the point $(1+2\Delta x,1)$, and so on up until $(\Delta x \lfloor q/\Delta x \rfloor,1)$.  This says that there is no way to control the growth 
of the support of the numerical solution in terms of the sole lower bound of the Hessian matrix. Somehow, the growth of $\nabla W$ plays a key role. 
This is in stark contrast with the support of the real solution, which may be bounded independently of $q$, as emphasized in the proof of Theorem \ref{Exist}.

A possible way to overcome the fact that the numerical scheme does not preserve any ball containing the initial support in the general case when $W$ is not radial would be to truncate the scheme. We feel more reasonable not to address this question in this paper, as it would require to revisit in deep the arguments used to tackle the case $\lambda \leq 0$. 
\end{remark}

}

\subsection{Comparison with a potential non-increasing scheme}
\label{subse:potential:nonincreasing}

It must be stressed that the scheme could be defined differently in order to force the potential 
(or total energy: $\iint_{\RR^d \times \RR^d} W(x - y) \, \rho(dx)\, \rho(dy)$) 
to be non-increasing. Basically, this requires the velocity $a$ to be defined as a discrete derivative.

For simplicity, we provide the construction of the scheme in dimension 1 only. For a probability measure $\varrho \in {\mathcal P}(\ZZ)$ and a cell $I \in \ZZ$, we consider the following two discrete convolutions of finite differences:
\begin{equation*}
\begin{split}
&\frac1{\Delta x}\sum_{J \in \ZZ} 
\Bigl[
\Bigl( W\bigl( \Delta x (I + 1 - J) \bigr)
-
 W\bigl( \Delta x (I - J) \bigr)
 \Bigr)  \varrho_{J}
 \Bigr] 
\\ 
&\hspace{15pt} = \biggl[ \int_{\RR^d} \frac{W(x+\Delta x - y ) - W(x-y)}{\Delta x}
 \varrho_{\Delta x}(dy) \biggr]_{\vert x=I \Delta x}
\\
\textrm{\rm and}
\quad
&\frac1{\Delta x} \sum_{J \in \ZZ} 
\Bigl[
\Bigl( W\bigl( \Delta x (I -1 - J) \bigr)
-
 W\bigl( \Delta x (I - J) \bigr)
 \Bigr)  \varrho_{J}
 \Bigr]
 \\
&\hspace{15pt}  
 = \biggl[ \int_{\RR^d} \frac{W(x-\Delta x - y ) - W(x-y)}{\Delta x}
 \varrho_{\Delta x}(dy) \biggr]_{\vert x=I \Delta x},
 \end{split}
\end{equation*}
where, as before, $\varrho_{\Delta x}$ is obtained by pushing forward $\varrho$ by the mapping $y \mapsto \Delta x \, y$. 
The two terms above define velocities at the interfaces of the cell $I$. Namely, we call 
the first term $-a_{I+\tfrac12}$ and the second one $a_{I-\tfrac12}$. 
Of course, the sign $-$ in the former term guarantees the consistency of the notation, that is $a_{(I+1)-\tfrac12}$ is equal to $a_{I+\tfrac12}$. 

Following 
\eqref{dis_num}, the scheme is defined by: 
\begin{equation}
\label{dis_num_Carrillo}
\rho^{n+1}_{J} := \rho^n_{J} - \frac{\Delta t}{\Delta x}
 \Bigl( \bigl( a^n_{J+\tfrac12} \bigr)^{+} \rho_{J}^n - 
\bigl( a^n_{J+\tfrac12} \bigr)^{-} \rho_{J+1}^n + 
\bigl( a^n_{J-\tfrac12} \bigr)^{-} \rho_{J}^n - 
\bigl( a^n_{J-\tfrac12} \bigr)^{+} \rho_{J-1}^n 
\Bigr),
\end{equation} 
for $n \in \NN$ and $J \in \ZZ$. It is shown in \cite{CCH} that the potential is non-increasing for the semi-discretized version of this scheme, which is to say that, up to a remainder of order 2 in $\Delta t$ (the value of $\Delta x$ being fixed), the potential
of the fully discretized scheme does not increase from one step to another. The proof of the latter claim follows from a direct expansion of the quantity
\begin{equation*}
\frac12 
\int_{\RR^d}
\int_{\RR^d} W(x-y)
 \rho_{\Delta x}^{n+1}(dx)
 \rho_{\Delta x}^{n+1}(dy)
\end{equation*}
by using the updating rule for $\rho^{n+1}_{J}$ in terms of
 $\rho^{n}_{J}$, $\rho^{n}_{J-1}$ and 
 $\rho^{n}_{J+1}$.

The numerical scheme investigated in this paper does not satisfy the same property. 
Indeed, we provide a counter example, which shows that the potential may increase when $W$ is convex, as a consequence of the numerical diffusion.
However, the same example, but in dimension 1, shows that the scheme \eqref{dis_num_Carrillo} may not be convergent for certain forms of potential for which Theorem \ref{TH} applies, see Subsection \ref{subse:newtonian}.

\begin{proposition}
Choose $d=2$, $W(x)= | x |$ and take $\Delta x_{1} = \Delta x_{2} = 1$. 
Let the initial condition of the scheme, which we just denote by
$\rho^0$, charge the points $0=(0,0)$, $e_{1}=(1,0)$ and $e_{2}=(0,1)$ with
$1-p$, $p/2$ and $p/2$ as respective weights, where $p \in (0,1)$. 

Then, denoting by $\rho^1$ the distribution at time $1$ obtained by implementing the upwind scheme, it holds that:
\begin{equation}
\label{eq:counterexample}
\int_{\RR^2}\int_{\RR^2} \vert x- y \vert 
\rho^1(dx)
\rho^1(dy)
= \int_{\RR^2}\int_{\RR^2} \vert x- y \vert 
\rho^0(dx)
\rho^0(dy)
+ \bigl( \sqrt{2} -1 \bigr) p^2 (2p-1) \Delta t + O(\Delta t^2),
\end{equation}
where the Landau symbol $O(\cdot)$ may depend upon $p$. 
\end{proposition}

Choosing $p>1/2$ in \eqref{eq:counterexample},
we see that the potential may increase at the same rate as the time step. 
\vspace{5pt}

\begin{proof}
We first compute the potential at time $0$.
To do so, 
we compute $\int_{\RR^2} \vert x- y \vert \rho^0(dy)$, for $x \in \{0,e_{1},e_{2}\}$:
\begin{equation*}
\begin{split}
&\int_{\RR^2} \vert y \vert \rho^0(dy)
= p, \quad
\int_{\RR^2} \vert e_{1}- y \vert \rho^0(dy)
=
\int_{\RR^2} \vert e_{2}- y \vert \rho^0(dy)
=
 (1-p) + \frac{p}{\sqrt{2}},
\end{split}
\end{equation*}
so that 
\begin{equation*}
\begin{split}
&\int_{\RR^2}\int_{\RR^2} \vert x- y \vert 
\rho^0(dx)
\rho^0(dy)
= 2 (1-p) p + \frac{p^2}{\sqrt{2}}. 
\end{split}
\end{equation*}

In order to compute the potential at time 1, we compute the velocity at each of the above points. 
Observing that the velocity at point $x$ is given by the formula:
\begin{equation*}
{a_{i}}^0_{x} = \int_{\RR^2} \frac{y_{i} -x_{i} }{\vert y-x \vert}
\rho^0(dy), \qquad i =1,2,
\qquad \textrm{\rm with the convention} \ \frac{0}{0} = 0,
 \end{equation*}
we get:
\begin{alignat*}{2}
&{a_{1}}^{0}_{(0,0)} = \frac{p}{2}, \quad &&\displaystyle {a_{2}}^{0}_{(0,0)} = \frac{p}{2},
\\
&{a_{1}}^{0}_{(1,0)} = - (1-p)  - \frac{p}{2\sqrt{2}}, \quad 
&&\displaystyle  {a_{2}}^{0}_{(1,0)} = \frac{p}{2\sqrt{2}},
\\
&{a_{1}}^{0}_{(0,1)} = \frac{p}{2\sqrt{2}}, \quad 
&&\displaystyle  {a_{2}}^{0}_{(0,1)} = - (1-p) - \frac{p}{2\sqrt{2}}.
\end{alignat*}
We then compute the new masses at time $1$. There is one additional point which is charged: $e_{1}+e_{2}=(1,1)$. We have:
\begin{alignat*}{2}
\rho^{1}(0) &= (1-p) + \frac{p^2}{2\sqrt{2}} \Delta t,
\\
\rho^{1}(e_{1}) = \rho^{1}(e_{2}) &= \frac{p}{2} 
- \frac{p^2}{2 \sqrt{2}} \Delta t,
\\
\rho^{1}(e_{1}+e_{2}) &= \frac{p^2}{2 \sqrt{2}} \Delta t. 
\end{alignat*}
We now have all the required data to compute the potential at time 1. 
\begin{alignat*}{2}
\int_{\RR^2} \vert y \vert \rho^1(dy)
&= p   - \frac{p^2}{\sqrt{2}} \Delta t
+ \frac{p^2}{2} \Delta t,
\\
\int_{\RR^2} \vert e_{1}- y \vert \rho^1(dy)
=
\int_{\RR^2} \vert e_{2}- y \vert \rho^1(dy)
&=
 (1-p)  + \frac{p}{\sqrt{2}}
  + 
 \frac{p^2}{\sqrt{2}} \Delta t 
 - 
 \frac{p^2}{2} \Delta t ,
 \\
 \int_{\RR^2} \vert  e_{1}+e_{2} - y \vert \rho^1(dy)
 & = (1-p) \sqrt{2}  
 + p 
  +  \frac{p^2}{2} \Delta t 
- \frac{p^2}{\sqrt{2}} \Delta t.  
\end{alignat*}
Finally, the potential at time 1 is given by:
\begin{equation*}
\begin{split}
\int_{\RR^2}\int_{\RR^2} \vert x- y \vert 
\rho^{1}(dx)
\rho^{1}(dy)
&= \Bigl(
(1-p) + \frac{p^2}{2 \sqrt{2}} \Delta t
\Bigr) 
\Bigl( 
p - \frac{p^2}{\sqrt{2}} \Delta t 
+ \frac{p^2}{2} \Delta t \Bigr)
\\
&\hspace{5pt} +\Bigl( p - \frac{p^2}{\sqrt{2}} \Delta t
\Bigr) 
\Bigl(
 (1-p)  + \frac{p}{\sqrt{2}}
  + 
 \frac{p^2}{\sqrt{2}} \Delta t   - 
 \frac{p^2}{2} \Delta t 
 \Bigr)
 \\
 &\hspace{5pt} +
 \frac{p^2}{2 \sqrt{2}} \Delta t
\Bigl( (1-p) \sqrt{2} 
 + p 
  +  \frac{p^2}{2} \Delta t 
- \frac{p^2}{\sqrt{2}} \Delta t 
\Bigr).
\end{split}
\end{equation*}
We expand the above right-hand side in powers of $\Delta t$. The zero-order term is exactly equal to 
$\int_{\RR^2} \int_{\RR^2} \vert x-y \vert \rho^0(dx) \rho^0(dy)$. 
So, we just compute the terms in $\Delta t$. It is equal to 
\begin{equation*}
(1- \sqrt{2}) (1-p) p^2 + (\sqrt{2}-1)p^3
= (\sqrt{2}-1) p^2 (2p-1), 
\end{equation*}
which completes the proof. 
\end{proof}

\section{Order of convergence}
\label{sec:ordre}

This section is devoted to the proof of Theorem \ref{TH}.

\subsection{Preliminaries}

Before presenting the proof, we introduce some notations and establish some useful properties. 
We first define the following interpolation weights: for $J\in\ZZ^d$ and 
$y \in \RR^d$, we let
\begin{equation}\label{def:alpha}
\alpha_J(y) = \left\{
\begin{array}{ll}
\ds 1-\sum_{i=1}^d \frac{|\langle y-{x_J},e_i\rangle|}{\Delta x_i} 
& \textrm{when} \ y\in C_J,
\vspace{5pt}
\\
\ds \frac{1}{\Delta x_i}\bigl(\langle y-x_{J-e_i},e_i\rangle\bigr)^+ 
&\textrm{when} \ y\in C_{J-e_i}, \ \ \textrm{for} \ i= 1,\dots,d,
\vspace{5pt}
\\ 
\ds \frac{1}{\Delta x_i}\bigl(\langle y-x_{J+e_i},e_i\rangle\bigr)^- 
&\textrm{when} \ y\in C_{J+e_i}, \ \ \textrm{for} \ i=1,\dots,d,
\vspace{5pt}
\\
0 \quad &\textrm{otherwise}.
\end{array}
\right.
\end{equation}
The terminology \textit{interpolation weights} is justified by the following straightforward observation. 
Given a collection of reals $(h_{J})_{J \in {\mathbb Z}^d}$ indexed by the cells of the mesh, which we may regard as a real-valued function $h : x_{J} \mapsto h_{J}$ defined at the nodes of the mesh, we may define an interpolation  
of $h=(h_{J})_{J \in {\mathbb Z}^d}$ by letting
\begin{equation}
\label{eq:matchalI}
{\mathcal I}(h)(y) = \sum_{J \in {\mathbb Z}^d} h_{J} \alpha_{J}(y), \quad y \in \RR^d. 
\end{equation}
Obviously, the sum in the right-hand side makes sense since only a finite number of weights are non-zero for a given value of $y$. Clearly, the functional ${\mathcal I}$ is an \textit{interpolation operator}. As explained below, ${\mathcal I}$ makes the connection between the analysis we perform in this paper and the one we performed in our previous work \cite{DLV}.

Several crucial facts must be noticed. 
The first one is that, contrary to what one could guess at first sight, the weights are not necessarily non-negative. For a given $J \in \ZZ^d$, take for instance $y=(y_{i}=(J_{i}-\tfrac12) \Delta x_{i})_{i=1,\dots,d}\in C_{J}$. Then $\alpha_{J}(y) = 1 - \tfrac{d}{2}$, which is obviously negative if $d\geq 3$.   
However, 
the second point is that, for \textit{useful values} of $y$, the weights are indeed non-negative provided that the CFL condition \eqref{CFL} is in force.  For 
a given $J \in \ZZ^d$, call indeed $U_{J}$ the subset of $C_{J}$ of so-called \textit{useful values} that are in $C_{J}$, as given by
\begin{equation*}
U_{J} = \bigl\{ y \in \RR^d : \bigl\vert \bigl\langle y - x_{J},e_{i} \rangle \bigr\vert \leq
w_{\infty} \Delta t, \quad  i = 1,\dots,d \bigr\}. 
\end{equation*}
Then, for any $J,L \in \ZZ^d$ and any $y \in U_{L}$, $\alpha_{J}(y)$ is non-negative, which is a direct consequence of the CFL condition \eqref{CFL}.
In fact, the CFL condition \eqref{CFL} says more, and this is the rationale for the additional factor 
$\tfrac12$ in \eqref{CFL}: $U_{J}$ is included in $C_{J}$. 
Of course, the consequence is that, under the CFL condition \eqref{CFL}, we have, for any $J\in\ZZ^d$, 
$x_J+a_J^n\Delta t \in C_J$, where $a^n_{J}$ is the 
$d$-dimensional vector with entries $({a_i}^n_{J})_{i=1,\cdots,d}$
 (indeed $|{a_{i}}_J^n| \Delta t \leq w_\infty \Delta t < \Delta x_i/2$). 
Another key fact is that  
the definition of $\alpha_{J}(y)$ in \eqref{def:alpha}
is closely related to the definition of the numerical scheme 
 \eqref{dis_num}. Indeed, we have the following formula, for any $J,L \in {\mathbb Z}^d$,
 \begin{equation}
 \label{eq:transition:probas}
 \alpha_{J} \bigl( x_{L} + \Delta t a^n_{L}
 \bigr) = 
  \left\{
\begin{array}{ll}
\ds 1-\sum_{i=1}^d \vert {a_{i}}^n_{J} \vert \frac{\Delta t}{\Delta x_i} 
& \textrm{when} \ L=J,
\vspace{5pt}
\\
\ds  \frac{\Delta t}{\Delta x_i}\bigl( {a_{i}}^n_{J-e_{i}} \bigr)^{+}
&\textrm{when} \ L = J-e_{i}, \ \ \textrm{for} \ i= 1,\dots,d,
\vspace{5pt}
\\ 
\ds  \frac{\Delta t}{\Delta x_i}\bigl( {a_{i}}^n_{J+e_{i}} \bigr)^{-}
&\textrm{when} \ L = J+e_{i}, \ \ \textrm{for} \ i= 1,\dots,d,
\vspace{5pt}
\\
0 \quad &\textrm{otherwise}.
\end{array}
\right.
\end{equation} 
In particular,
we may rewrite \eqref{dis_num} as
\begin{equation}\label{scheme3}
\forall\, J \in \ZZ^d, \quad 
\rho^{n+1}_J = \sum_{L\in\ZZ^d} \rho^n_L \alpha_{J}\bigl(x_L+\Delta t a^n_L\bigr),
\end{equation}
which is the core of our analysis below. In this regard, 
The following lemma gathers some useful properties.
\begin{lemma}\label{propalpha}
Let $(\alpha_{L}(y))_{L \in \ZZ^d,y \in \RR^d}$ be defined as in \eqref{def:alpha}.
Then, for any $y \in \RR^d$, we have
$$
\sum_{L\in \ZZ^d} \alpha_{L}(y) = 1 \quad \mbox{ and } \quad \sum_{L\in \ZZ^d} x_L \alpha_{L}(y)  = y.
$$
\end{lemma}
\begin{proof}
There exists a unique $J\in \ZZ^d$ such that $y\in C_J$.
Then, we compute
\begin{align*}
\sum_{L\in \ZZ^d} \alpha_{L}(y) & = \alpha_{J}(y) + \sum_{i=1}^d \bigl(\alpha_{J+e_i}(y)+\alpha_{J-e_i}(y)\bigr) 
\\
& = 1 - \sum_{i=1}^d \frac{|\langle y-{x_L},e_i\rangle|}{\Delta x_i}  
+ \frac{1}{\Delta x_i} \sum_{i=1}^d \bigl(\langle y-x_{J},e_i\rangle \bigr)^+ + \bigl(\langle y-x_{J},e_i\rangle \bigr)^- = 1
\end{align*}
Then, using the fact that $x_{J+e_i}-x_J = \Delta x_i e_i$, for $i=1,\ldots,d$, we have
\begin{align*}
\sum_{L\in \ZZ^d} x_L \alpha_{L}(y) & 
= x_J \alpha_J(y) + \sum_{i=1}^d \bigl(x_{J+e_i} \alpha_{J+e_i}(y)+ x_{J-e_i} \alpha_{J-e_i}(y)\bigr) 
\\
& = x_J + \sum_{i=1}^d \Bigl( \frac{1}{\Delta x_i} \bigl(\langle y-x_J, e_i\rangle \bigr)^+ \Delta x_i e_i 
- \frac{1}{\Delta x_i} \bigl(\langle y-x_J, e_i \bigr\rangle)^- \Delta x_i e_i \Bigr)  \\
& = x_J + \sum_{i=1}^d \langle y-x_J, e_i\rangle e_i = y,
\end{align*}
which completes the proof.
\end{proof}

\begin{remark}
\label{comparison}
Lemma \ref{propalpha} prompts us to draw a comparison with our previous paper \cite{DLV}.
For a given $y \in \RR^d$ in the set of useful values $U:=\cup_{J \in \ZZ^d} U_{J}$, namely $y \in U_{J}$ for some $J \in \ZZ^d$, 
the collection of weights $(\alpha_{L}(y))_{L \in \ZZ^d}$
forms a probability measure, as the weights are non-negative and their sum is 1!
In particular, 
${\mathcal I}(h)(y)$  in \eqref{eq:matchalI}, for $y \in U$, may be interpreted as an expectation. 

Using the same terminology as in \cite{DLV} (which is in fact the terminology of the theory of Markov chains), those weights should be regarded as transition probabilities: For a given $y$ in the set of useful values, 
$\alpha_{L}(y)$ reads as the probability of jumping from a \emph{certain state depending on the sole value of $y$} to 
the node $x_{L}$. Of course, the interpretation of the so-called \emph{certain state depending on the sole value of $y$} is better understood from \eqref{eq:transition:probas}. 
In 
\eqref{eq:transition:probas}, if we fix a cell $L \in \ZZ^d$ (or equivalently a node $x_{L}$), then  
$\alpha_{J}(x_{L}+\Delta t a^n_{L})$ should read as the probability of passing from the node $x_{L}$
to the node $x_{J}$ (or from the cell $L$ to the cell $J$) at the $n^{\textrm{\rm th}}$ step of a (time inhomogeneous) Markov chain having the collection of nodes (or of cells) as state space. 
In this regard, 
\eqref{scheme3} is nothing but the Kolmogorov equation for the corresponding Markov chain, as $(\rho^n_{J})_{J \in \ZZ^d}$ can be interpreted as the law at time $n$ of the Markov chain driven by the latter transition probabilities. 
The reader can easily check that the so-called \emph{stochastic characteristic} used in \cite{DLV} 
is in fact this Markov chain.

Below, we do not make use of the Markov chain explicitly. Still, we use the weights 
$(\alpha_{J}(y))_{J \in \ZZ^d, y \in \RR^d}$ to construct a coupling between the two measures 
$\rho^n_{\Delta x}$ and $\rho^{n+1}_{\Delta x}$, that is to construct a specific element of $\Gamma(\rho^n_{\Delta x},\rho^{n+1}_{\Delta x})$. In \cite{DLV}, 
this coupling does not explicitly show up but it is in fact implicitly used, as it coincides with the joint law of two consecutive states of the aforementioned Markov chain.

In a nutshell, the reader can reformulate the whole analysis below in a probabilistic fashion. The only (conceptual) difficulty to do so 
is that, in contrast with \cite{DLV}, the Markov chain is here \emph{nonlinear}: 
as $a^n$ in 
\eqref{def:aij} depends on $\rho^n$, the transition probabilities of the Markov do depend upon the marginal law of the Markov chain itself, which fact gives rise to a so-called \emph{nonlinear Markov chain}!
\end{remark}

\subsection{Proof of Theorem \ref{TH}}

{\bf 1st step.} We first consider the case where the initial datum is given by $\rho^{ini}:=\rho_{\Delta x}^0 = \sum_{J\in\ZZ^d} \rho_J^0 \delta_{x_J}$, where we recall that $\rho_J^0$ is defined in \eqref{disrho0}.
For $n\in\NN^*$, let us define 
$$
D_n:= d_W \bigl(\rho(t^n),\rho_{\Delta x}^n \bigr).
$$
Clearly, with our choice of initial datum, we have $D_0=0$.

Let $\gamma$ be an optimal plan in $\Gamma_0(\rho(t^n),\rho_{\Delta x}^n)$, we have
\[
D_n = \left(\iint_{\RR^d\times\RR^d} |x-y|^2 \gamma(dx,dy)\right)^{1/2}.
\]

Let us introduce $a^n_{\Delta x}$, the piecewise affine in each direction reconstruction of the velocity such that for all 
$J\in\ZZ^d$, $a^n_{\Delta x}(x_J)=a_J^n$ 
Denote also by 
$Z := 
Z_{\rho}$
the flow given by Theorem 
\ref{Exist}, when $\rho^{ini}$ is prescribed as above. 
Recalling the definition 
of $\alpha_J(y)$ from \eqref{def:alpha}, 
we then consider a new measure $\gamma'$, defined as
the image of $\gamma$ by the 
kernel ${\mathcal K}$ that associates with a point $(x,y) \in \RR^d \times \RR^d$
the point $(Z(t^{n+1};t^n,x),x_{L})$ with measure $\alpha_{L}(y+\Delta t  a^n_{\Delta x}(y))$, namely, 
for any 
 two Borel subsets $A$ and $B$ of $\RR^d$,
\begin{equation*}
\begin{split}
{\mathcal K} \bigl( (x,y), A \times B \bigr) &= {\mathbf 1}_{A}\bigl( Z(t^{n+1};t^n,x) \bigr) 
\sum_{L \in {\ZZ}^d} \alpha_{L}\bigl(y+\Delta t  a^n_{\Delta x}(y) \bigr)
{\mathbf 1}_{B}(x_{L}) 
\\
&= \iint_{\RR^d\times \RR^d} {\mathbf 1}_{A \times B}(x',y') 
\biggl[ \delta_{Z(t^{n+1};t^n,x)} \otimes \biggl( \sum_{L \in \ZZ^d}
\alpha_{L}\bigl(y+\Delta t  a^n_{\Delta x}(y) \bigr)
\delta_{x_{L}}\biggr) \biggr] (dx',dy'),
\end{split}
\end{equation*}
where $\delta_{z}$ denotes the Dirac mass
at point $z$,
and then
\begin{equation*}
\gamma'(A \times B) = \iint_{\RR^d \times \RR^d} {\mathcal K} \bigl( (x,y),A \times B \bigr) 
\gamma(dx,dy). 
\end{equation*} 
Equivalently, for any bounded Borel-measurable function $\theta : \RR^d \times \RR^d \rightarrow \RR$,
\begin{equation}\label{def:gamma'}
\iint_{\RR^d\times\RR^d} \theta(x,y) \gamma'(dx,dy) = \iint_{\RR^d\times\RR^d}
\biggl[\sum_{L\in \ZZ^d} \theta\bigl(Z(t^{n+1};t^n,x),x_L\bigr) 
\alpha_L\bigl(y+\Delta t a^n_{\Delta x}(y)\bigr)\biggr]\,\gamma(dx,dy).
\end{equation}
Then we have $\gamma'\in\Gamma(\rho(t^{n+1}),\rho_{\Delta}^{n+1})$.
Indeed, for any bounded Borel-measurable function $\theta_{1}: \RR^d \rightarrow \RR$, we have, from \eqref{def:gamma'} and Lemma \ref{propalpha},
\begin{align*}
\iint_{\RR^d\times\RR^d} \theta_{1}(x) \gamma'(dx,dy) 
& = \iint_{\RR^d\times\RR^d}
\biggl[\sum_{L\in \ZZ^d} \theta_{1}\bigl(Z(t^{n+1};t^n,x)\bigr) \alpha_L\bigl(y+\Delta t a^n_{\Delta x}(y)\bigr)\biggr]\,\gamma(dx,dy)  \\
& = \iint_{\RR^d\times\RR^d} \theta_{1}\bigl(Z(t^{n+1};t^n,x)\bigr)\,\gamma(dx,dy)  \\
& = \int_{\RR^d} \theta_{1}\bigl(Z(t^{n+1};t^n,x)\bigr) \rho(t^n,dx) = \int_{\RR^d} \theta_{1}(x) \rho(t^{n+1},dx),
\end{align*}
where we used Theorem \ref{Exist} and
where $\rho(t^n,dx)$ is a shorter notation for $\rho(t^n)(dx)$ and similarly 
for $\rho(t^{n+1},dx)$. 
Similarly, for any bounded Borel-measurable function $\theta_{2} : \RR^d \rightarrow \RR$,
\begin{align*}
\iint_{\RR^d\times\RR^d} \theta_{2}(y) \gamma'(dx,dy) 
& = \iint_{\RR^d\times\RR^d}
\biggl[\sum_{L\in \ZZ^d} \theta_{2}(x_L) \alpha_L\bigl(y+\Delta t a^n_{\Delta x}(y)\bigr)\biggr]\,\gamma(dx,dy)  \\
& = \sum_{J\in \ZZ^d}\sum_{L\in \ZZ^d} \theta_{2}(x_L) \alpha_L\bigl(x_J+\Delta t a_J^n\bigr) \rho_J^n  \\
& = \sum_{L\in \ZZ^d} \theta_{2}(x_L) \rho_L^{n+1} = \int_{\RR^d} \theta_{2}(y) \rho^{n+1}_{\Delta x}(dy),
\end{align*}
where we used \eqref{scheme3}.
In particular, we deduce
$$
D_{n+1}^2 \leq \iint_{\RR^d\times\RR^d} |x-y|^2 \gamma'(dx,dy).
$$
Using the definition of $\gamma'$ given in \eqref{def:gamma'}, we get
\begin{equation}\label{eq1}
D_{n+1}^2 \leq \iint_{\RR^d\times\RR^d} \sum_{L\in \ZZ^d} \bigl|Z(t^{n+1};t^n,x)-x_L\bigr|^2 \alpha_L \bigl(y+\Delta t a^n_{\Delta x}(y)\bigr) \gamma(dx,dy).
\end{equation}
Using both equalities of Lemma \ref{propalpha}, we compute\footnote{The probabilistic reader will easily recognize the standard computation of the $L^2$ norm of a random variable in terms of its variance and its expectation, which indeed plays, but under a conditional form, a key role in 
\cite{DLV}.}
\begin{align}
&\sum_{L\in \ZZ^d}\bigl|Z(t^{n+1};t^n,x) - x_L\bigr|^2 \alpha_L\bigl(y+\Delta t a^n_{\Delta x}(y)\bigr) \notag
\\
&=
\sum_{L\in \ZZ^d}\Bigl| \Bigl( Z(t^{n+1};t^n,x) - \bigl( y 
+\Delta t a^n_{\Delta x}(y)\bigr)
\Bigr) 
 - \Bigl( x_L
 - \bigl( y +
 \Delta t a^n_{\Delta x}(y)\bigr)
 \Bigr)
 \Bigr|^2 \alpha_L\bigl(y+\Delta t a^n_{\Delta x}(y)\bigr) \notag
\\
&=
 \bigl|Z(t^{n+1};t^n,x) - y - \Delta t a^n_{\Delta x}(y) \bigr|^2 
+ \sum_{L\in \ZZ^d} \bigl|x_L
- 
y - \Delta t a^n_{\Delta x}(y)
\bigr|^2 \alpha_L\bigl(y+\Delta t a^n_{\Delta x}(y)\bigr)
\nonumber\\
&\hspace{15pt} 
-2\biggl\langle Z(t^{n+1};t^n,x) - y - \Delta t  a^n_{\Delta x}(y),
\sum_{L \in \ZZ^d}
\bigl(x_L-y-\Delta t a^n_{\Delta x}(y)\bigr) \alpha_L\bigl(y+\Delta t a^n_{\Delta x}(y)
\biggr\rangle .  \label{eqZz}
\end{align}
Now, as a consequence of Lemma 
\ref{propalpha}, we observe that
$$
\sum_{L\in \ZZ^d} \bigl(x_L-y-\Delta t a^n_{\Delta x}(y)\bigr) \alpha_L\bigl(y+\Delta t a^n_{\Delta x}(y)\bigr) = 0. 
$$
Thus, equation \eqref{eqZz} rewrites
\begin{align*}
\sum_{L\in \ZZ^d}\bigl|Z(t^{n+1};t^n,x) - x_L\bigr|^2 \alpha_L\bigl(y+\Delta t a^n_{\Delta x}(y)\bigr) 
&=  \bigl|Z(t^{n+1};t^n,x)-y-\Delta t a^n_{\Delta x}(y)\bigr|^2 
 \\
&\hspace{15pt}+ \sum_{L\in \ZZ^d} \bigl|x_L-y-\Delta t a^n_{\Delta x}(y)\bigr|^2 \alpha_L\bigl(y+\Delta t a^n_{\Delta x}(y)\bigr).
\end{align*}
Injecting into \eqref{eq1}, we deduce
\begin{align}
D_{n+1}^2 \leq  
&\ \iint_{\RR^d\times\RR^d} \bigl|Z(t^{n+1};t^n,x)-y-\Delta t a^n_{\Delta x}(y)\bigr|^2 \gamma(dx,dy)  \nonumber \\
& + \int_{\RR^d} \sum_{L\in \ZZ^d} \bigl|x_L-y-\Delta t a^n_{\Delta x}(y) \bigr|^2 \alpha_L\bigl(y+\Delta t a^n_{\Delta x}(y)\bigr) \rho_{\Delta x}^n(dy), \label{eq2}
\end{align}
where we used the fact that $\rho_{\Delta x}^n$ is the second marginal of $\gamma$. 
By definition, $\rho_{\Delta x}^n(y) = \sum_{J\in\ZZ^d} \rho_J^n \delta_J(y)$, so that 
\begin{align*}
&\sum_{L\in \ZZ^d} \int_{\RR^d} \bigl|x_L-y-\Delta t a^n_{\Delta x}(y)\bigr|^2 \alpha_L\bigl(y+\Delta t a^n_{\Delta x}(y)\bigr) \rho_{\Delta x}^n(dy)  \\
&\hspace{5cm} = \sum_{J\in \ZZ^d} \sum_{L\in \ZZ^d} \bigl|x_L-x_J-\Delta t a^n_J\bigr|^2 \alpha_L\bigl(x_J+\Delta t a^n_J\bigr) \rho_J^n.
\end{align*}
Moreover using the definition of $\alpha_L$ in \eqref{def:alpha}, we compute
\begin{align*}
&\sum_{L\in \ZZ^d} \bigl|x_L-x_J-\Delta t a^n_J\bigr|^2  \alpha_L\bigl(x_J+\Delta t a^n_J\bigr) 
\\
&= \Delta t^2 |a_J^n|^2 \left(1-\sum_{i=1}^d \frac{\Delta t}{\Delta x_i} |{a_i}_J^n|\right)  +  \sum_{i=1}^d \left( \bigl|\Delta x_i e_i -\Delta t a^n_J\bigr|^2 \frac{\Delta t}{\Delta x_i} ({a_i}_J^n)^+ + 
\bigl|\Delta x_i e_i +\Delta t a^n_J\bigr|^2 \frac{\Delta t}{\Delta x_i} ({a_i}_J^n)^- \right) \\
& \leq  C\Delta t(\Delta t + \Delta x),
\end{align*}
where we used, for the last inequality, the CFL condition \eqref{CFL} and the fact that the velocity $(a_J^n)_J$ is uniformly bounded (see Lemma \ref{lem:CFL} or Lemma \ref{lem:CFL:lambda:>0}).
Then, \eqref{eq2} gives
\begin{equation}\label{eqD1}
D_{n+1}^2 \leq \iint_{\RR^d\times\RR^d} \bigl|Z(t^{n+1};t^n,x)-y-\Delta t a^n_{\Delta x}(y)\bigr|^2 \gamma(dx,dy)  + C \Delta t (\Delta t+\Delta x).
\end{equation}

{\bf 2nd step.}
We have to estimate the error between the exact characteristic $Z(t^{n+1};t^n,x)$ and the forward Euler discretization $y+\Delta t a^n_{\Delta x}(y)$.
By definition of the characteristics \eqref{eq:characteristics}, we have
\begin{align*}
Z(t^{n+1};t^n,x) & = x + \int_{t^n}^{t^{n+1}} \achapo_\rho\bigl(s,Z(s;t^n,x)\bigr) ds 
\\
& = x - \int_{t^n}^{t^{n+1}} \int_{\RR^d}\nabWchapo\bigl(Z(s;t^n,x)-Z(s;t^n,\xi)\bigr) \rho(t^n,d\xi) ds.
\end{align*}
We recall also that, by definition \eqref{def:aij}, the approximating velocity is given by
\begin{align*}
a_L^n = - \sum_{J\in\ZZ^d} \rho_J^n \nabWchapo(x_L-x_J),
\end{align*}so that for $y$, a node of the mesh, 
\begin{equation*}
y + \Delta t a^n_{\Delta x}(y) 
= y - \Delta t \int_{\RR^d} 
\nabWchapo( y - \zeta) \rho^n_{\Delta x}( d \zeta \bigr). 
\end{equation*}
Thus, by a straightforward expansion and still for $y$ a node of the mesh, 
\begin{align*}
&\bigl|Z(t^{n+1};t^n,x)-y-\Delta t a^n_{\Delta x}(y)\bigr|^2 \leq  |x-y|^2  \\
& - 2\int_{t^n}^{t^{n+1}}\iint_{\RR^d\times \RR^d} \bigl\langle x-y, \nabWchapo\bigl(Z(s;t^n,x)-Z(s;t^n,\xi)\bigr)-\nabWchapo(y-\zeta)\bigr\rangle \rho(t^n,d\xi) \rho_{\Delta x}^n(d\zeta) + C\Delta t^2.
\end{align*}
By definition of the optimal plan $\gamma\in \Gamma_0(\rho(t^n),\rho^n_{\Delta x})$, we also have
\begin{align*}
& \iint_{\RR^d\times \RR^d} \bigl\langle x-y, \nabWchapo\bigl(Z(s;t^n,x)-Z(s;t^n,\xi)\bigr)-\nabWchapo(y-\zeta)\bigr\rangle \rho(t^n,d\xi) \rho_{\Delta x}^n(d\zeta)  \\
& = 
\iint_{\RR^d\times \RR^d} \bigl\langle x-y, \nabWchapo\bigl(Z(s;t^n,x)-Z(s;t^n,\xi)\bigr)-\nabWchapo(y-\zeta)\bigr\rangle \gamma(d\xi,d\zeta)
\end{align*}
Injecting into \eqref{eqD1}, we get
\begin{align*}
D_{n+1}^2 \leq & \ D_n^2 + C \Delta t(\Delta t+\Delta x) \\
& - 2 \int_{t^n}^{t^{n+1}}\iint_{\RR^d\times \RR^d}\iint_{\RR^d\times \RR^d} \bigl\langle x-y, \nabWchapo\bigl(Z(s;t^n,x)-Z(s;t^n,\xi)\bigr)-\nabWchapo(y-\zeta) \bigr\rangle  \\
& \hspace{10cm} \gamma(d\xi,d\zeta) \gamma(dx,dy).
\end{align*}
Decomposing $x-y=x-Z(s;t^n,x)+Z(s;t^n,x)-y$ and using the fact that 
$|Z(s;t^n,x)-x|\leq w_\infty |s-t^n|$, we get
\begin{align*}
D_{n+1}^2 \leq & \ D_n^2 + C \Delta t(\Delta t+\Delta x) \\
& - 2 \int_{t^n}^{t^{n+1}}\iint_{\RR^d\times \RR^d}\iint_{\RR^d\times \RR^d} \bigl\langle Z(s;t^n,x)-y, \nabWchapo\bigl(Z(s;t^n,x)-Z(s;t^n,\xi)\bigr)-\nabWchapo(y-\zeta)\bigr\rangle  \\[-2mm]
& \hspace{11.7cm} \gamma(d\xi,d\zeta) \gamma(dx,dy).
\end{align*}
Then, we may use the symmetry of
the potential $W$ in assumption {\bf (A0)} for the last term to deduce
\begin{align*}
D_{n+1}^2 \leq & \ D_n^2 + C \Delta t(\Delta t+\Delta x) \\
& - \int_{t^n}^{t^{n+1}}\iint_{\RR^d\times \RR^d}\iint_{\RR^d\times \RR^d} \bigl\langle Z(s;t^n,x)-Z(s;t^n,\xi)-y+\zeta,  \\[-2mm]
& \hspace{4.7cm}  \nabWchapo\bigl(Z(s;t^n,x)-Z(s;t^n,\xi)\bigr)-\nabWchapo(y-\zeta)\bigr\rangle
\, \gamma(d\xi,d\zeta) \gamma(dx,dy).
\end{align*}
Moreover, from the $\lambda$-convexity of $W$ \eqref{lambdaconvWchapo}, we obtain
\begin{align*}
D_{n+1}^2 \leq & D_n^2 + C \Delta t(\Delta t+\Delta x) \\
& - \lambda \int_{t^n}^{t^{n+1}}\iint_{\RR^d\times \RR^d}\iint_{\RR^d\times \RR^d} \bigl|Z(s;t^n,x)-y-Z(s;t^n,\xi)+\zeta\bigr|^2
\, \gamma(d\xi,d\zeta) \gamma(dx,dy).
\end{align*}
Expanding the last term, we deduce
\begin{align}
D_{n+1}^2 \leq & \ D_n^2 + C \Delta t(\Delta t+\Delta x) 
- 2\lambda \int_{t^n}^{t^{n+1}} \iint_{\RR^d\times \RR^d} \bigl|Z(s;t^n,x)-y\bigr|^2
\,\gamma(dx,dy)  \nonumber \\
& + 2\lambda \int_{t^n}^{t^{n+1}} \biggl\vert \iint_{\RR^d\times \RR^d} \bigl(Z(s;t^n,x)-y\bigr)
\,\gamma(dx,dy) \biggr\vert^2.  \label{interm2}
\end{align}

{\bf 3rd step.} 
Now we distinguish between the two cases $\lambda\leq 0$ and $\lambda>0$.

(i) Starting with the case $\lambda\leq 0$, we have that the last term in \eqref{interm2} is nonpositive.
Using Young's inequality and the estimate $|x-Z(s;t^n,x)| \leq w_\infty (s-t^n)$, we get, for any $\varepsilon>0$,
$$
\bigl|Z(s;t^n,x)-y \bigr|^2 \leq (1+\varepsilon) |x-y|^2 + (1+\frac{1}{\varepsilon}) w_\infty^2 |s-t^n|^2.
$$
Hence, injecting into \eqref{interm2}, we deduce
$$
D_{n+1}^2 \leq \bigl(1 + 2(1+\varepsilon)|\lambda| \Delta t \bigr) D_n^2 + C \Delta t\Big(\Delta x+\Delta t(1+\frac{\Delta t}{\varepsilon})\Big).
$$
Applying a discrete Gronwall inequality, we obtain
$$
D_n^2 \leq e^{2 (1+\varepsilon)|\lambda| t^n}\left(D_0^2+ C t^n\Big(\Delta x+\Delta t(1+\frac{\Delta t}{\varepsilon})\Big)\right).
$$
We recall that our choice of initial data implies $D_0=0$.
Finally, taking $\varepsilon=\Delta t$, we conclude
$$
d_W\bigl(\rho(t^n),\rho_{\Delta}^n\bigr) \leq C e^{(1+\Delta t)|\lambda| t^n} \sqrt{t^n(\Delta x+\Delta t)}.
$$
It allows to conclude the proof of Theorem \ref{TH} (i) in the case $\rho^{ini}=\rho_{\Delta x}^0$.

(ii) Considering now the case $\lambda>0$, we have
$$
\iint_{\RR^d\times \RR^d} \bigl(Z(s;t^n,x)-y \bigr)\,\gamma(dx,dy) = 
\int_{\RR^d} \bigl(Z(s;t^n,x)-x\bigr)\rho(t^n,dx)
+ \int_{\RR^d} x \rho(t^n,dx) - \sum_{J\in\ZZ^d} x_J \rho_J^n.
$$
By conservation of the center of mass, see Lemma \ref{bounddismom} (i), we deduce that 
$$
\int_{\RR^d} x \rho(t^n,dx) - \sum_{J\in\ZZ^d} x_J \rho_J^n = 
\int_{\RR^d} x \rho^{ini}(dx) - \sum_{J\in\ZZ^d} x_J \rho_J^0 = 0,
$$
since we have chosen the initial data such that $\rho^{ini}=\rho_{\Delta x}^0$.
Using also the bound $|Z(s;t^n,x)-x| \leq w_\infty (s-t^n)$, we may bound the last term of \eqref{interm2} by $w_\infty^2 \Delta t^2$.
Moreover, using again Young's inequality and the estimate $|Z(s;t^n,x)-x| \leq w_\infty (s-t^n)$, we have, for any $\varepsilon>0$,
$$
|x - y|^2 \leq (1+\varepsilon) \bigl|Z(s;t^n,x)-y \bigr|^2 + (1+\frac{1}{\varepsilon}) w_\infty^2 |s-t^n|^2.
$$
It implies, for any $\varepsilon\in (0,1)$,
\begin{align*}
- \bigl|Z(s;t^n,x)-y \bigr|^2 & \leq - \frac{1}{1+\varepsilon} |x-y|^2 + \frac{1}{\varepsilon} w_\infty^2 |s-t^n|^2  \\
& \leq -(1-\varepsilon) |x-y|^2 + \frac{1}{\varepsilon} w_\infty^2 |s-t^n|^2.
\end{align*}
Thus we deduce that
$$
- 2 \lambda \int_{t^n}^{t^{n+1}}\!\! \iint_{\RR^d\times\RR^d} \bigl|Z(s;t^n,x)-y \bigr|^2 \gamma(dx,dy)
\leq -2 \lambda (1-\varepsilon) \Delta t D_n
+ \frac 23 \frac{\lambda}{\varepsilon} w_\infty^2 \Delta t^3.
$$
Injecting this latter inequality into \eqref{interm2} and taking $\varepsilon=\Delta t$, we deduce
$$
D_{n+1}^2 \leq \bigl(1-2\lambda (1-\Delta t)\Delta t \bigr) D_n^2 + C \Delta t(\Delta t+\Delta x) 
$$
Hence, since $2\lambda (1-\Delta t)\Delta t<1$, we have by induction, recalling that $D_0=0$,
$$
D_n^2 \leq C \Delta t(\Delta t+\Delta x) \sum_{k=0}^{n-1} \bigl(1-2\lambda (1-\Delta t)\Delta t
\bigr)^{k} \leq \frac{C}{2(1-\Delta t)\lambda} (\Delta t+\Delta x).
$$
Using the assumption $\Delta t \leq 1/2$, we conclude the proof of Theorem \ref{TH} (ii) in the case $\rho^{ini}=\rho_{\Delta x}^0$.
\vskip 2pt

{\bf 4th step.}
We are left with the case $\rho^{ini}\neq \rho_{\Delta x}^0$. 
Let us define $\rho'(t)=Z'(t)_\#\rho_{\Delta x}^0$, the exact solution with
initial data $\rho_{\Delta x}^0$. From the triangle inequality, we have
$$
d_W\bigl(\rho(t^n),\rho_{\Delta x}^n\bigr) \leq d_W\bigl(\rho(t^n),\rho'(t^n)\bigr) + 
d_W\bigl(\rho'(t^n),\rho_{\Delta x}^n\bigr).
$$
The last term in the right hand side may be estimated thanks to the above computations. 
For the first term in the right hand side, we use the estimates in Theorem \ref{Exist} (we apply $(i)$ if $\lambda \leq 0$ and $(ii)$ if $\lambda >0$):
\[
d_W\bigl(\rho(t^n),\rho'(t^n)\bigr) \leq 
e^{(\lambda)^{-} t^n} d_W\bigl(\rho^{ini},\rho_{\Delta x}^0\bigr),
\]
where $(\lambda)^{-}=\max(-\lambda,0)$ is the negative part of $\lambda$. 

Let us define $\tau:[0,1]\times \RR^d \to \RR^d$ by
$\tau(\sigma,x) = \sigma x_J + (1-\sigma)x$, for $x \in C_J$.
We have that $\tau(0,\cdot)=\mathrm{id}$ and $\tau(1,\cdot)_\# \rho^{ini} = \rho_{\Delta x}^0$. Then 
\begin{equation}
\label{eq:wp:condition:initiale}
\begin{split}
d_W\bigl(\rho^{ini},\rho_{\Delta x}^0\bigr)^2 &\leq  \ds 
\int_{\RR^d\times \RR^d} |x-y|^2 \, \bigl[(\mathrm{id}\times\tau(1,\cdot))_\# \rho^{ini}\bigr](dx,dy)  
\\
& \leq \ds \sum_{J\in \ZZ^d} \rho_J^0 \int_{C_J} |x-x_J|^2 \,\rho^{ini}(dx).
\end{split}
\end{equation}
We deduce $d_W(\rho^{ini},\rho_{\Delta x}^0) \leq \Delta x$.
Then, we get
$$
d_W\bigl(\rho(t^n),\rho'(t^n)\bigr) \leq e^{(\lambda)^- t^n} \Delta x.
$$

\section{Unstructured mesh}
\label{sec:unstruct}

We can extend our convergence result to more general meshes.
For the sake of simplicity
of the notation, we present the case of a triangular mesh in two dimensions.
This approach can be easily extended to meshes made of simplices, in any dimension.

\subsection{Forward semi-Lagrangian scheme}

Let us consider a triangular mesh $\calT = (T_k)_{k\in \ZZ}$ with nodes $(x_i)_{i\in \ZZ}$. We assume this mesh to be conformal: A summit cannot belong to an open edge of the grid. 
The triangles $(T_k)_{k \in \ZZ}$ are assumed to satisfy $\bigcup_{k\in\ZZ} T_k = \RR^2$ and $T_k \cap T_l = \emptyset$ if $k \neq l$ (in particular, the cells are here not assumed to be closed nor open). For any triangle $T$ with summits $x$, $y$, $z$, we will use also the notation
$(x,y,z) = T$. We denote by $\calV(T) = \calV(x,y,z)$ the area of this
triangle, and $h(T)$ its height (defined as the minimum of the three heights of the triangle $T$).
We make the assumption that the mesh satisfies $\hbar:=\inf_{k\in \ZZ} h(T_k) >0$.

For any node $x_i$, $i \in \ZZ$, we denote by $K(i)$ the set of 
indices indexing triangles that have $x_i$ as a summit, and we denote by $\calT_i$ the set of all 
triangles of $\calT$ that have $x_i$ as a summit: thus $\calT_i = \{T_k ; k \in K(i) \}$.

For any triangle $T_k$, $k \in \ZZ$, we denote by 
\[
I(k) = \{I_1(k), I_2(k), I_3(k)\}
\]
the set of indices indexing the summits of $T_k$
(for some arbitrary order, whose choice has no importance for the sequel).

We consider the following scheme, which may be seen as a forward semi-Lagrangian scheme on the triangular mesh.
\begin{itemize}
\item For an initial distribution $\rho^{ini}$
of the PDE \eqref{EqInter}, define
the probability weights $(\rho^0_{i})_{i \in \ZZ}$
through the following procedure: 
Consider the one-to-one mapping $\iota : 
\ZZ \ni k \mapsto \iota(k) \in \ZZ$ such that, for 
each $k \in \ZZ$, $x_{\iota(k)}$ is a node of the triangle $T_k$; 
$\iota$ is thus a way to associate a node 
with a cell; then,
for all $i \in \ZZ$, 
let
$\rho^0_{i} = \sum_{k:\iota(k) = i}\rho^{ini}(T_{k})$.
Observe from 
\eqref{eq:wp:condition:initiale}
 that 
$\rho^0_{\Delta x} = \sum_{j \in \ZZ} \rho^0_{j} \delta_{x_{j}}$ is an approximation of $\rho^{ini}$. 
\item Assume that, for a given $n \in \NN$,  
we already have 
probability weights  
$(\rho_i^n)_{i\in\ZZ}$ 
such that 
$\rho^n_{\Delta x} =
\sum_{j \in \ZZ} \rho^n_{j} \delta_{x_{j}}$
is an approximation of $\rho(t^n,\cdot)$, 
where $\rho$ is the solution to 
\eqref{EqInter}
with 
$\rho^{ini}$
as initial condition.
For $i \in \ZZ$, 
we let
\[
a_i^n:=-\int_{\RR^d} \nabWchapo(x_i-y)\,\rho_{\Delta x}^n(dy), \qquad \mbox{and} \qquad
y_i^n:= x_i + a_i^n \Delta t.
\]
Under the CFL-like condition
\begin{equation}\label{CFLT}
w_\infty \Delta t \leq \hbar, 
\end{equation}
$y_i^n$ belongs to one (and only one) of the elements of $\calT_i$.
We denote by $k_{i}^n$ the index of this triangle, namely 
$y_i^n\in T_{k_{i}^n}$.
\item 
We use a linear splitting rule between the summits of 
the triangle $T_{k_i^n}$: the mass $\rho_i^n$ is sent to the three points 
$x_{I_1(k_i^n)}$, $x_{I_2(k_i^n)}$, $x_{I_3(k_i^n)}$ according to the 
{\em barycentric coordinates} of $y_i^n$ in the triangle. 
\end{itemize}

\begin{center}
\begin{tikzpicture}
\draw (0,0) -- (5,0) -- (3,4) -- (0,0);
\draw (0,0) node[below]{$x_i =x_{I_1(k_{i}^n)}$};
\draw (5,0) node[below]{$x_{I_2(k_{i}^n)}$};
\draw (3,4) node[above]{$x_{I_3(k_{i}^n)}$};
\draw (2.5,2) node[below]{$y_i^n$};
\draw [->] (0,0) -- (2.5,2);
\draw [dashed] (5,0) -- (2.5,2) -- (3,4);
\end{tikzpicture}
\end{center}

Let us make more precise the latter point. Let $T = (x,y,z) \in \calT$, and $\xi \in T$. 
We define the barycentric coordinates of $\xi$ with respect 
to $x$, $y$ and $z$, $\lambda_{x}^{T}$, 
$\lambda_{y}^{T}$ and $\lambda_{z}^{T}$: 
\begin{equation}\label{lambdaetal}
\lambda_{x}^{T}(\xi) = \frac{\calV(\xi,y,z)}{\calV(T)}, \quad 
\lambda_{y}^{T}(\xi) = \frac{\calV(\xi,x,z)}{\calV(T)}, \quad 
\lambda_{z}^{T}(\xi) = \frac{\calV(\xi,x,y)}{\calV(T)}, \quad 
\end{equation}
and then have $\xi = \lambda_{x}^{T}(\xi) x + \lambda_{y}^{T}(\xi) y 
+ \lambda_{z}^{T}(\xi) z$. 
Note also that $\lambda_{x}^{T}(\xi) + \lambda_{y}^{T}(\xi) 
+ \lambda_{z}^{T}(\xi) = 1$. Therefore, we have the following fundamental 
property, which will be used in the sequel: 
\begin{equation}
\label{fundapropo}
\lambda_{x}^{T}(\xi) (x - \zeta) + 
\lambda_{y}^{T}(\xi) (y - \zeta)
+ \lambda_{z}^{T}(\xi) (z - \zeta) = \xi - \zeta,
\end{equation}
for any $\zeta \in \RR^2$. 

In the same spirit as in Section \ref{sec:ordre}, we here define the interpolation weights by: For $j\in\ZZ$, and $y\in\RR^2$,
\begin{equation}
\label{def:alphaT}
\alpha_j(y) :=
\left\{\begin{array}{ll}
\lambda_{x_j}^{T}(y),  & \textrm{when} \ y\in T,
\vspace{5pt}
\\
0, \ \ & \textrm{otherwise.}
\end{array}\right.
\end{equation}
Then, the numerical scheme reads
\begin{equation}\label{schemeT}
\rho_j^{n+1} = \sum_{i\in \ZZ} \rho_i^n \alpha_{j}(x_i+a_i^n\Delta t), 
\qquad j \in \ZZ, \ n \in \NN.
\end{equation}
We easily
verify  from \eqref{lambdaetal} and \eqref{fundapropo} that the interpolation weights satisfy:
\begin{lemma}\label{propalphaT}
Let $(\alpha_{j}(y))_{j \in \ZZ,y \in \RR^2}$ be defined as in \eqref{def:alphaT}.
Then, for any $j \in \ZZ$ and $y \in \RR^2$, $\alpha_{j}(y) \geq 0$. Moreover, 
for any $y \in \RR^2$, 
$$
\sum_{j\in\ZZ} \alpha_j(y) = 1, \qquad
\sum_{j\in\ZZ} x_j \alpha_j(y) = y.
$$
\end{lemma}

\subsection{Convergence result}

By the same token as in Section \ref{sec:ordre}, we can use Lemma \ref{propalphaT} and Theorem \ref{Exist}
to prove that the numerical scheme \eqref{schemeT} is of order $1/2$:
\begin{theorem}
\label{TH2}
Assume that $W$ satisfies hypotheses {\bf (A0)--(A3)}.
For $\rho^{ini} \in \calP_2(\RR^d)$, let $(\rho(t))_{t \ge 0}$ be the unique measure solution to the aggregation equation with initial data $\rho^{ini}$, as
given by Theorem \ref{Exist}.
Let us also consider a triangular conformal mesh $(T_k)_{k\in \ZZ}$ with nodes $(x_j)_{j\in \ZZ}$ 
such that $\hbar = \inf_{k\in \ZZ} h(T_k) >0$
and the CFL condition \eqref{CFLT} holds true. 
We denote by 
$\Delta x$ the longest edge in the mesh. 

Define $((\rho_j^n)_{j\in \ZZ})_{n \in {\mathbb N}}$
as in 
\eqref{schemeT}
and let
$$
\rho_{\Delta x}^n := \sum_{j\in \ZZ} \rho_j^n \delta_{x_j}, 
\quad n \in {\mathbb N}. 
$$
Then, there exists a nonnegative constant $C$, 
independent of the discretization parameters, such that, for all $n\in \NN^*$,
$$
d_W(\rho(t^n),\rho_{\Delta x}^n) \leq C e^{|\lambda|(1+\Delta t) t^n} \bigl( \sqrt{t^n \Delta x} + \Delta x \bigr).
$$
\end{theorem}

Importantly, we do not claim that 
$(ii)$ in the statement of Theorem \ref{TH} remains true in the framework of Theorem \ref{TH2}. Indeed, it would require to prove that the support of the numerical solution remains included in a ball when the support of the initial condition is bounded. As made clear by 
the proof of Lemma 
\ref{lem:CFL:lambda:>0}, this latter fact depends on the geometry of the mesh.

\section{Numerical illustrations}
\label{sec:sim}

We now address several numerical examples. In 
Subsection 
\ref{subsec:optimality}, 
we show that the rate of convergence established in 
Theorem \ref{TH} is optimal in a one-dimensional 
example. This prompts us to start with a short reminder 
on the Wasserstein distance in dimension $d=1$. 
In Subsection
\ref{subse:newtonian}, we provide several numerical examples in 
dimension $d=1$ for the Newtonian potential, whilst examples in 
dimension $d=2$ are handled in Subsection 
\ref{subsec:numerical:d2}.

\subsection{Wasserstein distance in one dimension}
The numerical computation of the Wasserstein distance between two probablity measures 
in any dimension is generally quite difficult. However, in dimension $d=1$, there is an explicit expression of the Wasserstein distance and this allows
for
direct computations, including numerical purposes, as shown  
in the pioneering work \cite{GT}.
Indeed, any probability measure $\mu$ on the real line $\RR$
can be described thanks to its cumulative distribution function
$F(x)=\mu((-\infty,x ])$, which is a right-continuous
and non-decreasing function with $F(-\infty)=0$ and $F(+\infty)=1$.
Then we can define the generalized inverse $Q_\mu$ of $F$ (or monotone rearrangement of $\mu$)
by $Q_\mu(z)=F^{-1}(z):=\inf\{x\in \RR : F(x) > z\}$; 
it is a right-continuous and non-decreasing function, defined on $[0,1)$. 
For every non-negative Borel-measurable map $\xi : \RR \rightarrow \RR$, we have
$$
\int_\RR \xi(x) \mu(dx) = \int_0^1 \xi(Q_\mu(z))\,dz.
$$
In particular, $\mu\in \calP_2(\RR)$ if and only if 
$Q_\mu\in L^2((0,1))$.
Moreover, in the one-dimensional setting, there exists a unique optimal transport plan
realizing the minimum in \eqref{defWp}.
More precisely, if $\mu$ and $\nu$ belong to $\calP_p(\RR)$, with monotone rearrangements
$Q_\mu$ and $Q_\nu$, then $\Gamma_0(\mu,\nu)=\{(Q_\mu,Q_\nu)_\# {\mathbb{L}}_{(0,1)}\}$ 
where 
${\mathbb{L}}_{(0,1)}$ is the restriction to $(0,1)$ of the Lebesgue measure.
Then we have the explicit expression of the Wasserstein distance (see \cite{rachev,Villani2}) 
\begin{equation}\label{dWF-1}
d_{W}(\mu,\nu) = \left(\int_0^1 |Q_\mu(z)-Q_\nu(z)|^2\,dz\right)^{1/2},
\end{equation}
and the map $\mu \mapsto Q_\mu$ 
is an isometry between $\calP_2(\RR)$ and the convex subset
of (essentially) non-decreasing functions of 
$L^2([0,1))$.

We will take advantage of this expression \eqref{dWF-1} of the Wasserstein distance in dimension 1 in our numerical simulations to estimate the numerical error of the upwind scheme \eqref{dis_num}.
This scheme in dimension 1 on a Cartesian mesh reads, with time step 
$\Delta t$ and cell size $\Delta x$:
\begin{equation}
\label{scheme1D}
\rho_j^{n+1} = \rho_j^n - \frac{\Delta t}{\Delta x}\left((a_{j}^n)^+ \rho_j^n 
-(a_{j+1}^n)^- \rho_{j+1}^n - (a_{j-1}^n)^+ \rho_{j-1}^n
+(a_{j}^n)^- \rho_j^n \right).
\end{equation}
With this scheme, we define the probability measure $\rho_{\Delta x}^n =\sum_{j\in \ZZ} \rho_j^n\delta_{x_j}$.
Then the generalized inverse of $\rho_{\Delta x}^n$, denoted by 
$Q_{\Delta x}^{n}$, is given by
\begin{equation}
\label{Fdelta}
Q_{\Delta x}^{n}(z) = x_{j+1}, \qquad \mbox{for } z\in \Big[\sum_{k\leq j}\rho_{k}^n,\sum_{k\leq j+1}\rho_{k}^n\Big).
\end{equation}

\subsection{Optimality of the order of convergence}
\label{subsec:optimality}

Thanks to formula \eqref{dWF-1} in dimension $d=1$, we can verify numerically 
the optimality of our result.
Let us consider the potential $W(x)=2x^2$ for $|x|\leq 1$ and $W(x)=4|x|-2$ for $|x|> 1$;
such a potential verifies our assumptions {\bf (A0)--(A3)} with 
$\lambda=0$.
We choose the initial datum
$\rho^{ini}=\frac 12 \delta_{-x_0}+\frac 12 \delta_{x_0}$ with $x_0=0.25$.
Then the solution to the aggregation equation \eqref{EqInter} is given by
$$
\rho(t) = \frac 12 \delta_{-x_0(t)} + \frac 12 \delta_{x_0(t)}, \qquad x_0(t)=\frac 14 e^{-4t}, \qquad t \ge 0.
$$
The generalized inverse $Q_{\rho}(t,\cdot)= Q_{\rho(t)}$ of $\rho(t)$ is given, for $z\in [0,1)$,
by $Q_{\rho}(t,z) = -x_0(t)$ if $z\in [0,1/2)$, and $Q_{\rho}(t,z) = x_0(t)$ if $z\in [1/2,1)$.
Therefore, letting $u_j^n:=\sum_{k\leq j} \rho_k^n$ for $j \in \ZZ$, we can easily  compute the error at time $t^n=n\Delta t$
by means of the two formulas \eqref{dWF-1}--\eqref{Fdelta}:
$$
e_n:=d_W\bigl(\rho(t^n),\rho_{\Delta x}^n\bigr) = \sum_{k\in\ZZ} \int_{u_{k-1}^n}^{u_k^n} |x_{k} - Q_{\rho}(t^n,z)|dz.
$$
We then define the numerical error as $e:=\max_{n\leq T/\Delta t} e_n$.
We display in Figure \ref{fig:error} the numerical error with respect to the number
of nodes in logarithmic scale, as computed with the above procedure (the time steps being chosen in a such a way that the ratio \eqref{CFL} in the CFL condition is kept constant).
We observe that the computed numerical error is of order $1/2$.

\begin{figure}[!ht]
\centering
\includegraphics[width=8cm,height=6.5cm]{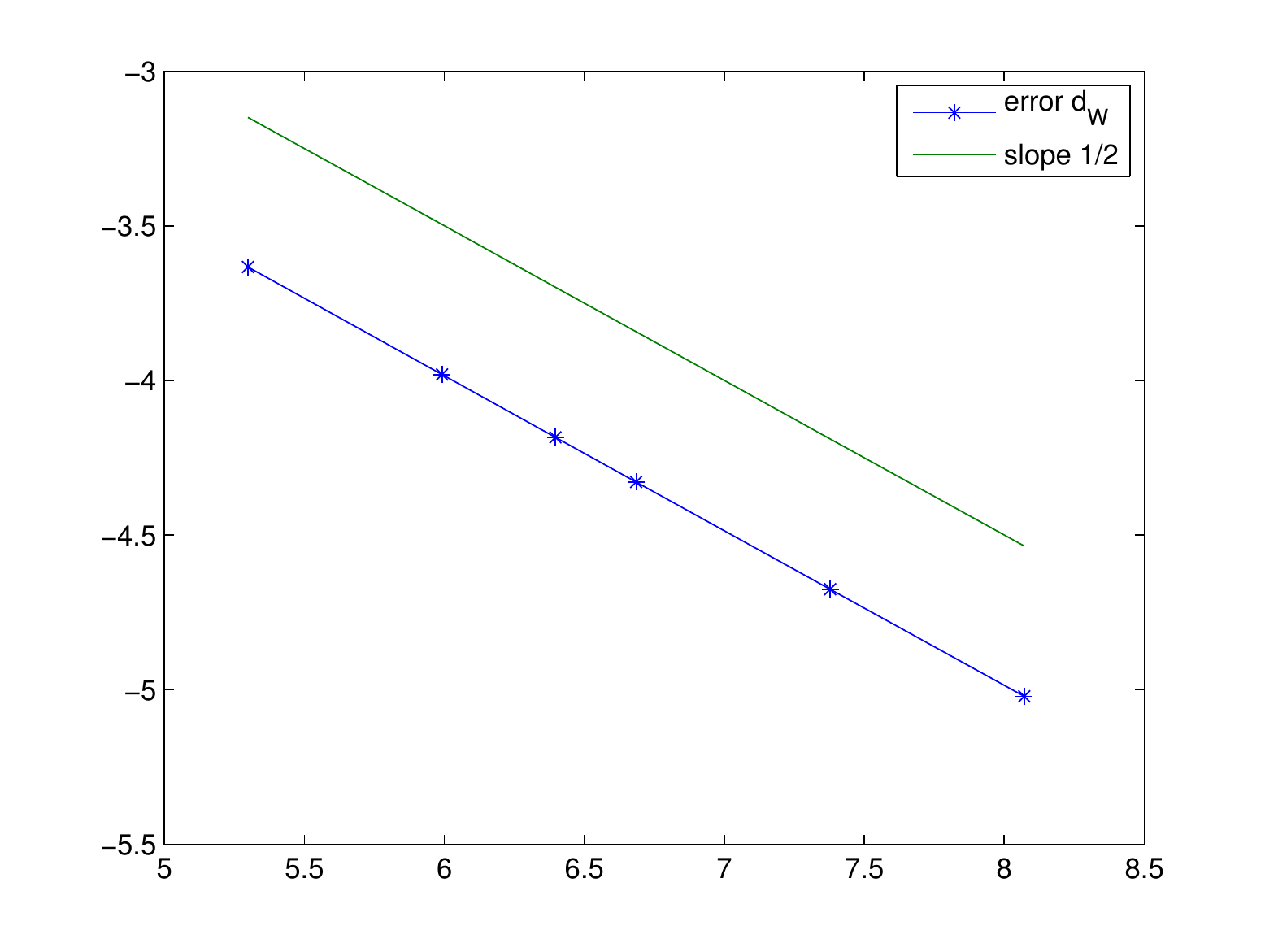}
\caption{Numerical error with respect to the number of nodes in logarithmic scale for the upwind scheme 
in Wasserstein distance for the potential $W$ defined by $W(x)=2x^2$ for $|x|\leq 1$
and $W(x)=4|x|-2$ for $|x|>1$, and an initial datum given by the sum of two Dirac deltas.}
\label{fig:error}
\end{figure}

\subsection{Newtonian potential in one dimension}
\label{subse:newtonian}

An interesting and illustrative example is the Newtonian potential in dimension $d=1$. 
Let us indeed consider the case $W(x)=|x|$ and an initial datum given by the sum of two masses
located at points $x_{i_1}$ and $x_{i_2}$ of the grid mesh, namely
$\rho^{ini}=\frac 12 \delta_{x_{i_1}}+\frac 12 \delta_{x_{i_2}}$, 
with say $x_{i_1}<x_{i_2}$.
The solution of the aggregation equation in Theorem \ref{Exist} is given by
$\rho(t) =\frac 12 \delta_{x_1(t)}+ \frac 12 \delta_{x_2(t)}$, where
$$
x_1(t) = x_{i_1} + \frac{t}{2}, \qquad
x_2(t) = x_{i_2} - \frac{t}{2}, \qquad  \mbox{for } \  t<x_{i_2}-x_{i_1}.
$$
Indeed, recalling definition \eqref{achapo}, we have, for $t<x_{i_2}-x_{i_1}$:
$$
\achapo_{\rho}(t,x) = \left\{ \begin{array}{cl}
1, \qquad & \mbox{ if } x<x_1(t), \\[1mm]
\frac 12, \qquad & \mbox{ if } x=x_1(t), \\[1mm]
0, \qquad & \mbox{ if } x_1(t)<x<x_2(t), \\[1mm]
-\frac 12, \qquad & \mbox{ if } x=x_2(t),  \\[1mm]
- 1, \qquad & \mbox{ if } x>x_2(t).
\end{array}\right.
$$
At $t= x_{i_2}-x_{i_1}$, the two particles collapse, then for $t\geq x_{i_2}-x_{i_1}$, we have 
$\rho(t)=\delta_{\frac 12 (x_{i_1}+x_{i_2})}$.
\\

{\bf Standard finite volume upwind scheme.} 
This simple example explains why we have chosen the scheme \eqref{scheme1D}
instead of the standard finite volume upwind scheme introduced in 
Subsection 
\ref{subse:potential:nonincreasing}. In dimension 
$d=1$ and 
on a Cartesian grid, this latter one reads
\begin{equation}\label{badscheme}
\rho_i^{n+1} = \rho_i^n - \frac{\Delta t}{\Delta x} \left((a_{i+1/2}^n)^+ \rho_i^n 
-(a_{i+1/2}^n)^- \rho_{i+1}^n - (a_{i-1/2}^n)^+ \rho_{i-1}^n + (a_{i-1/2}^n)^- \rho_{i}^n \right),
\end{equation}
where $a_{i+1/2}^n=-\sum_{k\in\ZZ} \rho_k^n \mbox{ sign} (x_{i+1/2}-x_k)$.

Assume indeed that, at time $t^n$, for some $n \in \NN$, we have obtained the approximation
$\rho_i^n = 0$ for $i\in\ZZ\setminus\{i_1,i_2\}$, and $\rho_{i_1}^n=\rho_{i_2}^n=1/2$.
We then compute 
$$
a_{i+1/2}^n = 
\left\{
\begin{array}{cl}
1, \quad &\mbox{for } i<i_1\\ 
0, \quad &\mbox{for } i_1\leq i < i_2\\ 
-1, \quad &\mbox{for } i\geq i_2.
\end{array}
\right.
$$
So, when applying the upwind scheme for $i\in \{i_1-1,i_1,i_1+1\}$, we get
$$
\begin{array}{l}
\ds \rho_{i_1-1}^{n+1} = \rho_{i_1-1}^n - \frac{\Delta t}{\Delta x}\left(\rho_{i_1-1}^n-\rho_{i_1-2}^n\right) = 0, \\[2mm]
\ds \rho_{i_1}^{n+1} = \rho_{i_1}^n + \frac{\Delta t}{\Delta x}\rho_{i_1-1}^n = \rho_{i_1}^n, \\[2mm]
\ds \rho_{i_1+1}^{n+1} = \rho_{i_1+1}^n = 0. \\
\end{array}
$$
Doing the same computation for $i\in \{i_2-1,i_2,i_2+1\}$, we deduce that 
$\rho^{n+1}=\rho^n$.
Thus the above upwind scheme may be not able to capture the correct dynamics of Dirac deltas.
The above computation is illustrated by the numerical results in Figure \ref{fig:wrong},
where a comparison between the numerical results obtained with \eqref{badscheme} (left)
and with \eqref{scheme1D} (right) is displayed.
We observe that the Dirac deltas are stationary when using the scheme \eqref{badscheme},
whereas the scheme \eqref{scheme1D} permits to catch the right dynamics.
Another interesting numerical illustration of this phenomenon is provided by Figure \ref{fig:wrong_exp}. In this example, we choose the potential
$W(x)=1-e^{-2|x|}$, which is $-4$-convex,  and a smooth initial datum given by the sum of two Gaussian functions:
$\rho^{ini}(x) = \frac{1}{M}(e^{-20(x-0.5)^2} + e^{-20(x+0.5)^2}$), {where $M=\|\rho^{ini}\|_{L^1}$ is a normalization coefficient.}
With this choice, we observe that the solution blows-up quickly. 
Dirac deltas appear in finite time and, as observed above, the scheme \eqref{badscheme}
(Fig. \ref{fig:wrong_exp}-left)
does not allow to capture the dynamics after blow-up time, whilst the scheme \eqref{scheme1D}
(Fig. \ref{fig:wrong_exp}-right) succeeds to do so.
For these numerical simulations, the numerical spatial domain is $[-1.25,1.25]$; it is discretized with a uniform Cartesian grid of $800$ nodes, and the ratio in the CFL condition \eqref{CFL} is $1/2$. 
\\

\begin{figure}[!ht]
\centering\includegraphics[width=8cm,height=6.5cm]{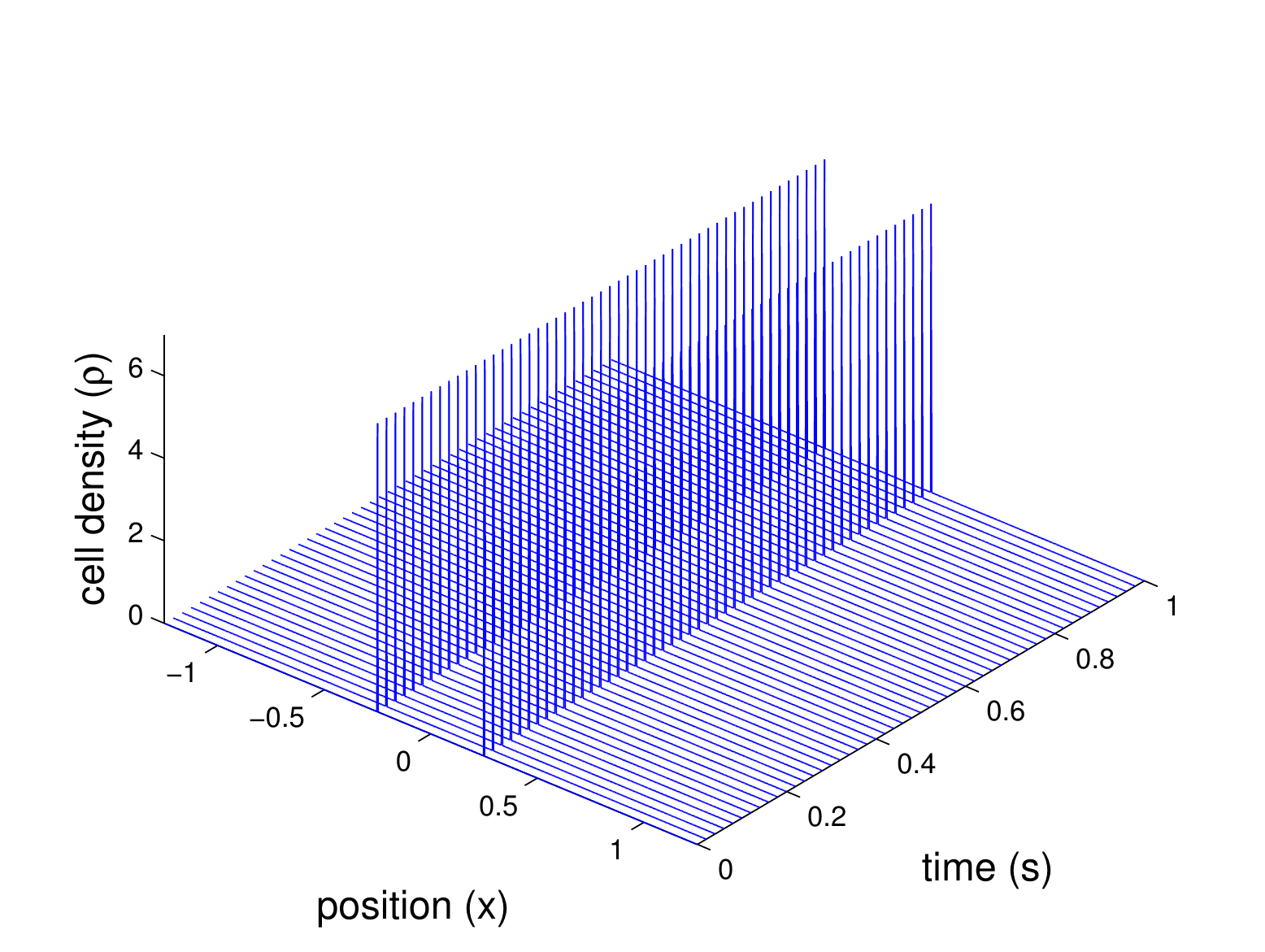}
\includegraphics[width=8cm,height=6.5cm]{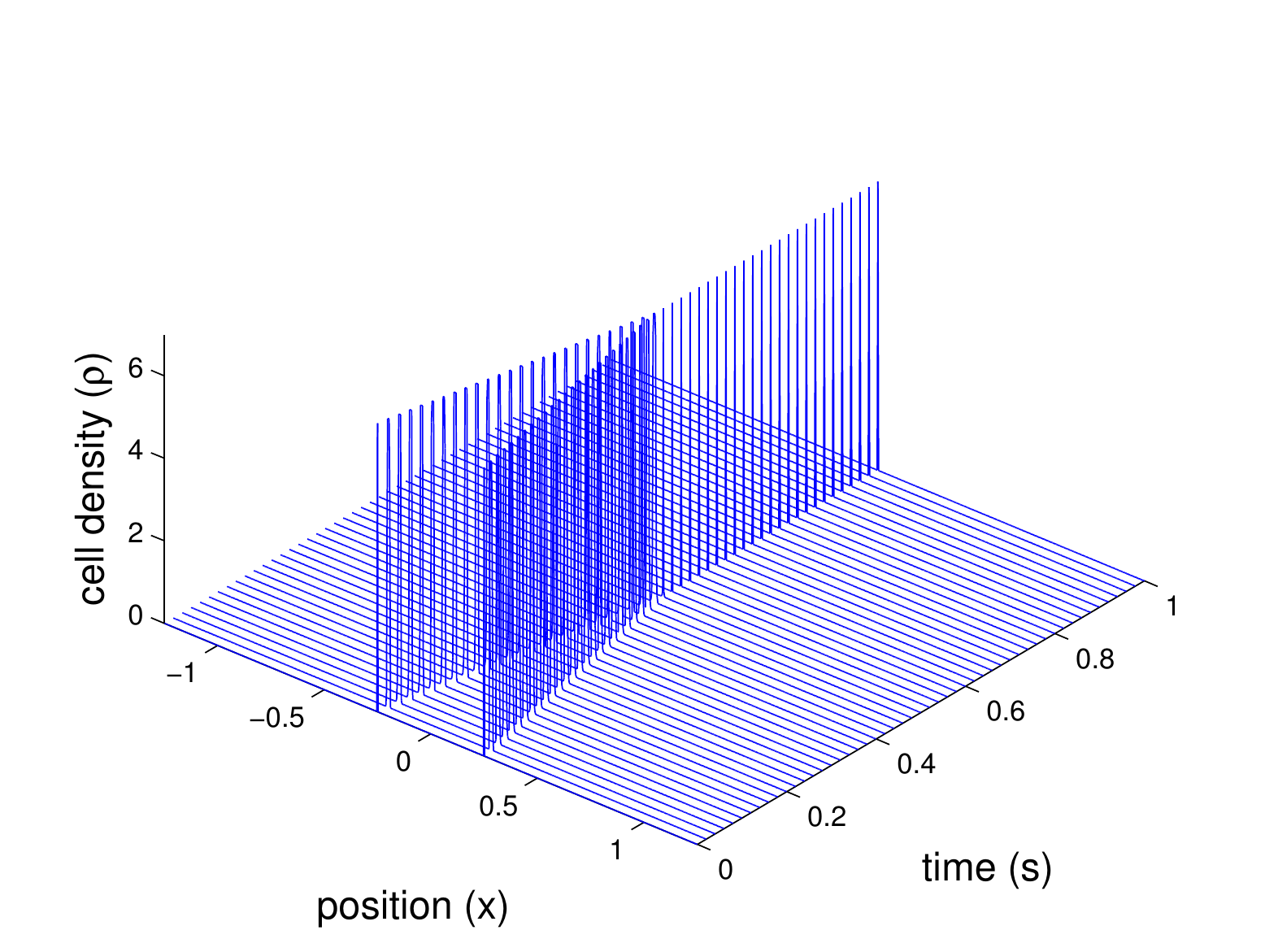}
\caption{Numerical result for the one dimensional aggregation equation with $W(x)=|x|$ 
and an initial datum given by two Dirac deltas. Left: Result obtained with the standard upwind scheme 
\eqref{badscheme} with a velocity computed at the interfaces of the mesh. 
Right: Result with the scheme \eqref{scheme1D}.
As already emphasized in Example 
\ref{ex1D}, this shows once again that a great care must be paid to the choice of the scheme in order to recover the correct dynamics of Dirac deltas.}
\label{fig:wrong}
\end{figure}

\begin{figure}[!ht]
\centering\includegraphics[width=8cm,height=6.5cm]{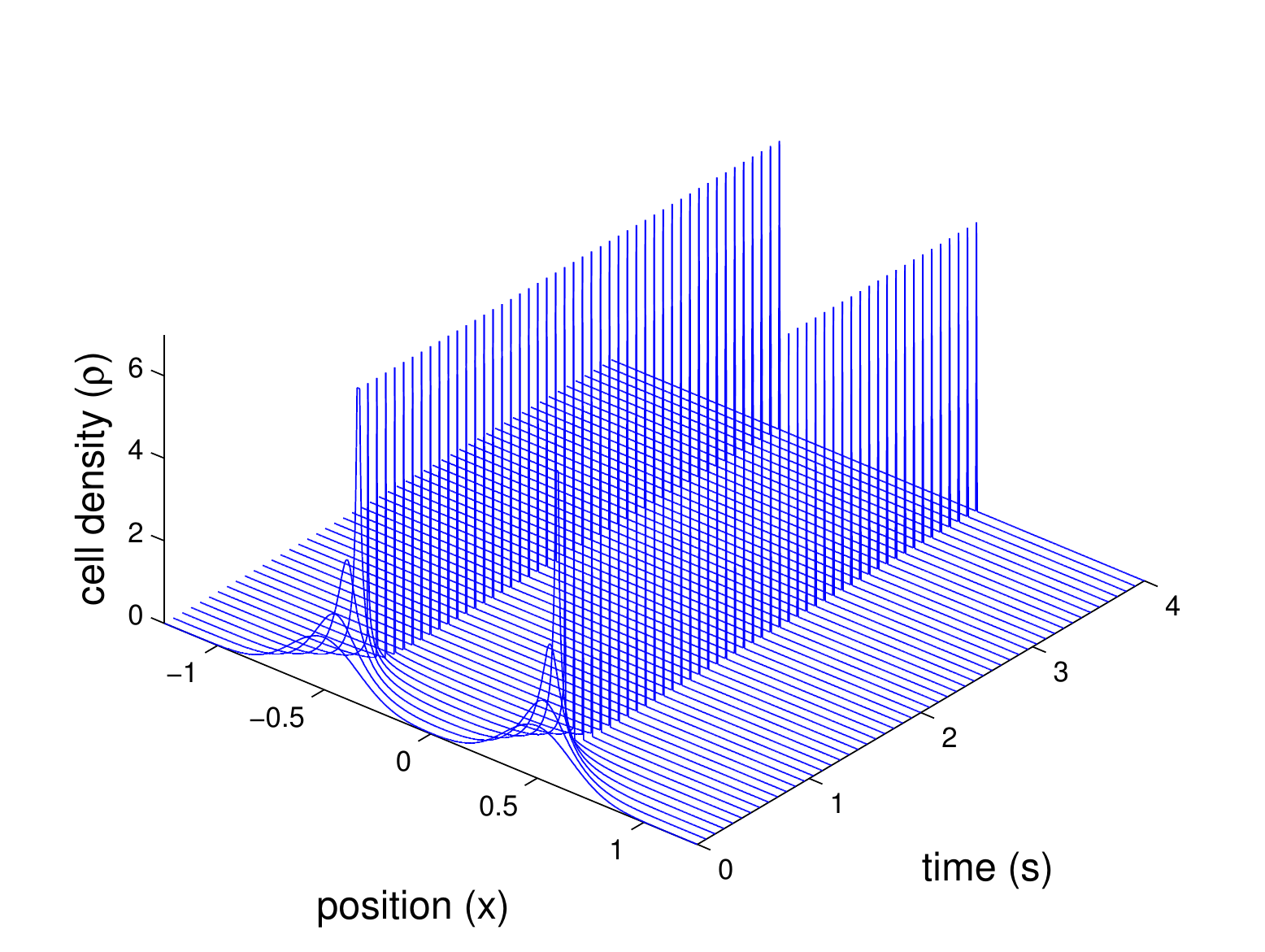}
\includegraphics[width=8cm,height=6.5cm]{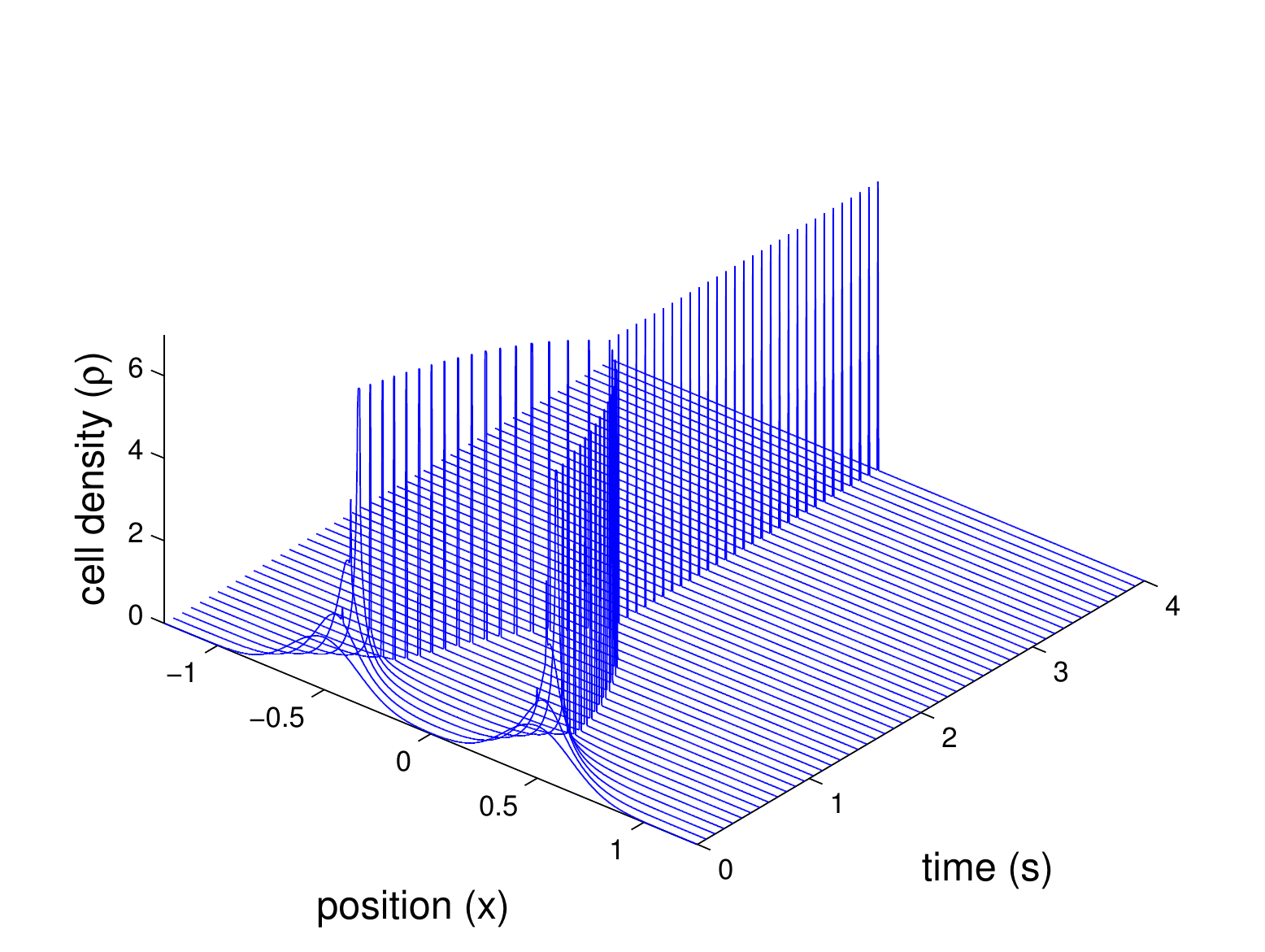}
\caption{Numerical result for the one dimensional aggregation equation with $W(x)=1-e^{-2|x|}$ 
and an initial datum given by the sum of two Gaussian functions. 
Left: Result obtained with the standard upwind scheme 
\eqref{badscheme} with a velocity computed at the interfaces of the mesh. 
Right: Result with the scheme \eqref{scheme1D}.
As in Fig. \ref{fig:wrong}, the upwind scheme \eqref{badscheme} does not capture
the right dynamics of the Dirac deltas after blow-up time.}
\label{fig:wrong_exp}
\end{figure}

{\bf Comparison with Burgers-Hopf equation.}
Considering the potential $W(x)=\frac 12 |x|$,
it has been proved in \cite{GF_dual} (see also \cite{bonaschi}) 
that the following equivalence
 holds true:
$\rho$ is the solution in Theorem \ref{Exist} 
if and only if $u=-W'*\rho$ is the entropy solution of the Burgers-Hopf equation $\pa_t u+\frac 12 \pa_x u^2 = 0$.

Let $(\rho^n_i)_{i \in \ZZ,n \in \NN}$ be given by the scheme \eqref{dis_num}--\eqref{def:aij}.
By conservation of the total mass, see Lemma \ref{bounddismom}, we have $\sum_{k\in \ZZ} \rho_k^n=1$.
Introducing 
\begin{equation*}
u_i^n := \frac 12 - \sum_{k\leq i} \rho_k^n, \quad i \in \ZZ, \ n \in \NN,
\end{equation*}  
we deduce, by summing \eqref{dis_num}
and by using the fact that 
$\rho_i^n=-(u_i^n-u_{i-1}^n)$, 
that the family $(u_i^n)_{i \in \ZZ,n \in \NN}$ satisfies:
\begin{equation}
\label{eq:scheme:burgers}
u_i^{n+1} = u_i^n -\frac{\Delta t}{\Delta x} \bigl((a_i^n)^+ (u_i^n - u_{i-1}^n)
-(a_{i+1}^n)^- (u_{i+1}^n-u_i^n) \bigr),
\end{equation}
where, with \eqref{def:aij}, we have 
$$
a_i^n = -\frac 12 \sum_{k\neq i} \rho_k^n \mbox{sign }(x_i-x_k).
$$
Then
\begin{equation*}
a_i^n = -\frac 12 \biggl(\sum_{k<i} \rho_k^n - \sum_{k>i} \rho_k^n \biggr)
= -\frac 12 \biggl(\sum_{k<i} \rho_k^n - 1 + \sum_{k\leq i} \rho_k^n \biggr)
= \frac 12 (u_{i-1}^n + u_i^n).
\end{equation*}
Moreover, as $\rho_i^n$ remains nonnegative under the CFL condition (see Lemma \ref{lem:CFL}), $u_i^n - u_{i-1}^n = -\rho_i^n \leq 0$, so that 
$$
(a_i^n)^+ (u_i^n-u_{i-1}^n) = -\left(a_i^n (u_i^n-u_{i-1}^n)\right)^- = -\frac 12\left( (u_i^n)^2-(u_{i-1}^n)^2\right)^-.
$$
Similarly, we get 
$$
(a_{i+1}^n)^- (u_{i+1}^n-u_{i}^n) = -\left(a_{i+1}^n (u_{i+1}^n-u_{i}^n)\right)^+ = -\frac 12\left( (u_{i+1}^n)^2-(u_{i}^n)^2\right)^+, 
$$
so that the scheme
\eqref{eq:scheme:burgers}
for $u$ finally rewrites
\begin{equation}\label{schemeBurgers}
u_i^{n+1} = u_i^n -\frac{\Delta t}{2\Delta x} \Big( 
((u_{i+1}^n)^2-(u_i^n)^2)^-
-((u_i^n)^2 - (u_{i-1}^n)^2)^+
\Big).
\end{equation}

Then we may apply the main result of this paper and deduce the convergence at order $1/2$
of the above scheme:
\begin{lemma}
Let $u^{ini}$ be given in $BV(\RR)$ such that $\pa_xu^{ini} \leq 0$ and $TV(u^{ini})=1$.
Define the family $(u^n_{i})_{i \in \ZZ,n\in\NN}$ 
by means of \eqref{schemeBurgers}, 
with the initial data $u_i^{0} := \frac 12 + \pa_x u^{ini}(-\infty,x_{i+\frac 12})$, and 
let $u_{\Delta x}^n:=\sum_{i\in\ZZ} u_i^n \mathbf{1}_{[x_i,x_{i+1})}$. 
Let $u$ be the entropy solution to the Burgers equation $\pa_t u + \frac 12 \pa_x u^2=0$
with $u^{ini}$ as initial condition.
Then, there exists $C\geq 0$, independent of the discretization parameters, such that if the CFL condition 
$\Delta t <  \Delta x$ is satisfied, one has 
$$
\|u(t^n)-u^n_{\Delta x}\|_{L^1} \leq C \bigl( 
\sqrt{t^n \Delta x} + \Delta x \bigr).
$$
\end{lemma}

\begin{remark}
We do not claim that the scheme converges for any initial datum of the Cauchy problem for the Burgers equation (and actually it does not). The convergence result above only applies to a non-increasing initial condition belonging to $[-1/2, 1/2]$. \\
Note that this scheme is not conservative, but, surprisingly (see \cite{Legloc}) this does not prevent it from converging toward the right solution. 
\end{remark}

\begin{proof}
First remark that the CFL condition that is here required is $w_\infty \Delta t < \frac 12 \Delta x$, with $w_\infty = 1/2$ as $W(x) = \frac{1}{2} |x|$. \\ 
The entropy solution $u$ of the Burgers equation with a nonincreasing $BV$ initial datum
is a nonincreasing $BV$ function. 
By Cauchy-Schwarz inequality, we have
$$
\int_{0}^{1} |Q_{\rho(t^n)}(z) - Q_{\rho^n_{\Delta x}}(z)|\,dz \leq 
\| Q_{\rho(t^n)} - Q_{\rho^n_{\Delta x}}\|_{L^2(0,1)} = d_W\bigl(\rho(t^n),\rho_{\Delta x}^n\bigr),
$$
where $(\rho(t))_{t \ge 0}$ is the solution of 
\eqref{EqInter}, with $W(x) = \frac{1}{2}|x|$ as before and $\rho^{ini}=-\partial_{x} u^{ini}$ as initial condition, and $(\rho^n_{\Delta x})_{n \geq 0}$ is the numerical solution obtained by Scheme \eqref{dis_num} with $d = 1$ together with initial condition \eqref{disrho0} (numerical solution whose convergence at order $1/2$ is stated in Theorem \ref{TH}).

Observing that $W$ is convex, we apply Theorem \ref{TH} with $\lambda =0$. We obtain
$$
\int_{0}^{1} |Q_{\rho(t^n)}(z) - Q_{\rho^n_{\Delta x}}(z)|\,dz \leq 
d_W(\rho(t^n),\rho_{\Delta x}^n) \leq
C \bigl( \sqrt{t^n \Delta x} + \Delta x \bigr).
$$
The claim follows provided we prove that 
\begin{equation}\label{eq:integral}
\int_{\RR} |u(t^n,x)-u^n_{\Delta x}(x)|\,dx = \int_{0}^{1} 
|Q_{\rho(t^n)}(z) - Q_{\rho^n_{\Delta x}}(z)|\,dz.
\end{equation}

In order to prove \eqref{eq:integral}, we notice that, from a geometrical point of view,
the left hand side of equality \eqref{eq:integral} corresponds to the area between 
the curves $x\mapsto u(t^n,x)$ and $x\mapsto u^n_{\Delta x}(x)$. 
Also, the right hand side is a measure of the area between their generalized inverses.
However, the graph of the pseudo-inverse of a function may be obtained by flipping 
the graph of the function 
with respect to the diagonal. 
Since this operation conserves the area, we deduce that both areas are equal, that is
\eqref{eq:integral} holds.

Another way to prove the identity \eqref{eq:integral} is to observe that the solution 
$u$ of the Burgers-Hopf equation reads:
\begin{equation*}
u(t,x) = \frac12\bigl[ \rho\bigl(t,(x,+\infty)\bigr) - 
\rho\bigl(t,(-\infty,x)\bigr) \bigr], \quad t \geq 0, \ x \in {\mathbb R},
\end{equation*}
where $\rho$ is the solution in Theorem \ref{Exist}. In fact, as the number of points $x$ for which $\rho(t,\{x\})>0$ is at most countable for any given $t >0$, we have the almost everywhere equality:
\begin{equation*}
u(t,x) = \rho\bigl(t,(x,+\infty)\bigr) - \frac12.
\end{equation*}
Similarly, 
\begin{equation*}
\begin{split}
u^n_{\Delta x}(t,x) &= \sum_{i \in {\mathbb Z}}
u ^n_{i} {\mathbf 1}_{[x_{i},x_{i+1})}(x)
= \frac12 - 
\sum_{i \in {\mathbb Z}} {\mathbf 1}_{[x_{i},x_{i+1})}(x)
\sum_{k \leq i} \rho^n_{k}
\\
&= \frac12 - \sum_{i \in {\mathbb Z}} 
{\mathbf 1}_{[x_{i},x_{i+1})}(x) \rho^n_{\Delta x}(t,(-\infty,x_{i}])
= \frac12 - \rho^n_{\Delta x}\bigl(t,(-\infty,x]\bigr) = \rho^n_{\Delta x}\bigl(t,(x,+\infty)\bigr) - 
\frac12.
\end{split}
\end{equation*} 
 
So, to complete the proof, it suffices to use the fact that, for any two
probability measures $\mu$ and $\mu'$ on $\RR$, 
\begin{equation*}
\begin{split}
\int_{\RR} \bigl\vert \mu\bigl((x,+\infty)\bigr)-
\mu'\bigl((x,+\infty)\bigr) \bigr\vert dx &= \int_{0}^1 
\vert Q_{\mu}(z) - Q_{\mu'}(z) \vert 
 dz,
\end{split}
\end{equation*}  
see 
\cite[Theorems 2.9 and 2.10]{bobkov:ledoux}, noticing that the function 
$Q_{\mu}$ we use here is the right continuous version of the quantile function used in 
\cite{bobkov:ledoux}.

\end{proof}

\subsection{Numerical simulation in two dimensions}
\label{subsec:numerical:d2}

\begin{figure}[ht!]
\includegraphics[width=.51\textwidth]{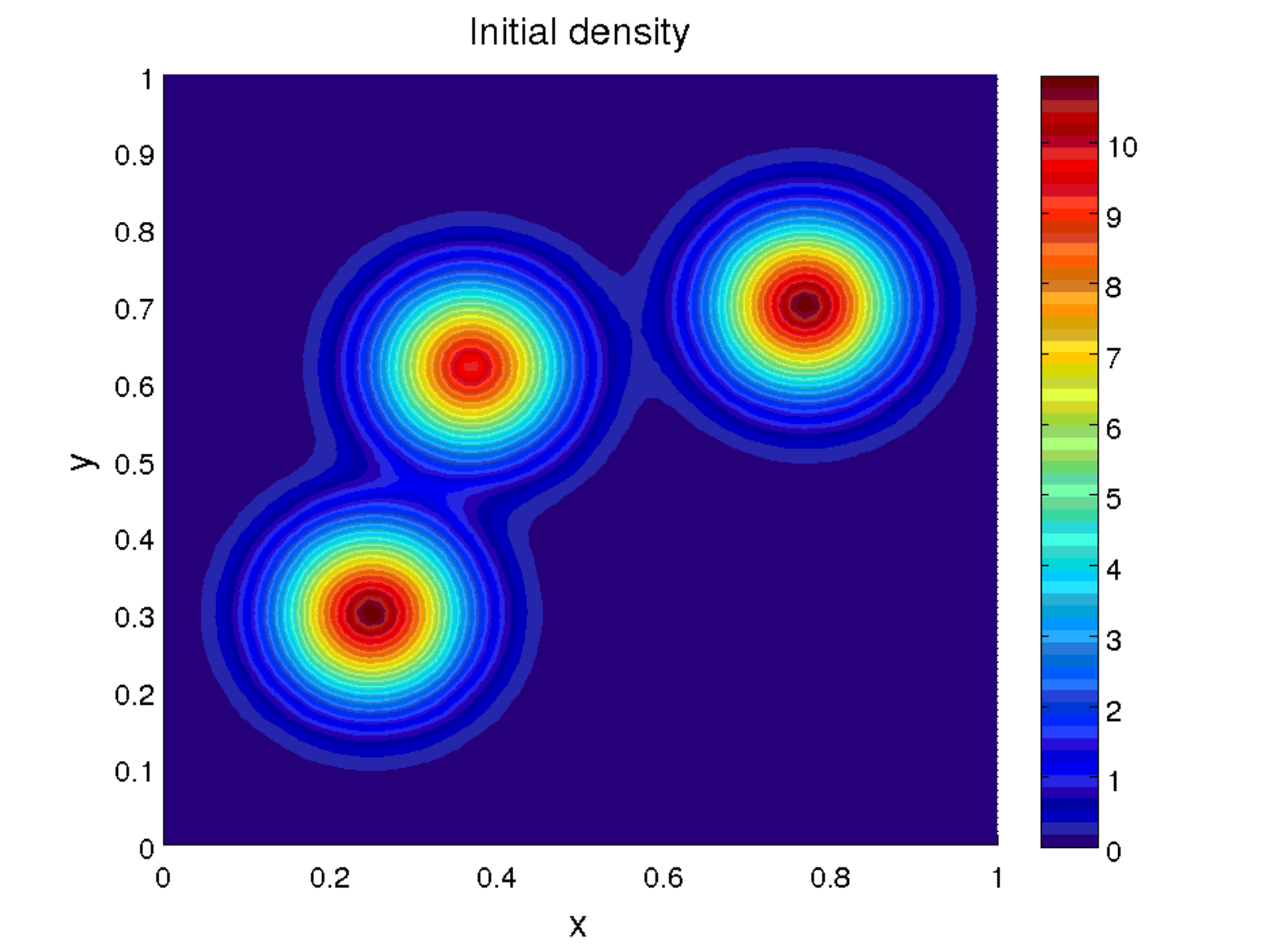}
\includegraphics[width=.51\textwidth]{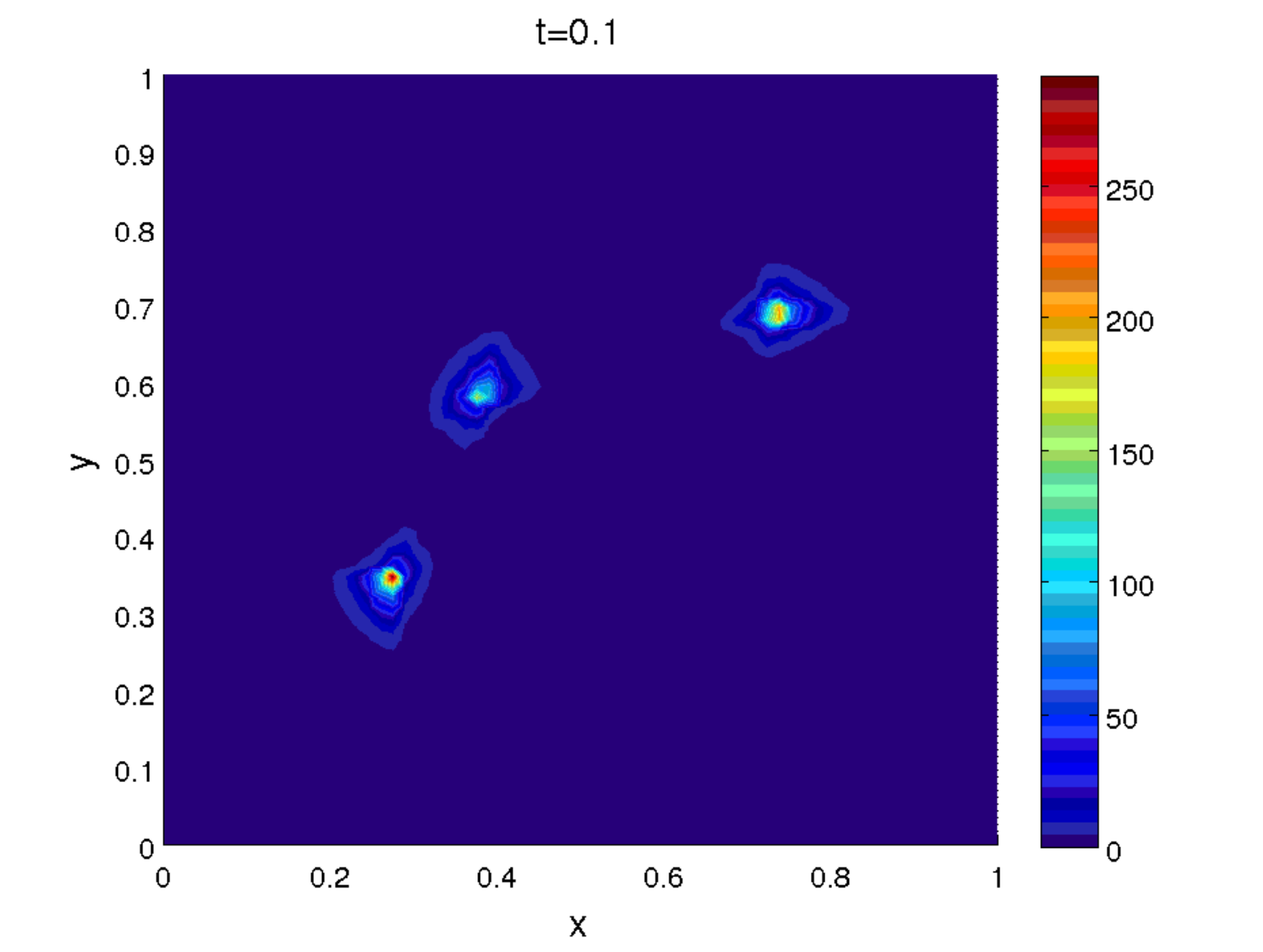}
\includegraphics[width=.51\textwidth]{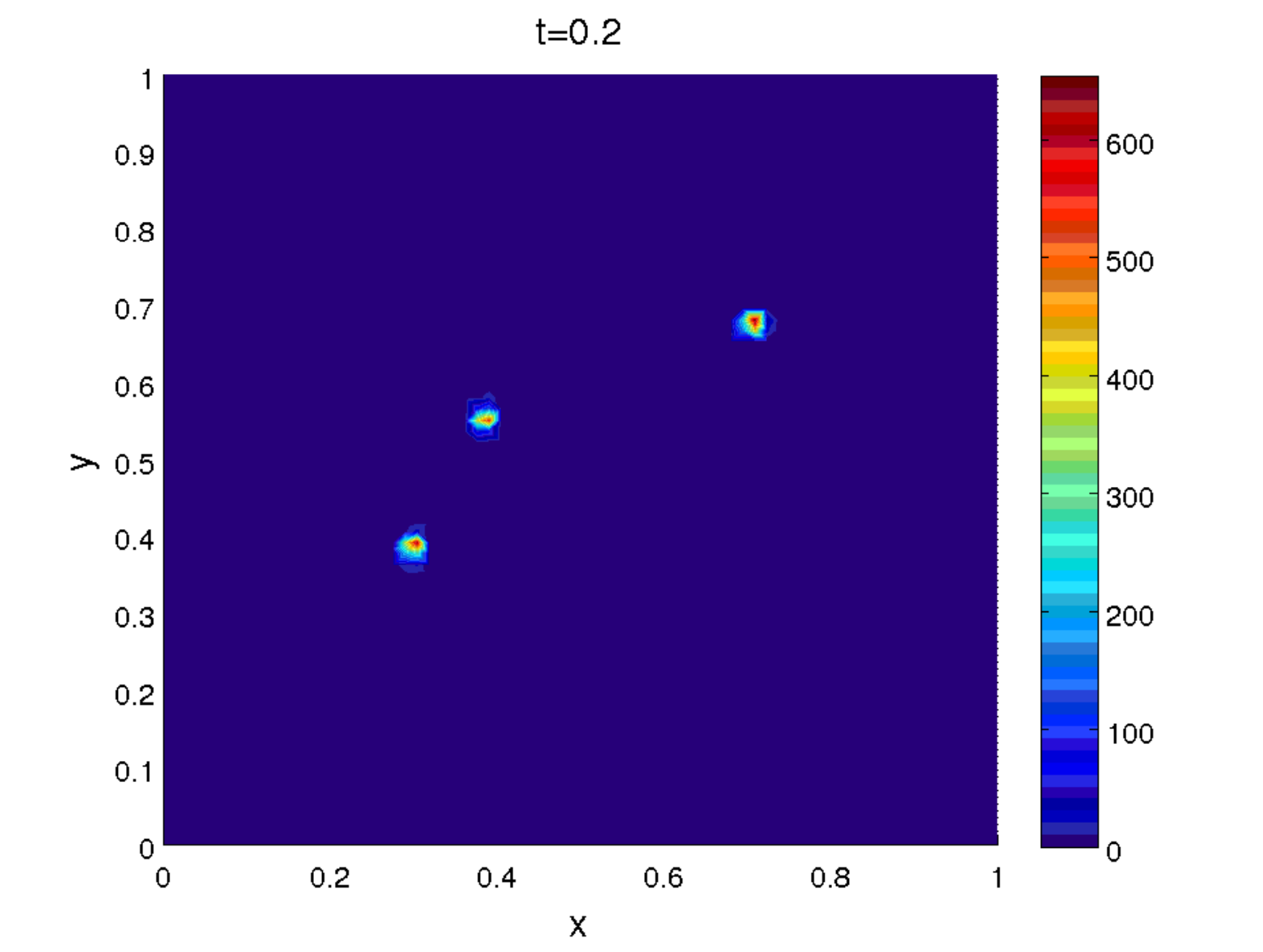}
\includegraphics[width=.51\textwidth]{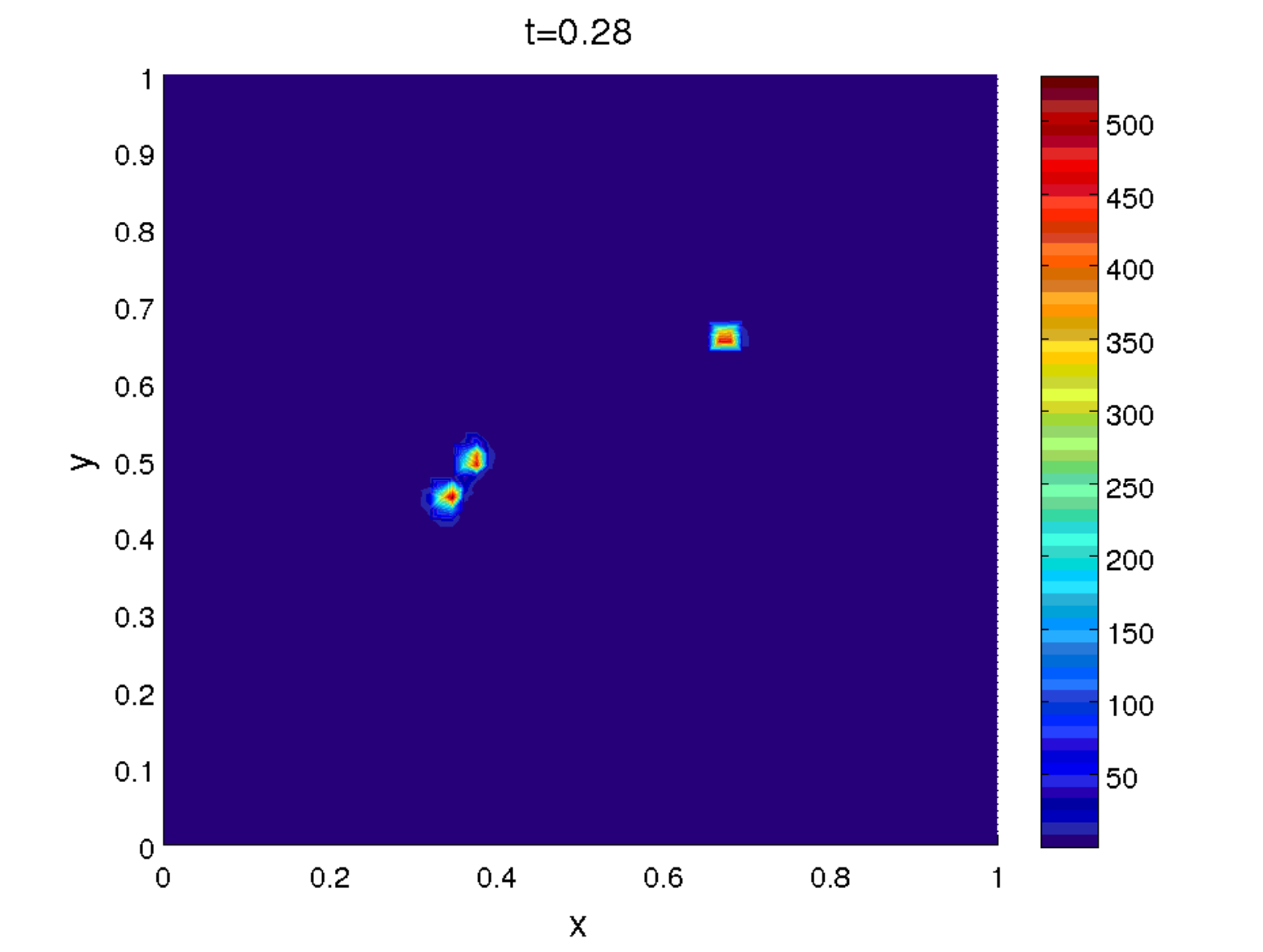}
\includegraphics[width=.51\textwidth]{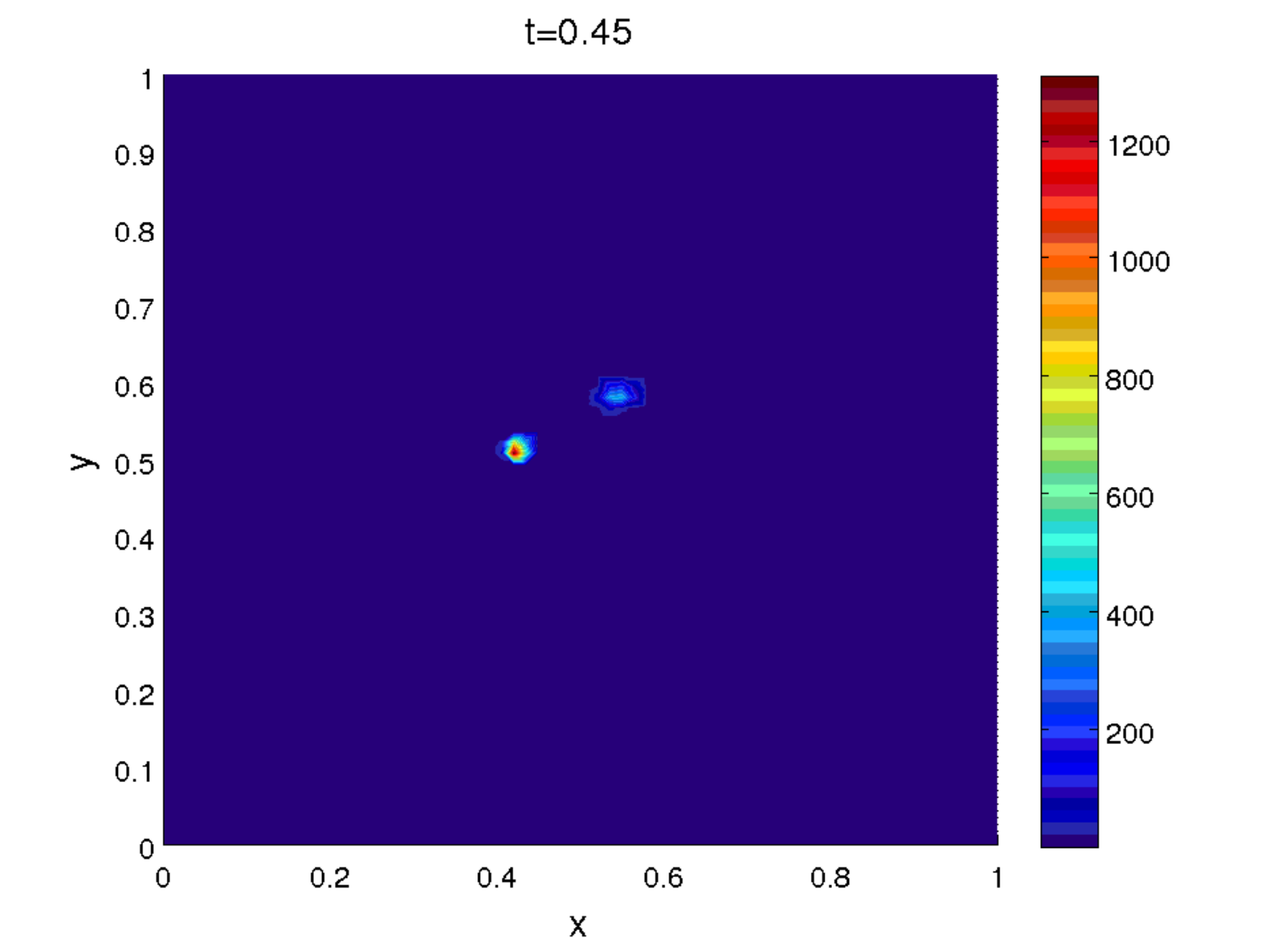}
\includegraphics[width=.51\textwidth]{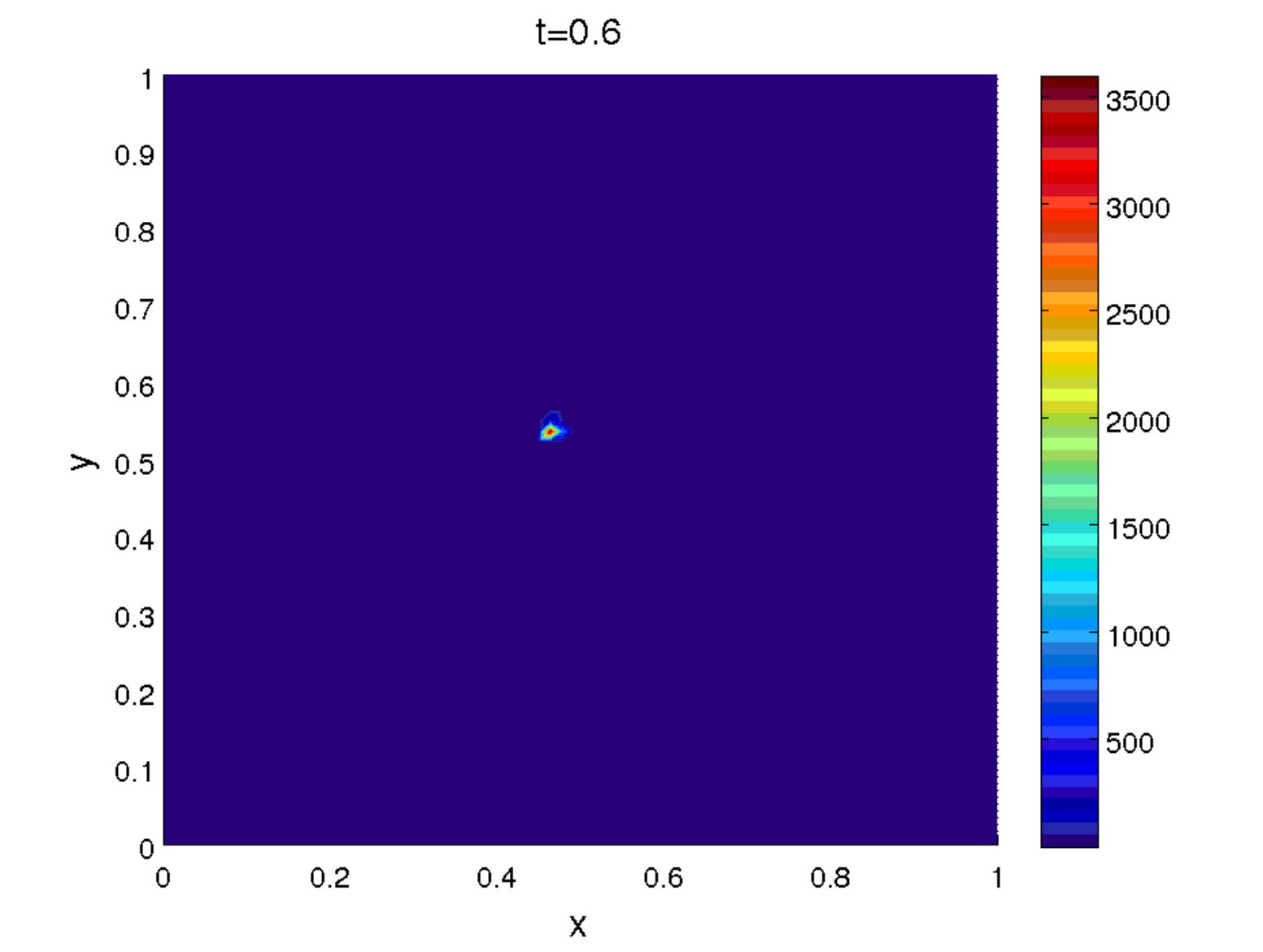}
\caption{Time dynamics of the numerical solution of the aggregation equation 
\eqref{EqInter} with {$W(x)=W_1(x) = 1-e^{-5|x|}$} and an initial datum given by the sum of 
three bumps.
Time increases from top left to bottom right.}
\label{bump2D}       
\end{figure}

\begin{figure}[ht!]
\includegraphics[width=.51\textwidth]{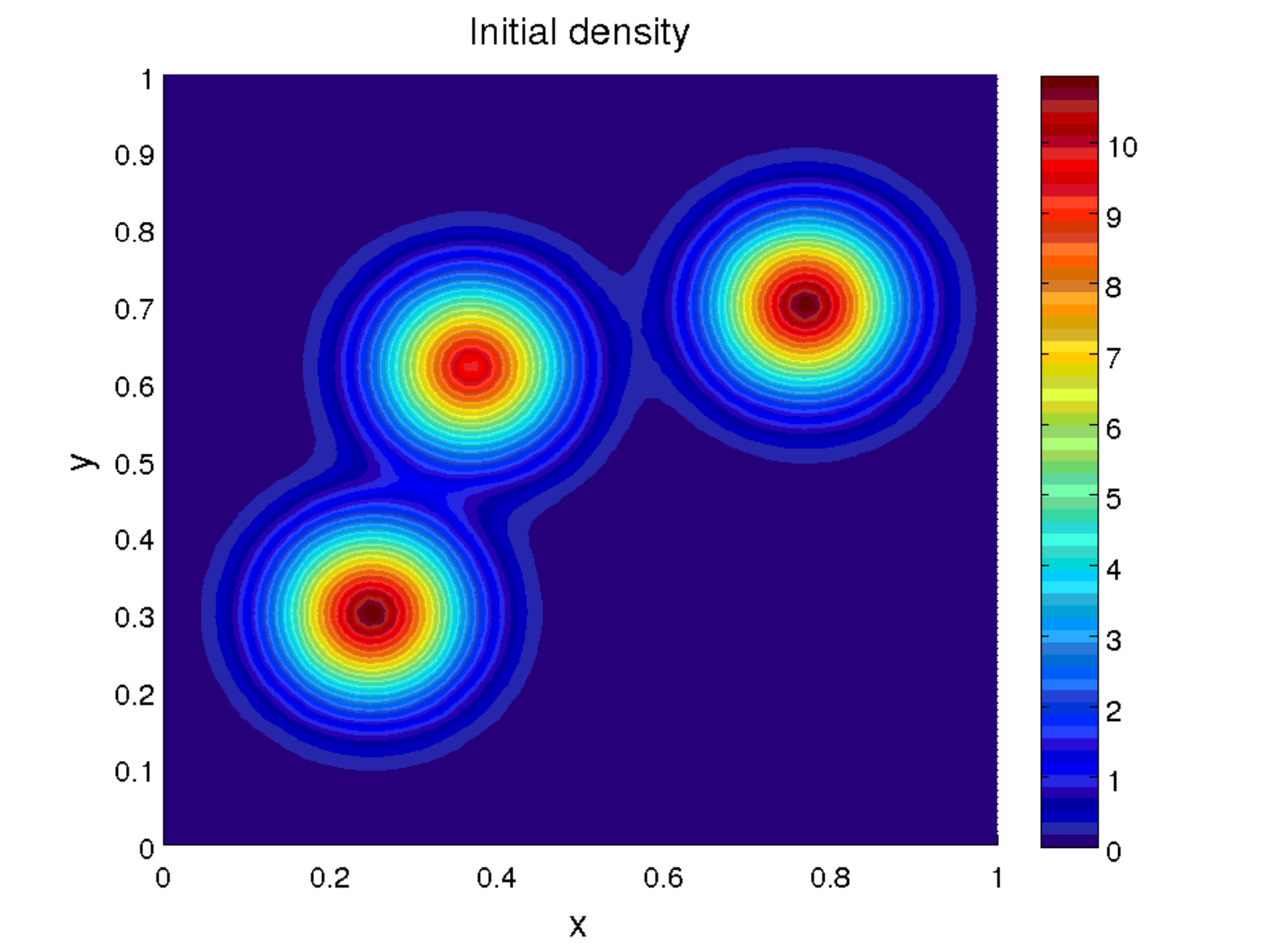}
\includegraphics[width=.51\textwidth]{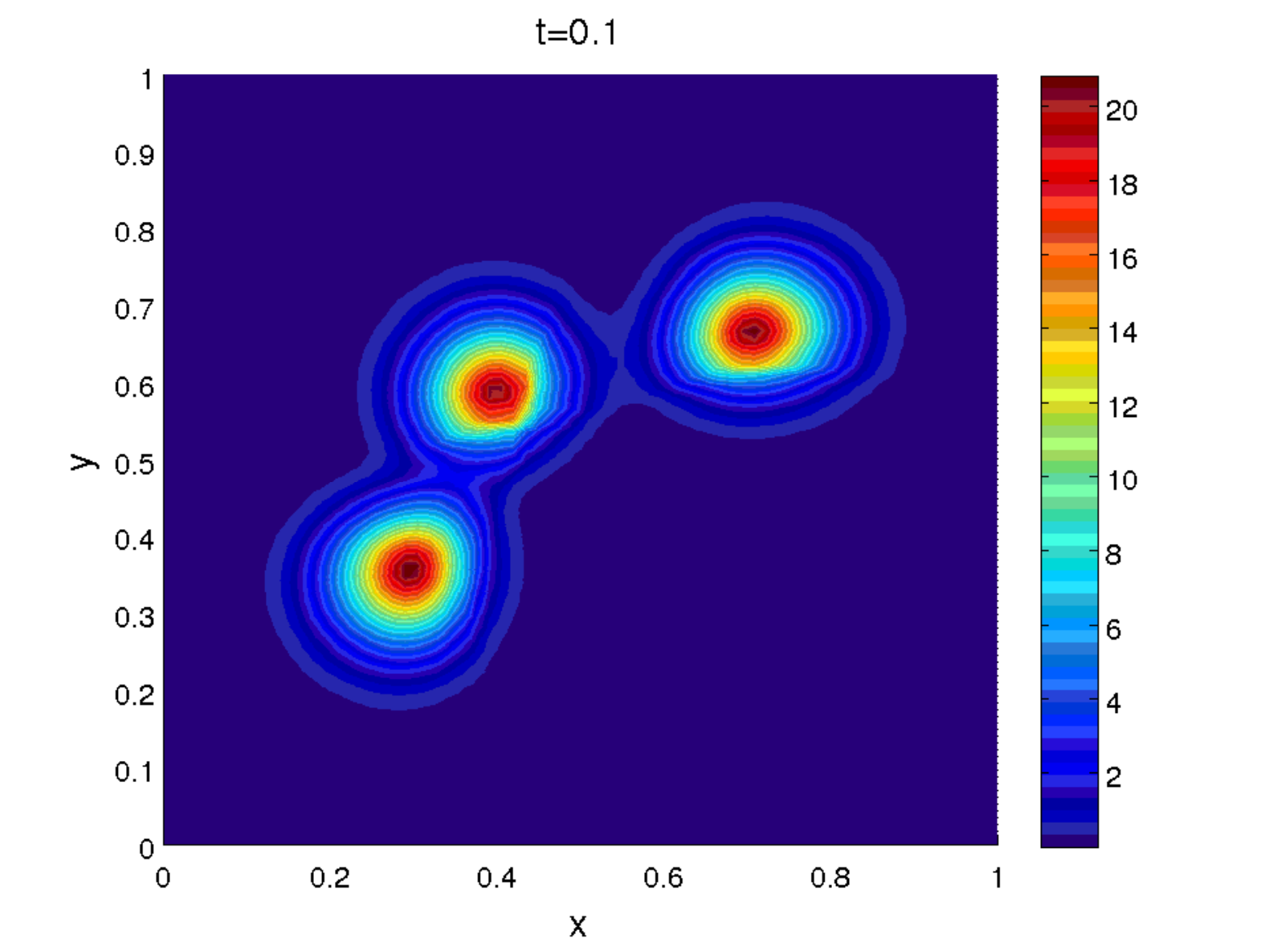}
\includegraphics[width=.51\textwidth]{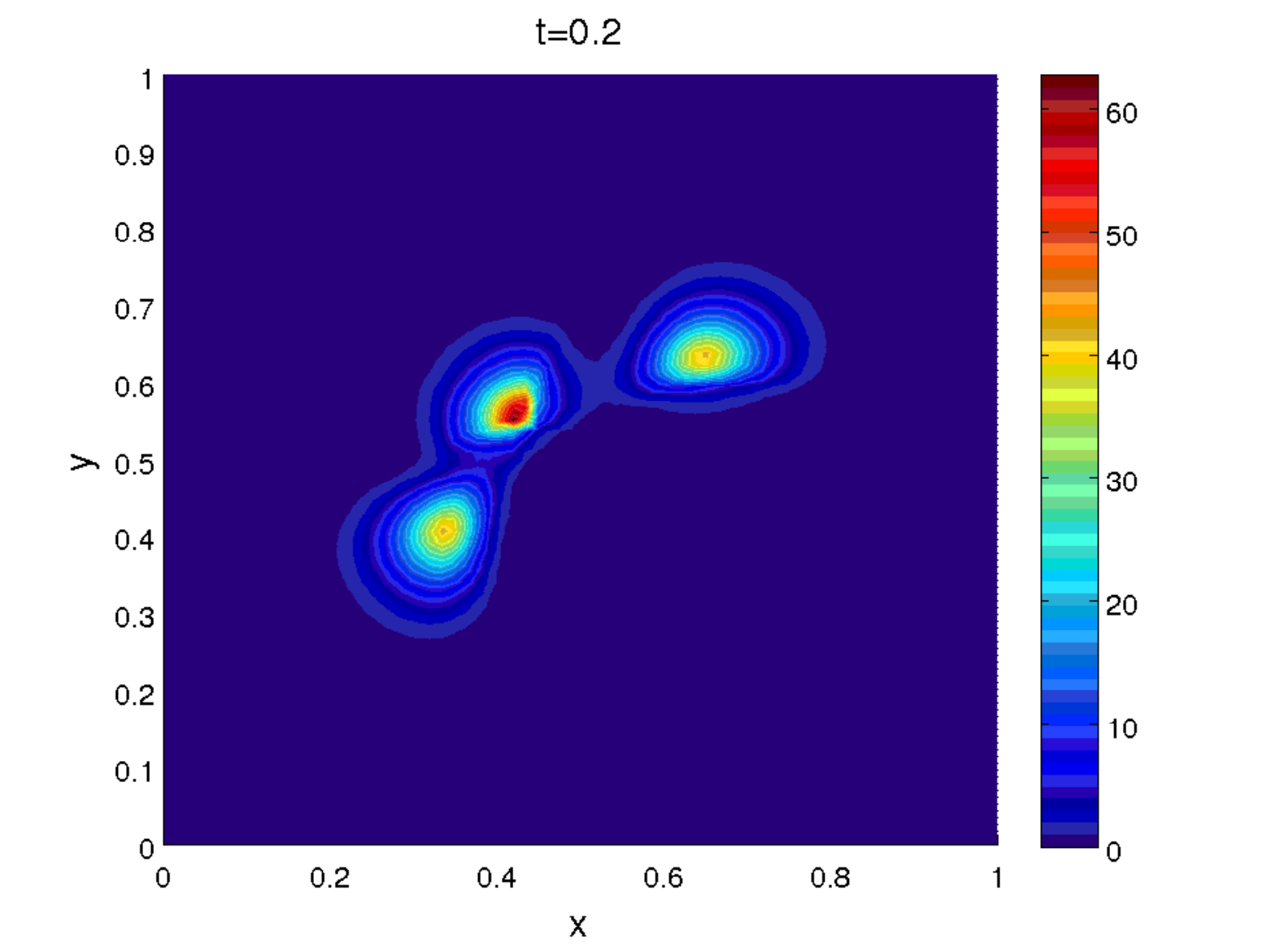}
\includegraphics[width=.51\textwidth]{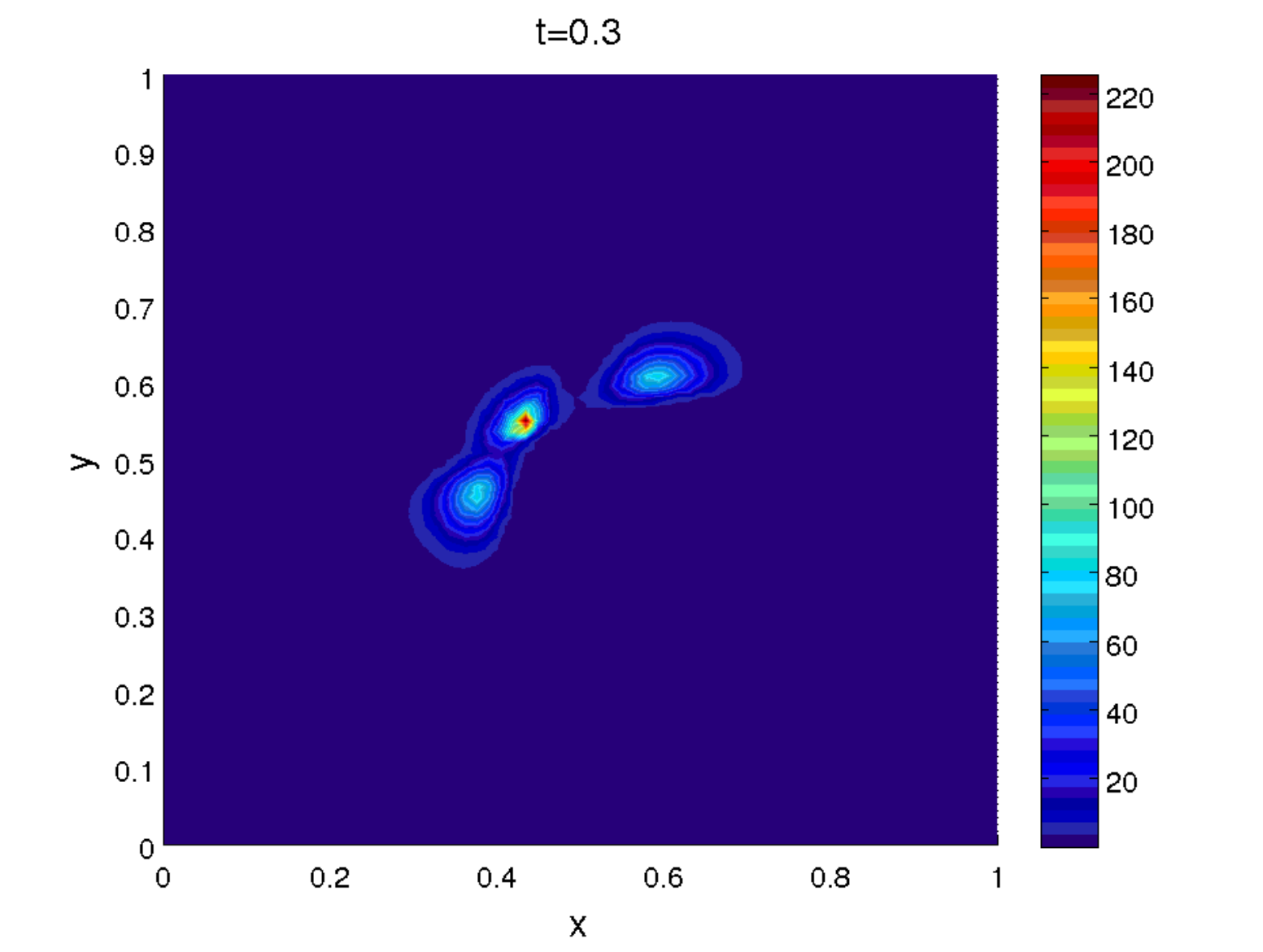}
\includegraphics[width=.51\textwidth]{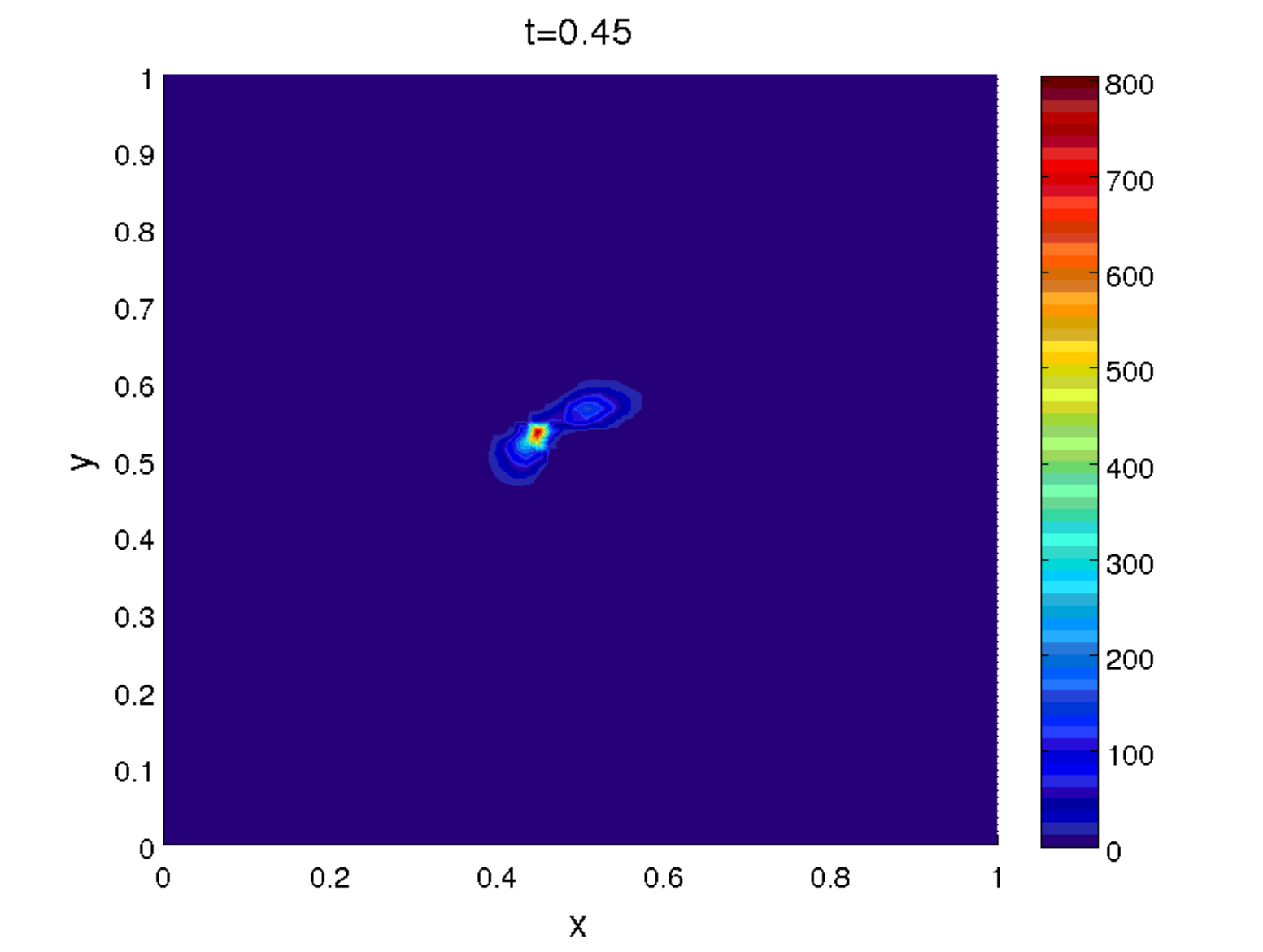}
\includegraphics[width=.51\textwidth]{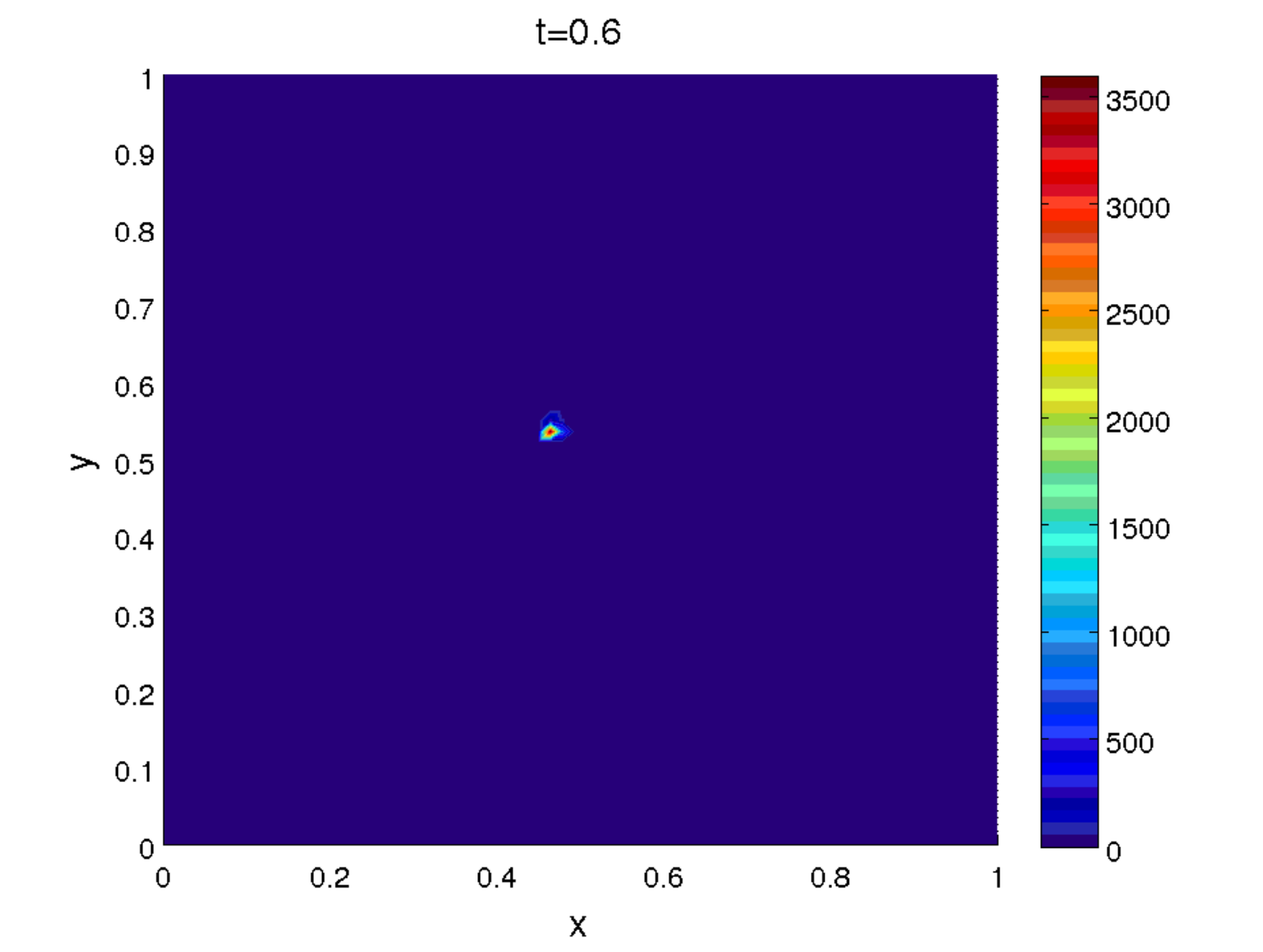}
\caption{Time dynamics of the numerical solution of the aggregation equation 
\eqref{EqInter} with $W(x)=W_2(x) = 5|x|$ and an initial datum given by the sum of three
bumps.
Time increases from top left to bottom right.}
\label{bump2Dbis}       
\end{figure}

\begin{figure}[ht!]
\includegraphics[width=.51\textwidth]{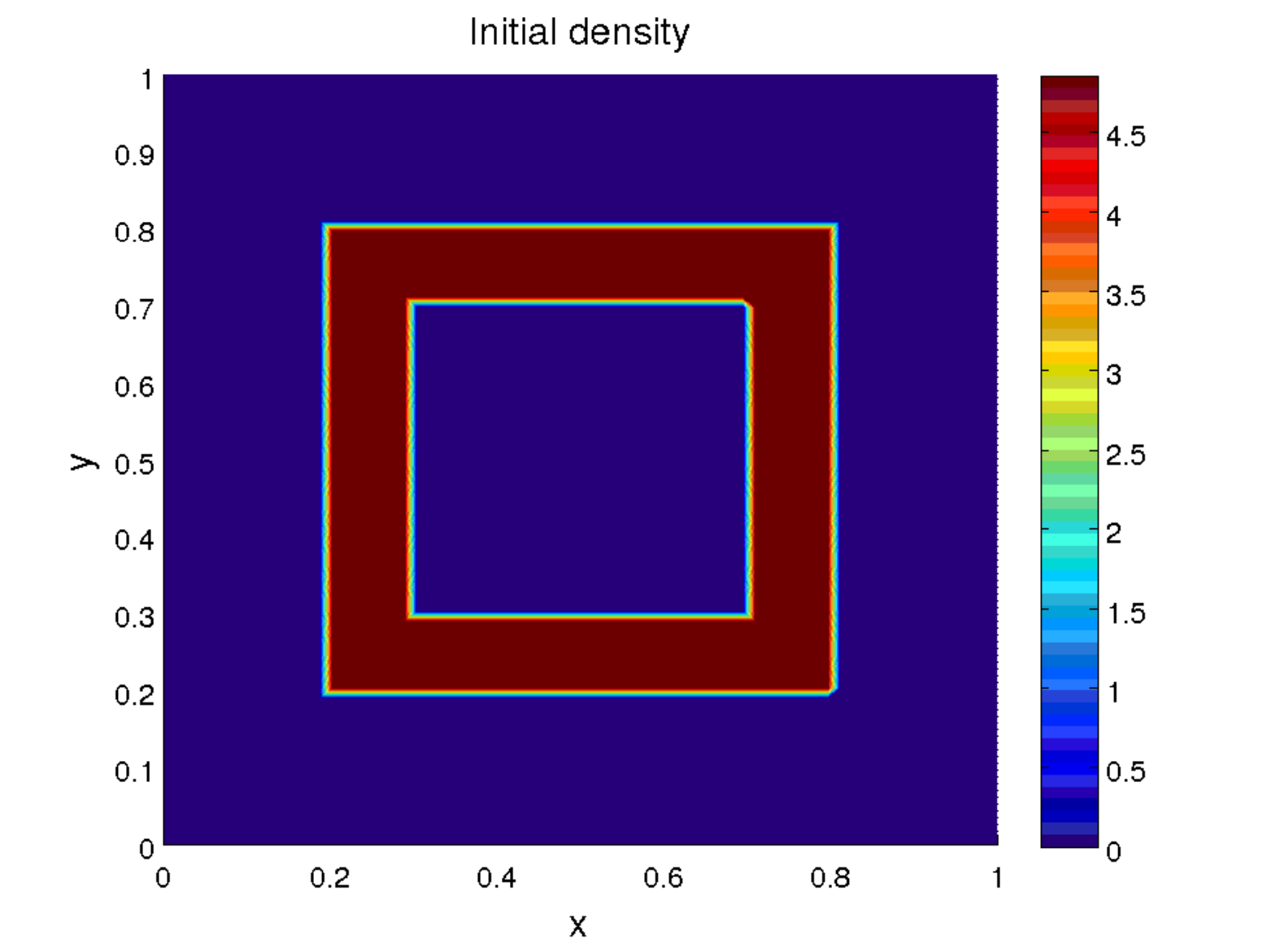}
\includegraphics[width=.51\textwidth]{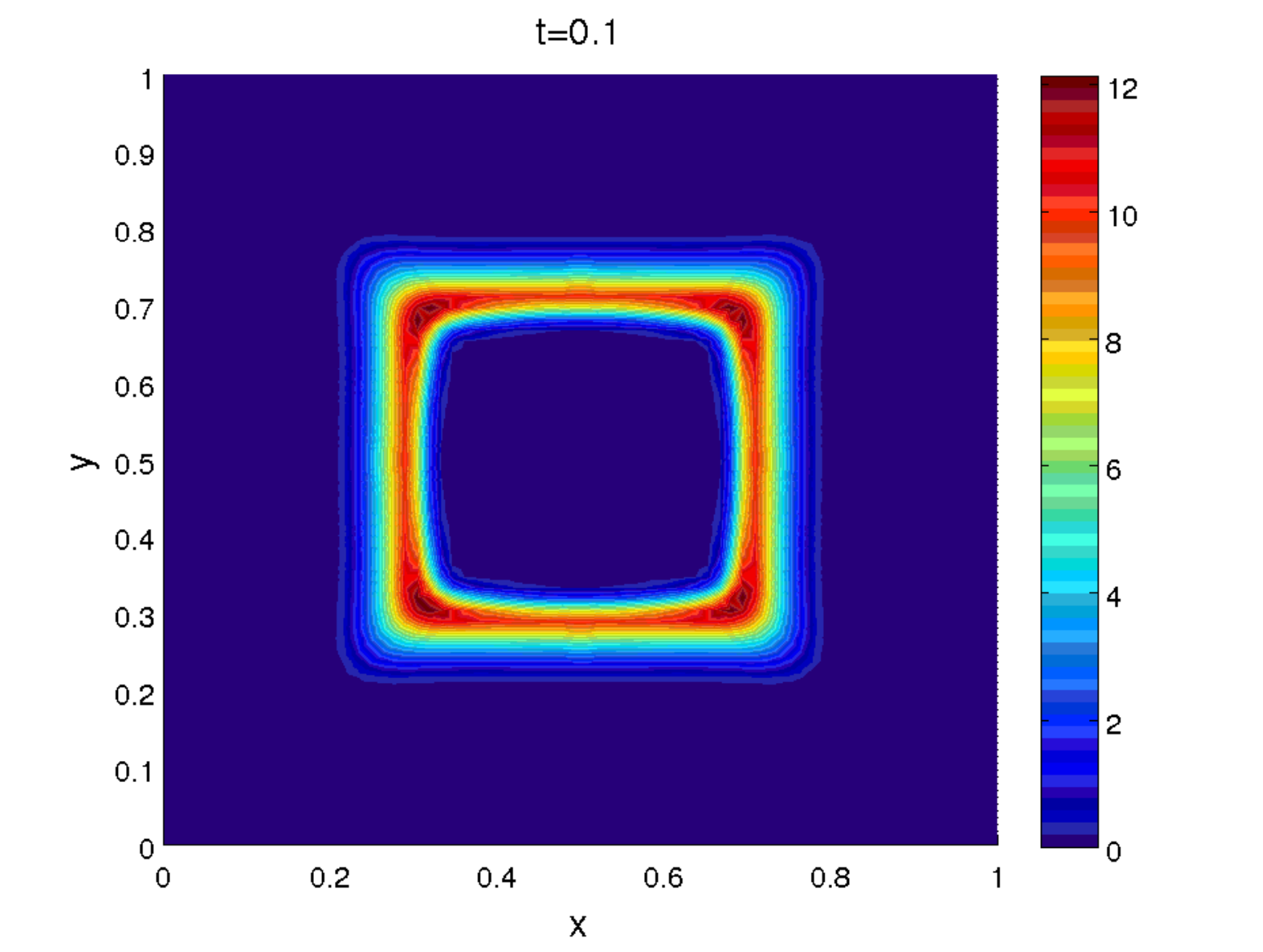}
\includegraphics[width=.51\textwidth]{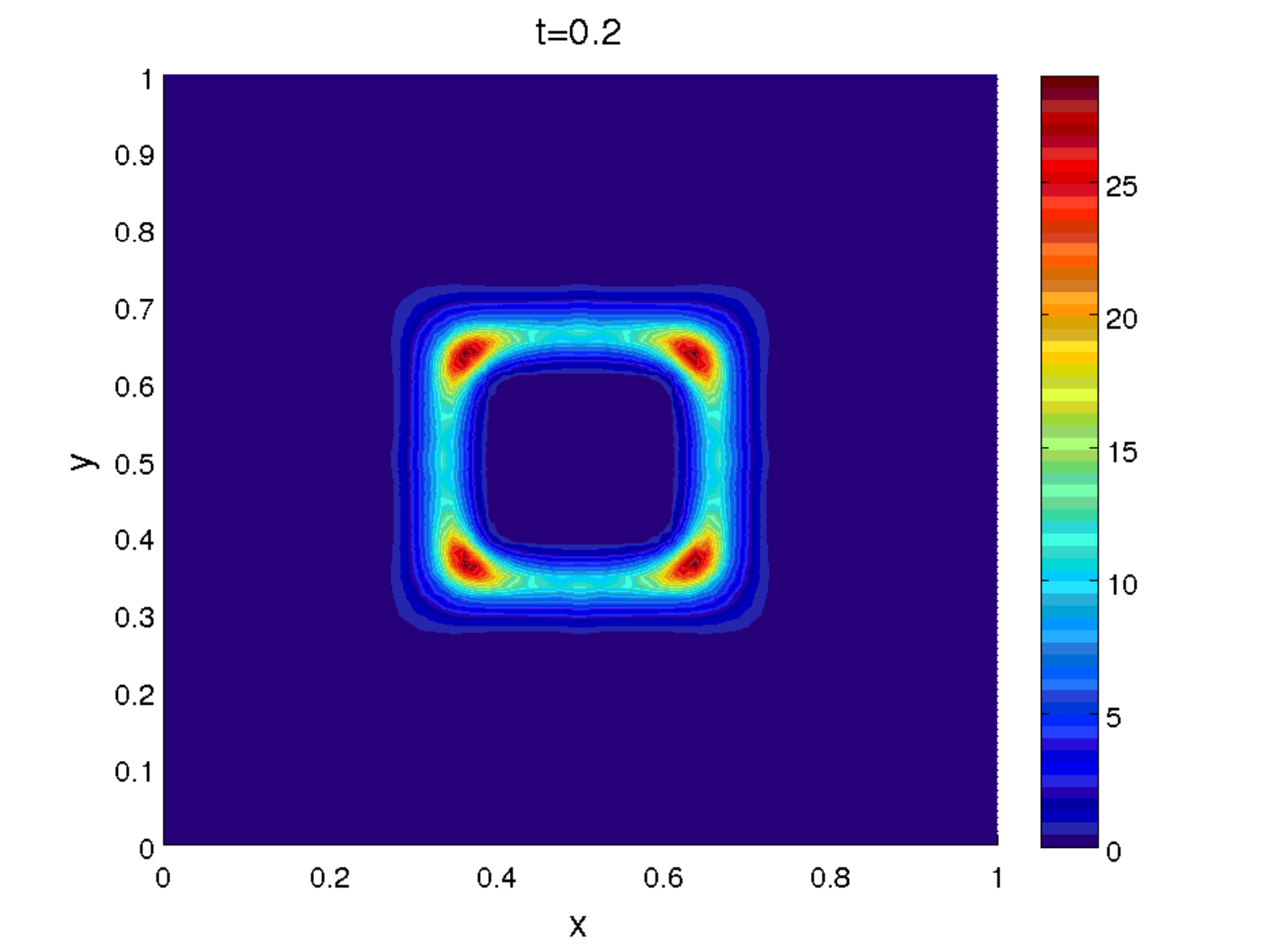}
\includegraphics[width=.51\textwidth]{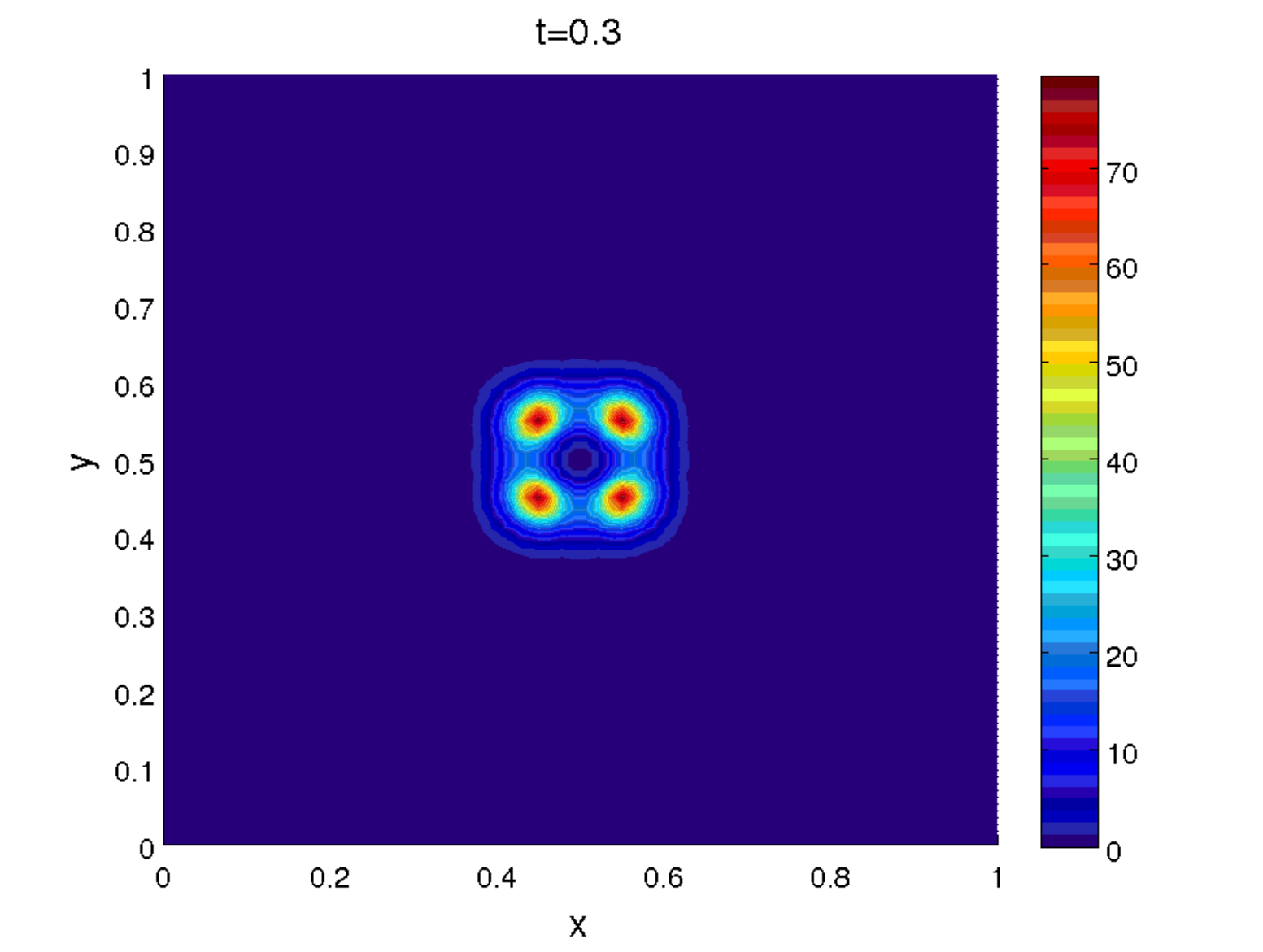}
\includegraphics[width=.51\textwidth]{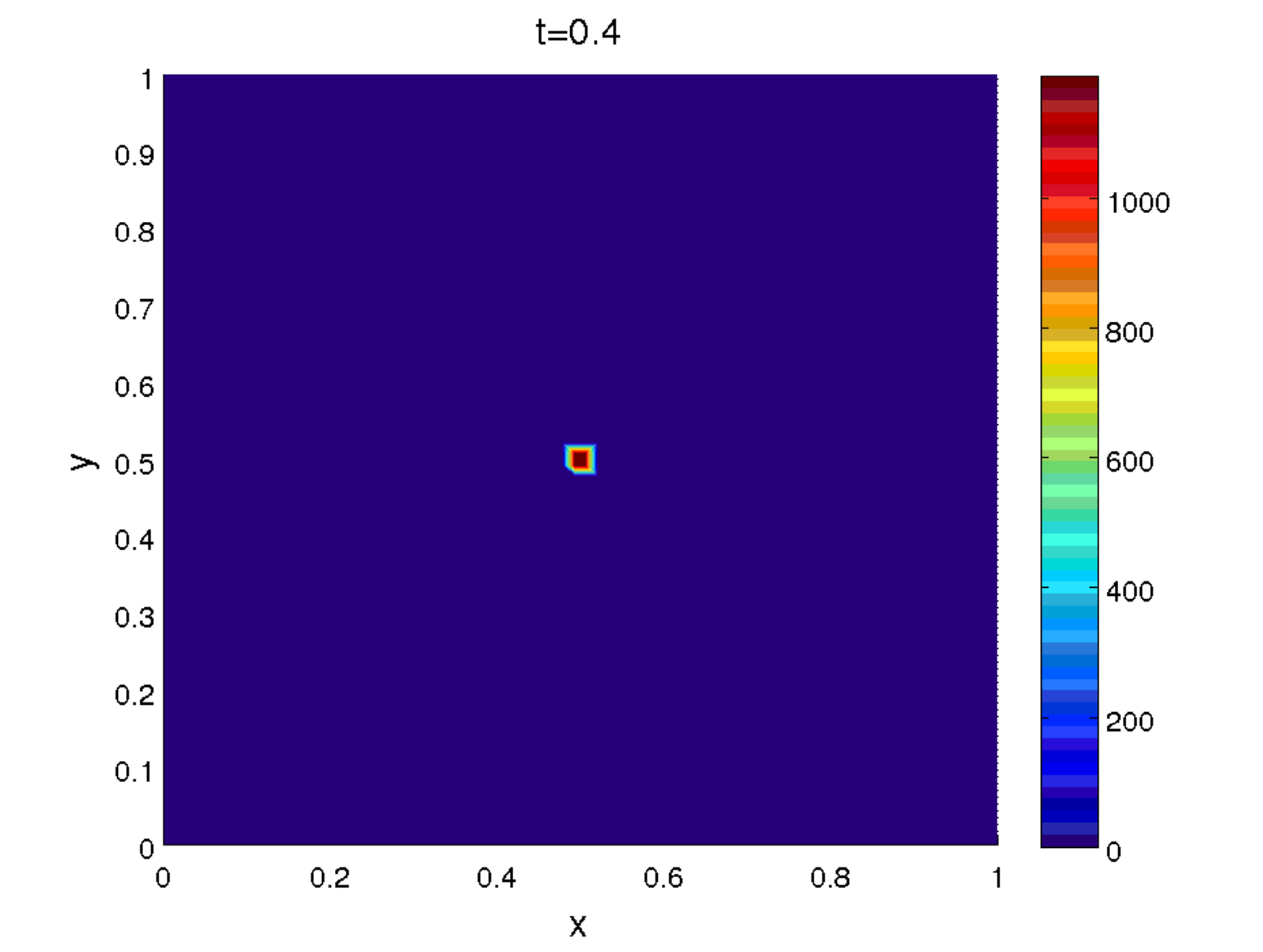}
\includegraphics[width=.51\textwidth]{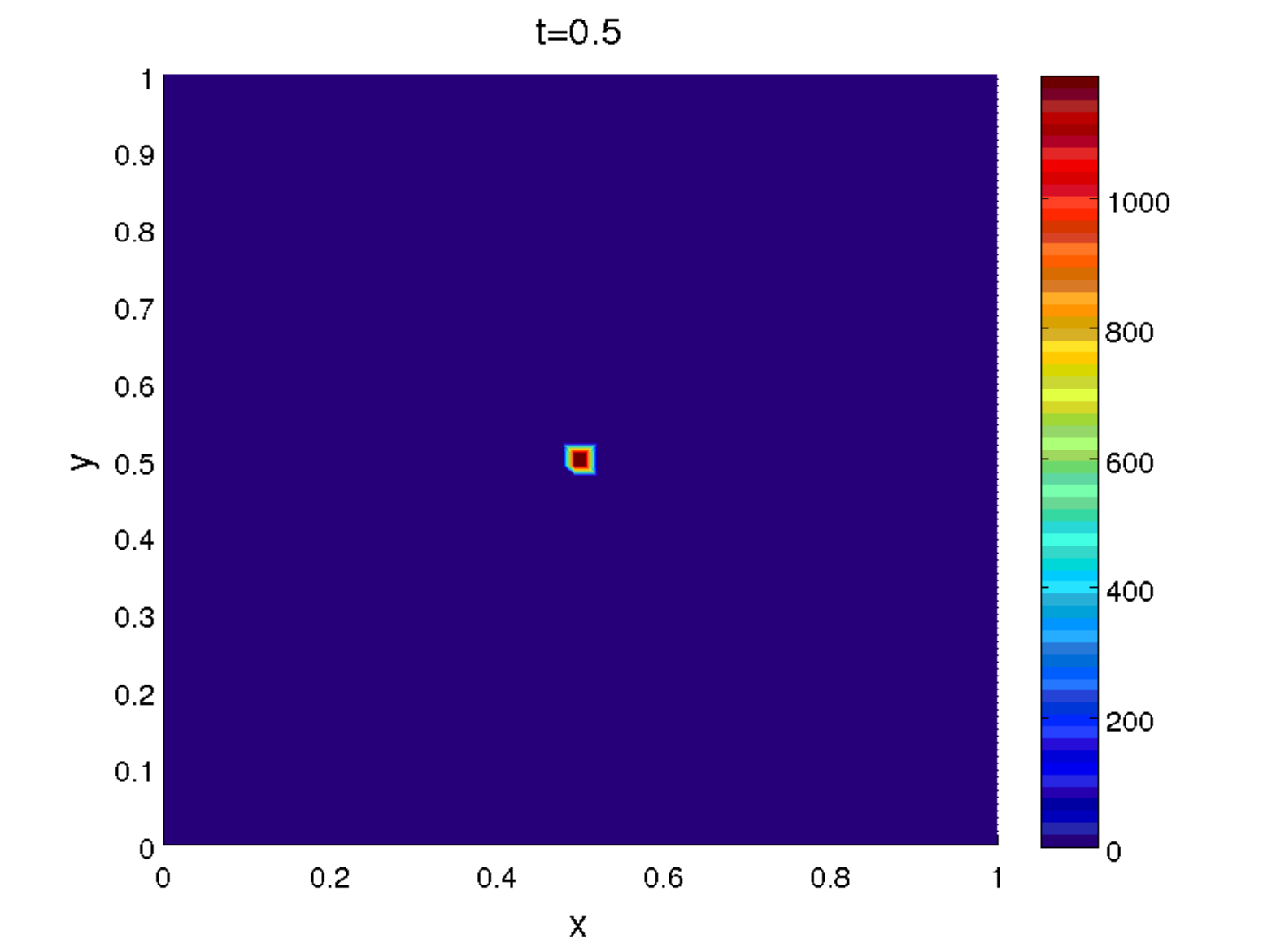}
\caption{Time dynamics of the numerical solution of the aggregation equation 
\eqref{EqInter} with {$W(x)=W_1(x) = 1-e^{-5|x|}$} and an initial datum given by a square.
Time increases from top left to bottom right.}
\label{sq2D}       
\end{figure}

\begin{figure}[ht!]
\includegraphics[width=.51\textwidth]{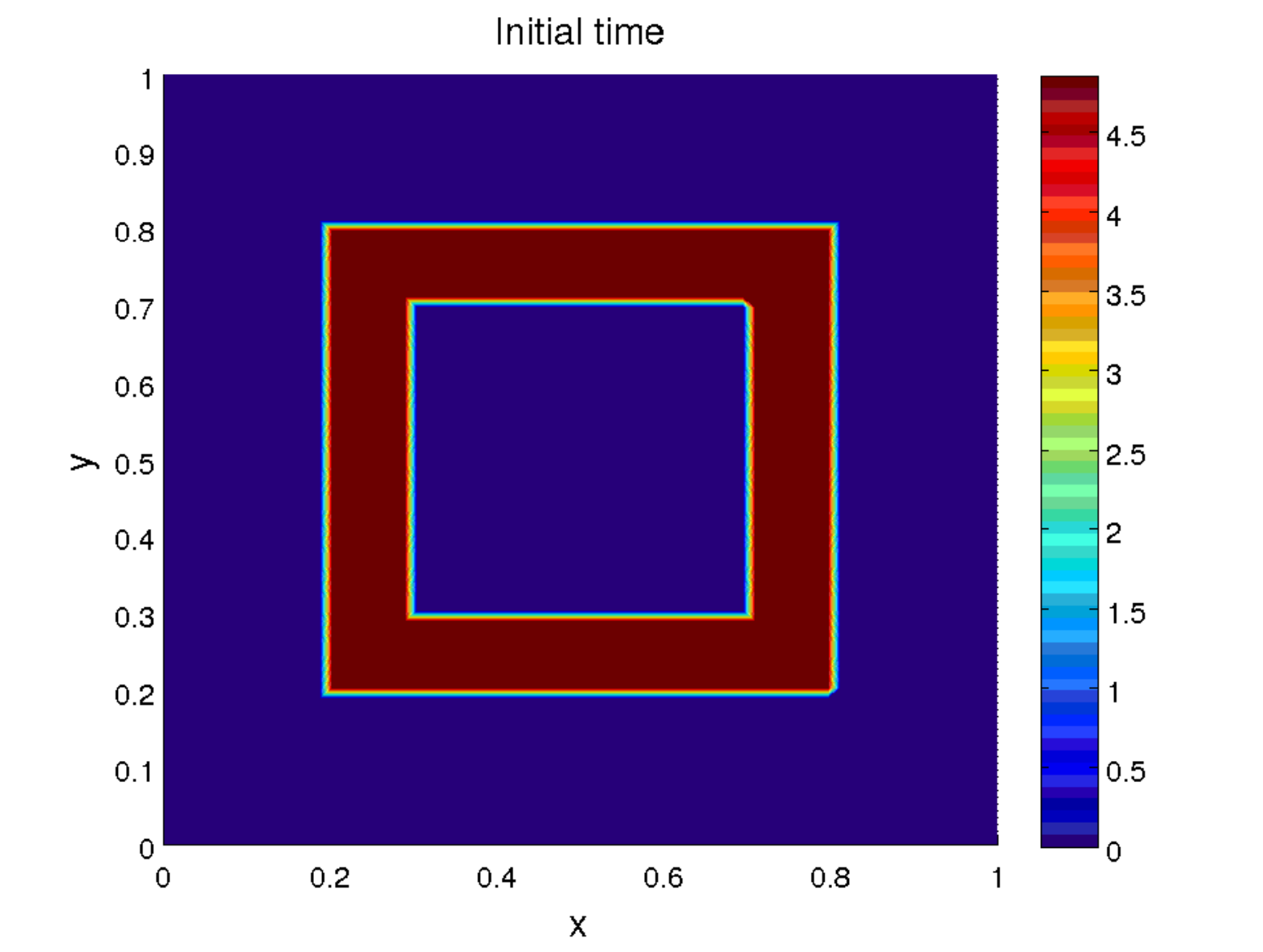}
\includegraphics[width=.51\textwidth]{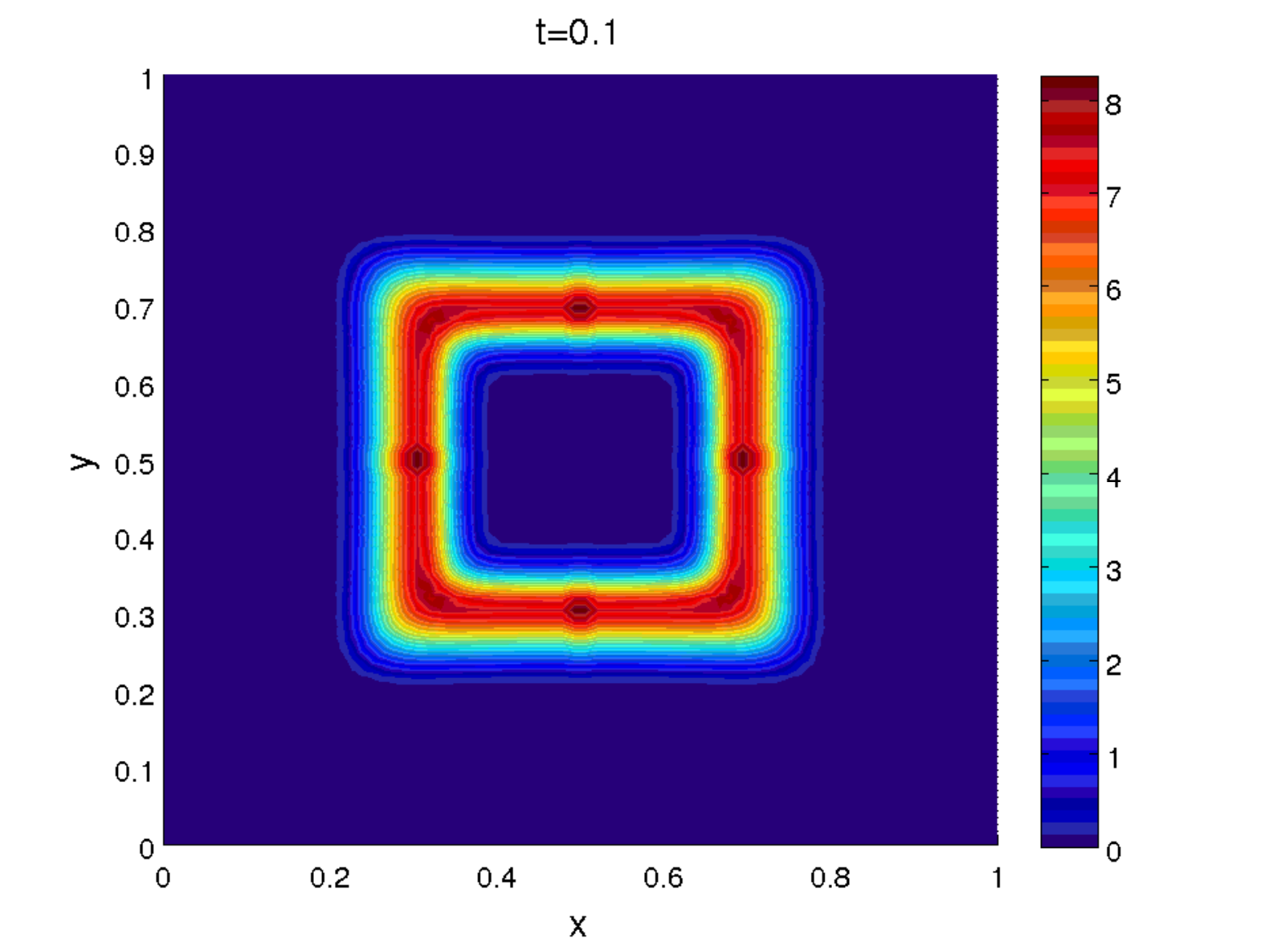}
\includegraphics[width=.51\textwidth]{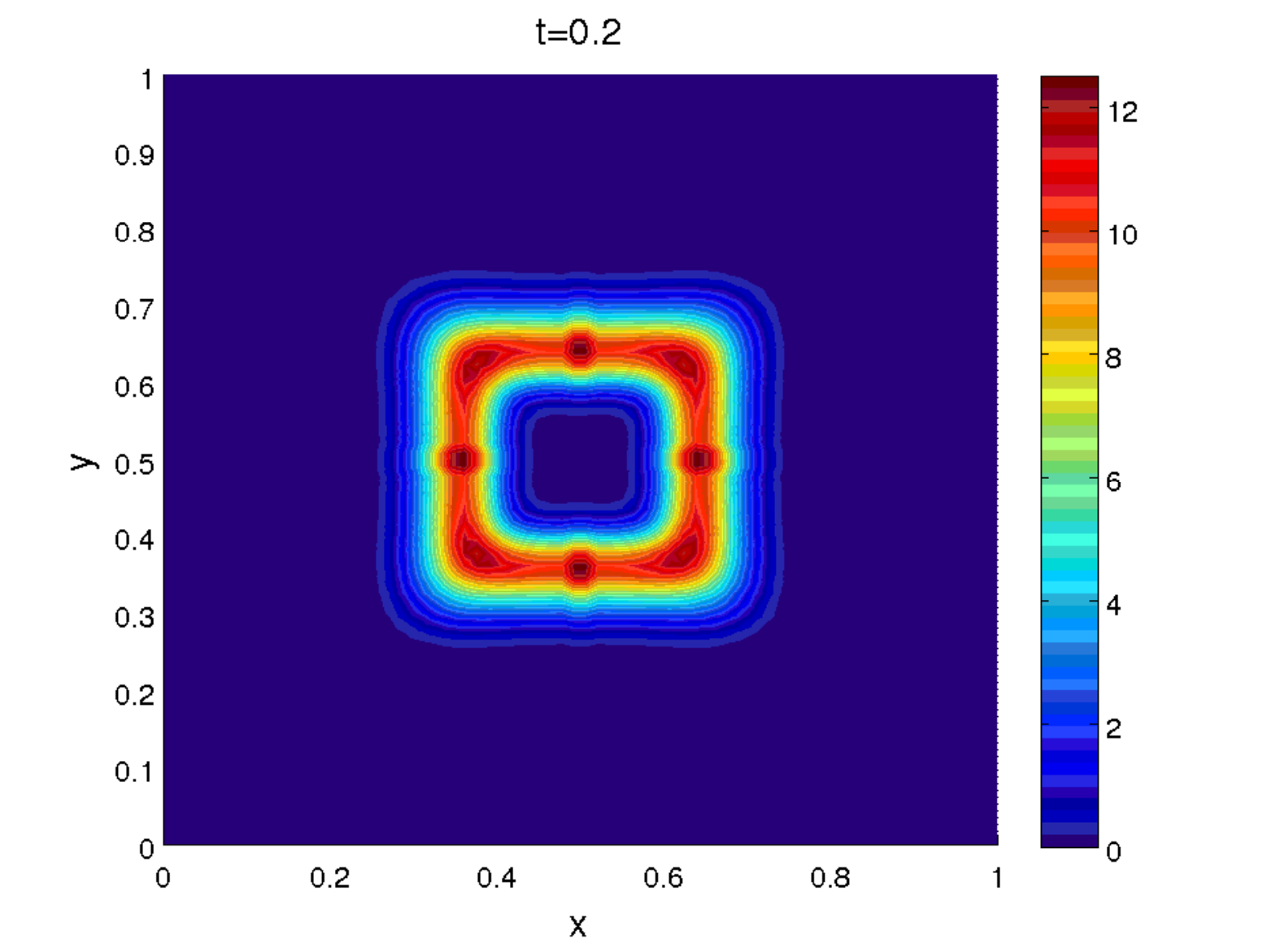}
\includegraphics[width=.51\textwidth]{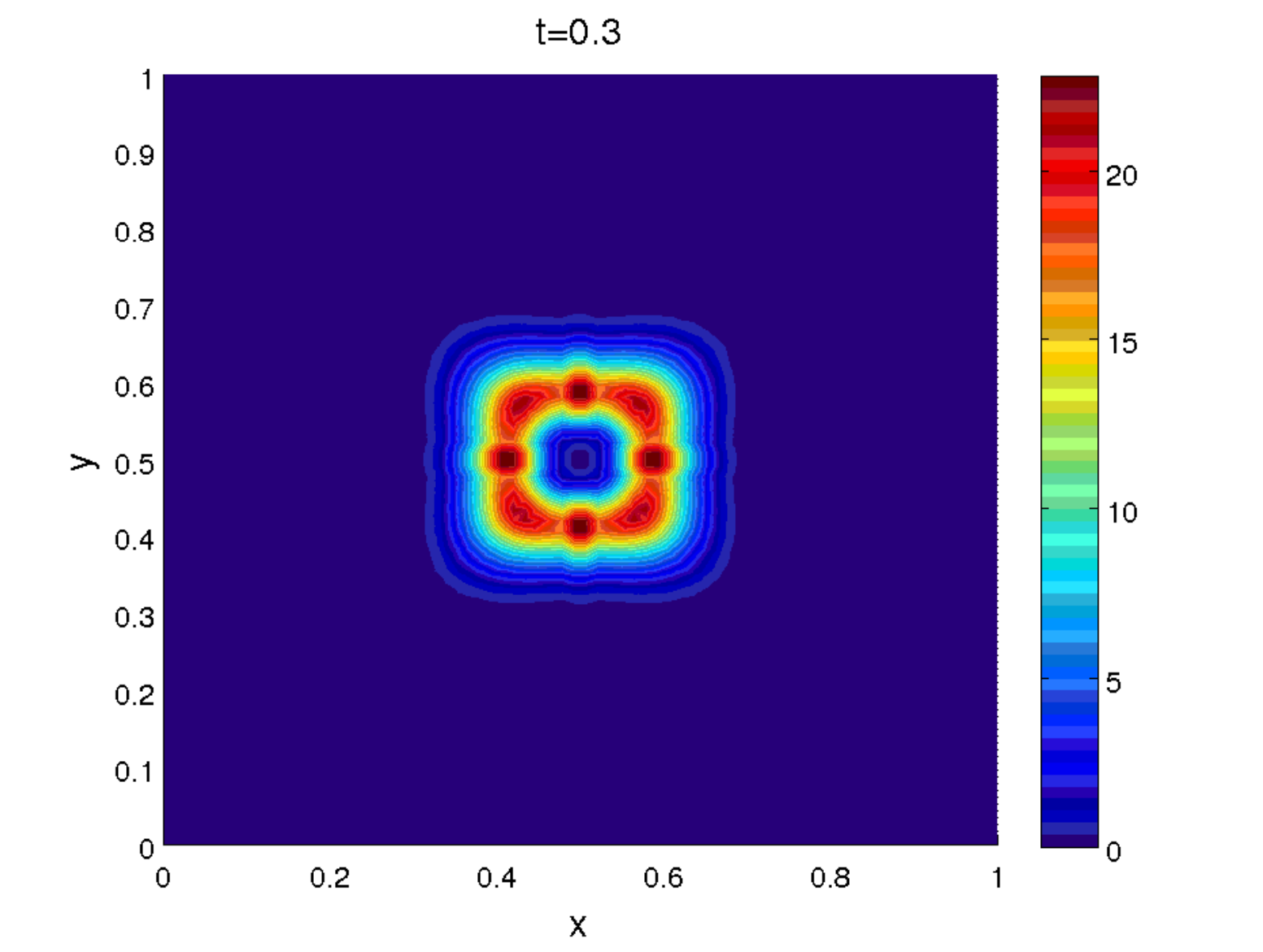}
\includegraphics[width=.51\textwidth]{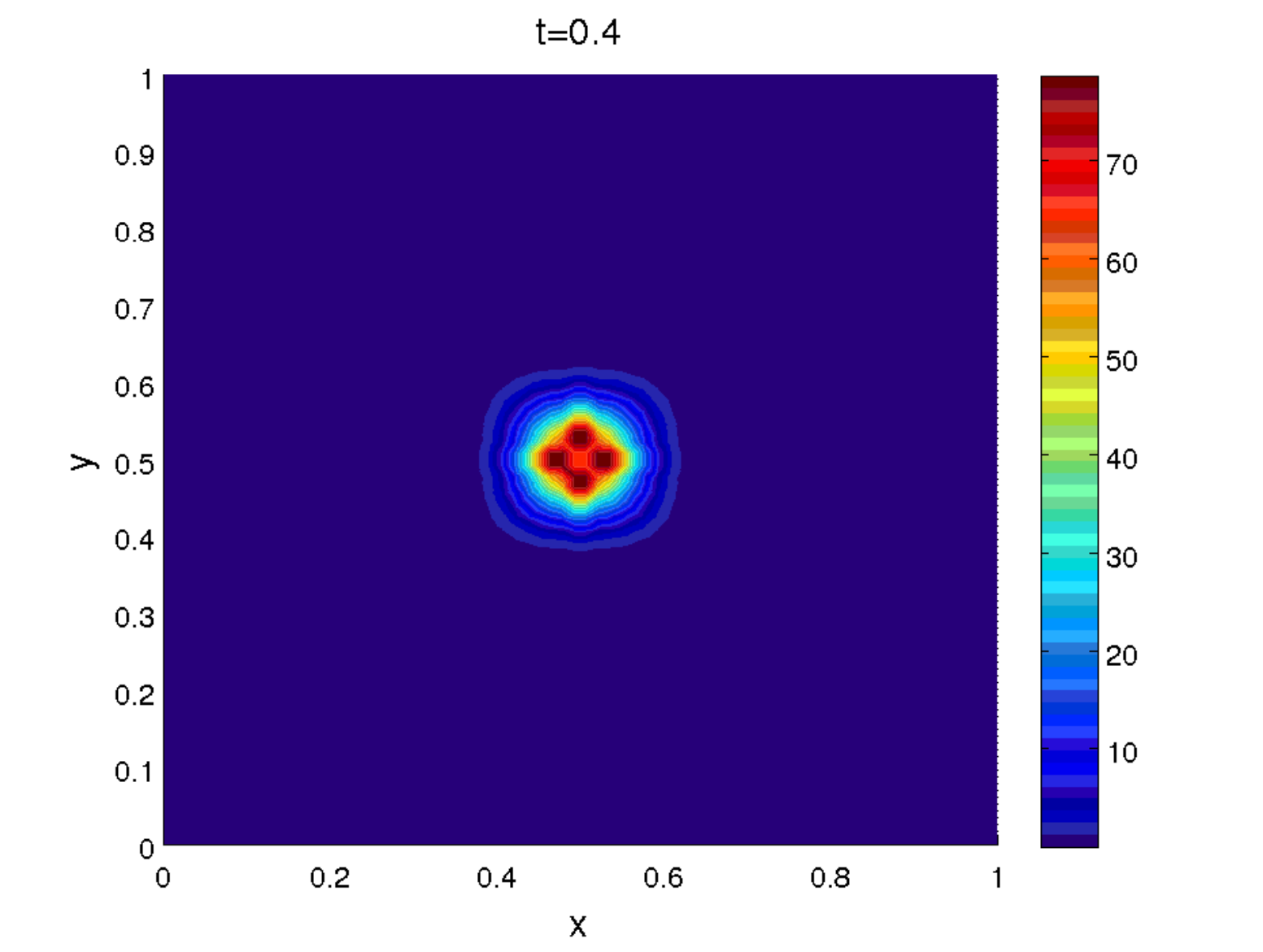}
\includegraphics[width=.51\textwidth]{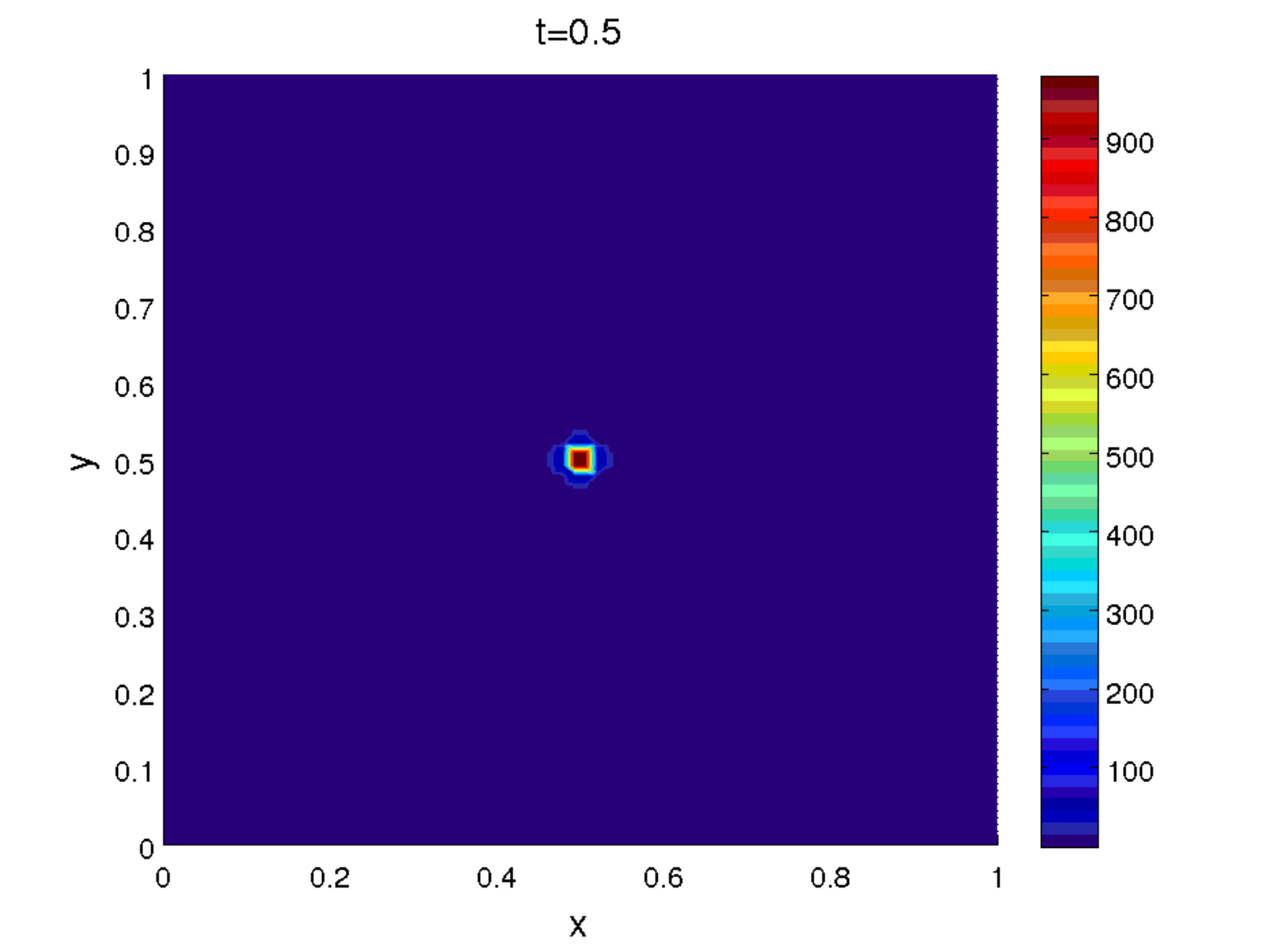}
\caption{Time dynamics of the numerical solution of the aggregation equation 
\eqref{EqInter} with {$W(x)=W_2(x) = 5|x|$} and an initial datum given by a square.
Time increases from top left to bottom right.}
\label{sq2Dbis}       
\end{figure}

As an illustration, we propose now a numerical example in two dimensions.
The spatial domain is the square $[0,1]\times[0,1]$; it
is discretized with $N_x=70$ nodes in the $x$-direction and $N_y=70$
nodes in the $y$-direction; we take a time step $\Delta t=10^{-3}$.
We consider two different initial data: the sum of three bumps
(as in \cite{CJLV})
\begin{equation*}
\begin{split}
&\rho^{ini}(t,x)
\\
&\hspace{5pt} = \frac{1}{M}\left(e^{-100((x_1-0.25)^2+(x_2-0.3)^2)}+e^{-100((x_1-0.77)^2+(x_2-0.7)^2)}
+ 0.9 e^{-100((x_1-0.37)^2+(x_2-0.62)^2)}\right),
\end{split}
\end{equation*}
where $M$ is a normalization constant such that $\|\rho^{ini}\|_{L^1}=1$; 
and an initial density with a square shape
$$
\rho^{ini}(t,x)=5\times\mathbf{1}_{[0.2,0.8]\times[0.2,0.8]\setminus [0.3,0.7]\times[0.3,0.7]}.
$$
With these numerical data, we compare the numerical results between the two potentials
{$W_1(x)=1-e^{-5|x|}$ and $W_2(x)=5|x|$.}
For $|x|$ close to $0$, we have that $\nabla W_1 \sim \nabla W_2$. 
Then the short range interaction is similar between both potentials,
but the long range interaction is different.
The numerical results are displayed in Figures \ref{bump2D} and \ref{sq2D} for the potential 
{$W_1(x)=1-e^{-5|x|}$} and in Figures \ref{bump2Dbis} and \ref{sq2Dbis} for the potential
{$W_2(x)=5|x|$}.

In each case, we observe, as expected, the aggregation in finite time of $\rho$ towards a Dirac delta.
Indeed it has been proved in \cite{Carrillo} that when the initial data is compactly supported,
solutions converge towards a Dirac delta in finite time.
We also observe that the time dynamics during this step of concentration is different between 
potentials $W_1$ and $W_2$.

The case with an initial datum with three bumps has been implemented in \cite{CJLV} 
with a Lax-Friedrichs scheme. We obtain here similar results but we observe 
a smaller numerical diffusion. Then we can make similar comments for the comparison
between the two potentials $W_1$ and $W_2$.
For the potential $W_1$, we observe that each bump coalesces into a Dirac delta,
then the three remaining Dirac deltas merge into a single Dirac delta (see Fig~\ref{bump2D}).
For the potential $W_2$, the solution seems to be more regular and Dirac deltas seems to 
appear for larger time (see Fig~\ref{bump2Dbis}).

For the initial data with a square shape,
the density $\rho$ keeps, 
for both potentials,
a shape similar to the initial square shape which tightens
as time increases. However with the potential $W_1$ (Fig~\ref{sq2D}), we notice
a strong concentration at the corners of the square, whereas in the case of the potential $W_2$
(Fig~\ref{sq2Dbis}) the density  is homogeneous along the edges of the square 
with a slight concentration in the middle of the edges.

\bigskip
{\bf Acknowledgements.}
The authors acknowledge partial support from the french ``ANR blanche" project Kibord~: ANR-13-BS01-0004, as well as from the ``BQR Acceuil EC 2017'' grant from Universit\'e Lyon 1. 

\bigskip

%
%


\begin{thebibliography}{99}
\bs

\bibitem{AubinCellina} J.-P. Aubin, A. Cellina, {\bf Differential inclusions. Set-valued maps and viability theory}. Grundlehren der Mathematischen Wissenschaften [Fundamental Principles of Mathematical Sciences], {\bf 264}. Springer-Verlag, Berlin, 1984.

\bibitem{Ambrosio} L. Ambrosio, N. Gigli, G. Savar\'e, {\bf Gradient flows in metric space of probability measures}, Lectures in Mathematics, Birk\"auser, 2005

\bibitem{benedetto} D. Benedetto, E. Caglioti, M. Pulvirenti, {\it A kinetic equation for granular media}, RAIRO Model. Math. Anal. Numer., {\bf 31} (1997), 615-641.


\bibitem{Andrea_c_toaa} A.L. Bertozzi, J.B. Garnett, T. Laurent, {\it Characterization of radially symmetric finite time blowup in multidimensional aggregation equations}, SIAM J. Math. Anal. {\bf 44}(2) (2012) 651--681.

\bibitem{Bertozzi2} A.L. Bertozzi, T. Laurent, J. Rosado, {\it $L^p$ theory for the multidimensional aggregation equation}, Comm. Pure Appl. Math., {\bf 64} (2011), no 1, 45--83.

\bibitem{Bianchini} S. Bianchini, M. Gloyer, {\it An estimate on the flow generated by monotone operators}, Comm. Partial Diff. Eq., {\bf 36} (2011), no 5, 777--796.

\bibitem{bobkov:ledoux}
S. Bobkov, M. Ledoux, {\it One-dimensional empirical measures, order statistics, and Kantorovich transport distances}, to appear in Memoirs of AMS.  

\bibitem{BV} M. Bodnar, J.J.L. Vel\'azquez, {\it An integro-differential equation arising as a limit of individual cell-based models}, J. Differential Equations {\bf 222} (2006), no 2, 341--380.

\bibitem{bonaschi} G.A. Bonaschi, J.A. Carrillo, M. Di Francesco, M.A. Peletier, {\it Equivalence of gradient flows and entropy solutions for singular nonlocal interaction equations in 1D}, ESAIM Control Optim. Calc. Var. {\bf 21} (2015), no 2, 414--441.

\bibitem{bouche} D. Bouche, J.-M. Ghidaglia, F. Pascal, {\it Error estimate and the geometric corrector for the upwind finite volume method applied to the linear advection equation}, SIAM J. Numer. Anal. {\bf 43} (2005), no 2, 578--603.

\bibitem{bj1} F. Bouchut, F. James, {\it One-dimensional transport equations with discontinuous coefficients}, Nonlinear Analysis TMA, {\bf 32}  (1998), no 7, 891--933.

\bibitem{Freda} M. Campos Pinto, J.A. Carrillo, F. Charles, Y.-P. Choi, {\it Convergence of a linearly transformed particle method for aggregation equations}, preprint, \texttt{https://hal.archives-ouvertes.fr/hal-01180687}.


\bibitem{CCH} J.A. Carrillo, A. Chertock, Y. Huang, {\it A Finite-Volume Method for Nonlinear Nonlocal Equations with a Gradient Flow Structure}, Comm. in Comp. Phys. {\bf 17} (2015), no 1, 233--258.

\bibitem{Carrillo} J.A. Carrillo, M. DiFrancesco, A. Figalli, T. Laurent, D. Slep\v{c}ev, {\it Global-in-time weak measure solutions and finite-time aggregation for nonlocal interaction equations}, Duke Math. J. {\bf 156} (2011), 229--271.

\bibitem{CJLV} J.A. Carrillo, F. James, F. Lagouti\`ere, N. Vauchelet,  {\it The Filippov characteristic flow for the aggregation equation with mildly singular potentials}, J. Differential Equations. {\bf 260} (2016), no 1, 304--338. 

\bibitem{CCV} J.A. Carrillo, R.J. McCann, C. Villani, {\it Contractions in the 2-Wasserstein length space and thermalization of granular media}, Arch. Rational Mech. Anal. {\bf 179} (2006), 217--263.

\bibitem{pieton} R.M. Colombo, M. Garavello, M. L\'ecureux-Mercier, {\it A class of nonlocal models for pedestrian traffic}, Math. Models Methods Appl. Sci., {\bf 22} (2012), no 4:1150023, 34.

\bibitem{CB} K. Craig, A.L. Bertozzi, {\it A blob method for the aggregation equation}, Math of Comp {\bf 85} (2016), no 300, 1681--1717.

\bibitem{pieton2} G. Crippa, M. L\'ecureux-Mercier, {\it Existence and uniqueness of measure solutions for a system of continuity equations with non-local flow}, NoDEA Nonlinear Differential Equations Appl., (2013) {\bf 20} (2013), no 3, 523--537.

\bibitem{caniveau} F. Delarue, F. Lagouti\`ere, {\it Probabilistic analysis of the upwind scheme for transport equations}, Arch. Rational Mech. Anal. {\bf 199} (2011), 229--268.

\bibitem{DLV} F. Delarue, F. Lagouti\`ere, N. Vauchelet, {\it Analysis of finite volume upwind scheme for transport equation with discontinuous coefficients}. Accepted for publication in J. Math. Pures Appliqu\'ees. 

\bibitem{Despres} B. Despr\'es, {\it An explicit a priori estimate for a finite volume approximation of linear advection on non-Cartesian grid}, SIAM J. Numer. Anal. {\bf 42} (2004), no 2, 484--504.

\bibitem{dob} R. Dobrushin, {\it Vlasov equations}, Funct. Anal. Appl. {\bf 13} (1979), 115–123.

\bibitem{dolschmeis} Y. Dolak, C. Schmeiser, {\it Kinetic models for
chemotaxis: Hydrodynamic limits and spatio-temporal mechanisms},
J. Math. Biol., {\bf 51} (2005), 595--615.

\bibitem{filblaurpert} F. Filbet, P. Lauren\c{c}ot, B. Perthame, {\it Derivation of hyperbolic models for chemosensitive movement}, J. Math. Biol., {\bf 50} (2005), 189--207.

\bibitem{Filippov} A.F. Filippov, {\it Differential Equations with Discontinuous Right-Hand Side}, A.M.S. Transl. (2) {\bf 42} (1964), 199--231.

\bibitem{Golse} F. Golse, {\it On the Dynamics of Large Particle Systems in the Mean Field Limit}. In: A. Muntean, J. Rademacher, A. Zagaris (eds), Macroscopic and Large Scale Phenomena: Coarse Graining, Mean Field Limits and Ergodicity. Lecture Notes in Applied Mathematics and Mechanics, vol 3. Springer, Cham, 2016. 

\bibitem{GJ} L. Gosse, F. James, {\it Numerical approximations of one-dimensional linear conservation equations with discontinuous coefficients}, Math. Comput. {\bf 69} (2000) 987--1015.

\bibitem{GT} L. Gosse, G. Toscani, {\it Identification of Asymptotic Decay to Self-Similarity for One-Dimensional Filtration Equations}, SIAM J. Numer. Anal. {\bf 43} (2006) 2590--2606.

\bibitem{sisc} L. Gosse, N. Vauchelet, {\it Numerical high-field limits in two-stream kinetic models and 1D aggregation equations}, SIAM J. Sci. Comput. {\bf 38} (2016), no 1, A412--A434.

\bibitem{Legloc} T.Y. Hou, P.G. LeFloch, {\it Why nonconservative schemes converge to wrong solutions: error analysis}, Math. Comp. {\bf 62} (1994), no 206, 497--530. 

\bibitem{Huang1} Y. Huang, A.L. Bertozzi, {\it Asymptotics of blowup solutions for the aggregation equation}, Discrete and Continuous Dynamical Systems - Series B, {\bf 17} (2012), 1309--1331.

\bibitem{Huang2} Y. Huang, A.L. Bertozzi, {\it Self-similar blowup solutions to an aggregation equation in $\RR^n$}, SIAM Journal on Applied Mathematics, {\bf 70} (2010), 2582--2603.

\bibitem{NoDEA} F. James, N. Vauchelet, {\it Chemotaxis: from kinetic equations to aggregation dynamics}, Nonlinear Diff. Eq. and Appl. (NoDEA), {\bf 20} (2013), no 1, 101--127.

\bibitem{GF_dual} F. James, N. Vauchelet, {\it Equivalence between duality and gradient flow solutions for one-dimensional aggregation equations}, Disc. Cont. Dyn. Syst., {\bf 36} (2016), no 3, 1355--1382.

\bibitem{sinum} F. James, N. Vauchelet, {\it Numerical method for one-dimensional
    aggregation equations}, SIAM J. Numer. Anal. {\bf 53} (2015), no 2, 895--916.

\bibitem{keller} E.F. Keller, L.A. Segel, {\it Initiation of slime mold aggregation viewed as an instability}, J. Theor. Biol., {\bf 26} (1970), 399--415.

\bibitem{Kuznetsov} N.N. Kuznetsov, {\it The accuracy of certain approximate methods for the computation of weak solutions of a first order quasilinear equation}, \v{Z}. Vy\v{c}isl. Mat. i Mat. Fiz. {\bf 16} (1976), no 6, 1489--1502.

\bibitem{lava} F. Lagouti\`ere, N. Vauchelet, {\it Analysis and simulation of nonlinear and nonlocal transport equation}, to appear in Innovative algorithms and analysis, Springer INdAM Series 16, L. Gosse and R. Natalini Ed, 2016. 

\bibitem{Li} H. Li, G. Toscani, {\it Long time asymptotics of kinetic models of granular flows}, Arch. Rat. Mech. Anal., {\bf 172} (2004), 407--428.

\bibitem{M} B. Merlet, {\it $L^\infty$- and $L^2$-error estimates for a finite volume approximation of linear advection}, SIAM J. Numer. Anal. {\bf 46} (2007), no 1, 124--150.

\bibitem{MV} B. Merlet, J. Vovelle, {\it Error estimate for finite volume scheme}, Numer. Math. {\bf 106} (2007), 129--155.

\bibitem{morale} D. Morale, V. Capasso, K. Oelschl\"ager, {\it An interacting particle system modelling aggregation behavior: from individuals to populations}, J. Math. Biol., {\bf 50} (2005), 49--66.

\bibitem{okubo} A. Okubo, S. Levin, {\bf Diffusion and Ecological Problems: Modern Perspectives}, Springer, Berlin, 2002.

\bibitem{patlack} C.S. Patlak, {\it Random walk with persistence and external bias}, Bull. Math. Biophys., {\bf 15} (1953), 311-338.


\bibitem{PoupaudRascle} F. Poupaud, M. Rascle, {\it Measure solutions to the linear multidimensional transport equation with discontinuous coefficients}, Comm. Partial Diff. Equ., {\bf 22} (1997), 337--358.

\bibitem{rachev} S.T. Rachev and L. R\"uschendorf, {\bf Mass Transportation Problems. Vol. I. Theory}, Probab. Appl. (N. Y.), Springer-Verlag, New York, 1998.

\bibitem{Filippo c touo} F. Santambrogio, {\bf Optimal transport for applied mathematicians. Calculus of variations, PDEs, and modeling}. Progress in Nonlinear Differential Equations and their Applications, 87. Birkh\"auser/Springer, Cham, 2015. 

\bibitem{schlichting} A. Schlichting, C. Seis, {\it Convergence rates for upwind schemes with rough coefficients}, SIAM J. Numer. Anal. {\bf 55} (2017), no 2, 812--840.

\bibitem{topaz} C.M. Topaz, A.L. Bertozzi, {\it Swarming patterns in a two-dimensional kinematic model for biological groups}, SIAM J. Appl. Math. {\bf 65} (2004), 152--174.

\bibitem{Toscani} G. Toscani, {\it Kinetic and hydrodynamic models of nearly elastic granular flows}, Monatsh. Math. {\bf 142} (2004), 179--192.

\bibitem{Villani1} C. Villani, {\bf Optimal transport, old and new}, Grundlehren der Mathematischen Wissenschaften 338, Springer, 2009.

\bibitem{Villani2} C. Villani, {\bf Topics in optimal transportation}, Graduate Studies in Mathematics {\bf 58}, Amer. Math. Soc, Providence, 2003.




\end{thebibliography}


\end{document}